\documentclass[]{amsart}
\usepackage[letterpaper,hmargin=27mm,vmargin=27mm,bmargin=27mm,tmargin=22mm]{geometry}
\usepackage[utf8]{inputenc}
\usepackage{amsmath,amsthm,amssymb}
\usepackage{mathtools}
\usepackage[T1]{fontenc}
\usepackage[english]{babel}
\usepackage{mathrsfs}
\usepackage{yfonts}
\usepackage{pgf,tikz}
\usetikzlibrary{arrows}
\usepackage{listings}
\usepackage{graphicx}
\usepackage{float}
\usepackage{caption}
\usepackage{subcaption}
\usepackage[normalem]{ulem}
\usepackage{changepage}
\usepackage{listings}
\usepackage[final]{pdfpages}
\usepackage{hyperref}
\usepackage{lipsum}
\usepackage{url}
\usepackage{algorithm}
\usepackage{algpseudocode}
\usepackage{pgfplots}
\usepackage{mathrsfs}
\usepackage[toc]{glossaries}
\usepackage{afterpage}
\usetikzlibrary{arrows}
\usepackage{authblk}
\usepackage{bbm}
\usepackage{enumitem}

\usepackage[
backend=biber, %
style=numeric, %
 url=false, %
 doi=false, %
 eprint=true, %
 giveninits=true, %
 isbn=false,
 maxbibnames=5, %
 sortcites=true, %
]{biblatex}

\AtEveryBibitem{\clearlist{language}}
\DeclareFieldFormat{title}{#1}

\newtheorem{theorem}{Theorem}[section]
\newtheorem{prop}[theorem]{Proposition}
\newtheorem{assump}[theorem]{Assumption}

\newtheorem{lemma}[theorem]{Lemma}
\newtheorem{definition}[theorem]{Definition}
\newtheorem{remark}[theorem]{Remark}

\theoremstyle{definition}
\newtheorem{example}[theorem]{Example}
\newcommand{\norm}[1]{\left\lVert#1\right\rVert}
\newcommand{\abs}[1]{\left\lvert#1\right\rvert}

\DeclareMathOperator{\im}{Im}
\DeclareMathOperator{\re}{Re}

\DeclareMathOperator{\supp}{supp}
\DeclareMathOperator{\auid}{AUID}
\DeclareMathOperator{\Tr}{Tr}
\DeclareMathOperator{\GUE}{GUE}
\DeclareMathOperator{\GOE}{GOE}
\DeclareMathOperator{\TW}{TW}
\DeclareMathOperator{\Dim}{\Delta \im}
\newcommand{\di}{\mathop{}\!{d}}
\newcommand{\dm}{\frac{\di}{\di t}}
\newcommand{\dF}{\frac{\di}{\di t}}
\DeclareMathOperator{\ii}{i}

\newcommand{\code}[1]{\texttt{#1}}
\newcommand{\bd}{\boldsymbol{\delta}}
\newcommand{\bc}{\boldsymbol{c}}
\newcommand{\e}{\mathrm{e}}
\newcommand{\GOEast}{\widehat{\text{GOE}}}
\DeclareMathOperator{\EX}{\mathbb{E}}%
\DeclareMathOperator{\edge}{S_{\text{edge}}}%
\DeclareMathOperator{\indicator}{\mathbbm{1}} %
\DeclareMathOperator{\Prob}{\mathbb{P}}%

\newcommand\Tau{\mathcal{T}}

\newcommand{\nc}{\normalcolor}
\newcommand{\C}{{\mathbb C}}
\newcommand{\R}{{\mathbb R}}
\newcommand{\ie}{\emph{i.e., }}
\newcommand{\eg}{\emph{e.g., }}
\newcommand{\cf}{\emph{c.f., }}
\newcommand{\dd}{\mathrm{d}}

\newcommand{\ppartial}{\widehat{\partial}}

\numberwithin{equation}{section}

\makeglossaries

\begin{document}

\begin{center}
	
	\begin{minipage}{0.85\textwidth}
		\vspace{2.5cm}
		
		\begin{center}
			\large\bf
			Quantitative Tracy--Widom laws for sparse random matrices
		\end{center}
	\end{minipage}
\end{center}

\vspace{0.5cm}

\begin{center} \begin{minipage}{1\textwidth}
		\begin{minipage}{0.33\textwidth}
			\begin{center}
				Teodor Bucht\\
				\footnotesize 
				{KTH Royal Institute of Technology}\\
				{\it tbucht@kth.se}
			\end{center}
		\end{minipage}
		\begin{minipage}{0.33\textwidth}
			\begin{center}
				Kevin Schnelli\\
				\footnotesize 
				{KTH Royal Institute of Technology}\\
				{\it schnelli@kth.se}
			\end{center}
		\end{minipage}
		\begin{minipage}{0.33\textwidth}
			\begin{center}
				Yuanyuan Xu\\
				\footnotesize 
				{AMSS, CAS}\\
				{\it yyxu2023@amss.ac.cn}
			\end{center}
		\end{minipage}
	\end{minipage}
\end{center}

\vspace{0.5cm}

\begin{center}
	\begin{minipage}{0.83\textwidth}\footnotesize{
			{\bf Abstract.}
			We consider the fluctuations of the largest eigenvalue of sparse random matrices, the class of random matrices that includes the normalized adjacency matrices of the Erd\H{o}s--Rényi graph $\mathcal{G}(N, p)$. We show that the fluctuations of the largest eigenvalue converge to the Tracy--Widom law at a rate almost $O(N^{-1/3 } + p^{-2} N^{-4/3})$ in the regime $p \gg N^{-2/3 }$. Our proof builds upon the Green function comparison method initiated by Erd\H{o}s, Yau, and Yin \cite{green_function_comparison_reference}. To show a Green function comparison theorem for fine spectral scales, we implement algorithms for symbolic computations involving averaged products of Green function entries.
		}
	\end{minipage}
\end{center}

\vspace{5mm}

{\small
	\footnotesize{\noindent\textit{Date}: \today}

	\vspace{2mm}
	
}

\pagenumbering{arabic}

\section{Introduction}\label{sec:introduction}

A prominent example of a sparse random matrix is the adjacency matrix of the Erd\H{o}s--Rényi graph $\mathcal{G}(N, p)$ with $p\equiv p(N)$, where each (undirected) edge of the complete graph with $N$ vertices appears independently with probability $p$. It has applications in combinatorics, graph theory, mathematical physics, and network theory. Properties of a graph can be inferred from the spectral behavior of its adjacency matrix, see references \cite{random_graphs_bollobas, random_graphs_sla}.

We denote the adjacency matrix of the random graph $\mathcal{G}(N, p)$ by $A_N$. Each column of $A_N$ typically has $pN$ non-zero entries.  One can use the parameter $q:=\sqrt{pN}$ to measure the sparsity. We say the matrix $A_N$ is sparse when $q\ll \sqrt{N}$ and the matrix is full when $q\asymp\sqrt{N}$. Note that $A_N=(a_{ij})_{i,j=1}^N$ is a real symmetric matrix and the entries $(a_{ij})_{i\le j}$ are i.i.d.\ Bernoulli random variables with
\begin{equation*}
	\Prob \left(a_{ij} = 1\right) = q^2/N, \quad \qquad \Prob \left(a_{ij} = 0\right) = 1 - q^2/N. 
\end{equation*} 
One may center and normalize the matrix $A_N$ by setting
\begin{equation*}
	H_N := \frac{A_N - (q^2 / N) E_N}{q \sqrt{1 - q^2/N}}, \quad \qquad E_N:=(1)_{i,j=1}^{N},
\end{equation*}
so that the entries of the rescaled matrix $H_N=(h_{ij})_{i,j=1}^{N}$ satisfy
\begin{equation}\label{moment_intro}
	\EX[h_{ij}] = 0, \quad \qquad  \EX[h_{ij}^2] = \frac{1}{N}, \qquad \quad \EX \left[\abs{h_{ij}}^k\right]  
	\asymp \frac{1}{N q^{k-2}}, \qquad k\geq 3.
\end{equation}
Having the Erd\H{o}s--Rényi model in mind, we consider real symmetric sparse matrices $H_N=(h_{ij})_{i,j=1}^N$, whose entries are i.i.d.\ random variables (up to symmetry) satisfying the moment conditions in (\ref{moment_intro}).
On the global scale the empirical eigenvalue distribution of $H_N$ follows Wigner's semicircle law if $q^2\gg 1$.

This general model was studied in \cite{erdos_renyi_graph_I, erdos_renyi_graph_II} and contains the (rescaled) adjacency matrices of Erd\H{o}s--Rényi graphs among other sparse matrix models, \eg diluted Wigner matrices \cite{Khorunzhy2001Sparse}. In the full case $q\asymp\sqrt{N}$, $H_N$ is a standard real symmetric Wigner matrix and the fluctuations about the spectral edges are given by the celebrated Tracy--Widom laws~\cite{tracy_widom1, tracy_widom2}, like for the invariant Gaussian ensembles. That is, we have the edge universality for the largest eigenvalue~$\lambda_N$ of $H_N$,
\begin{equation}\label{edge_wigner}
	\lim_{N\to \infty} \Prob \left(N^{2/3} (\lambda_N - 2) \leq r\right) = \TW_1(r),
\end{equation}
where $\TW_1$ is the cumulative distribution function of the Tracy--Widom law in the real symmetric case.
Edge universality was first established in~\cite{shoshnikov1999edge} for Wigner matrices with symmetrically distributed entries using the moment methods, and the assumption of symmetric distributions was partially removed in~\cite{peche2007Wigner,peche2008Spectral}. It was further established in~\cite{tao2010local} under a third moment matching condition. Edge universality without matching conditions was proved in~\cite{green_function_comparison_reference,lee2014edge}. Besides Wigner matrices, the Tracy--Widom laws are universal for the extremal eigenvalues of a wide class of random matrix models, \eg Wigner type matrices with variance profiles and correlations \cite{bourgade2014edge, lee2015edge, knowles2017anisotropic, alt2020band}, random band matrices~\cite{SodinBand2010,liu2023edge}, $\beta$-ensembles with general potentials \cite{bourgade2014edge, deift2007edge}, and general $\beta$-Dyson Brownian motions \cite{landon2017edge, adhikari2020dyson, amol2024edge}.

For sparse random matrices satisfying (\ref{moment_intro}) with $q\geq N^{1/3+\bd}$, for some $\bd>0$,  edge universality as in (\ref{edge_wigner}) was proved in \cite{erdos_renyi_graph_II} and for sparser matrices with $q\geq N^{1/6+\bd}$ in~\cite{LeeJiOon2018LlaT} with a deterministic edge correction term. In~\cite{huang2020transition} it was established, for $q\geq N^{1/9+\bd}$, that the fluctuations of the shifted largest eigenvalue satisfy
\begin{equation}\label{edge_sparse}
	\lim_{N\to \infty} \Prob \left(N^{2/3} (\lambda_N -L) \leq r\right) = \TW_1(r), \qquad \quad L=2+\frac{6 \kappa_4}{q^2}+\chi+O(q^{-4}),
\end{equation} 
where $\kappa_4$ is determined by the fourth order moments of $h_{ij}$ from (\ref{moment_intro}), and $\chi$ is a random correction term given by
\begin{align}\label{def_chi}
	\chi=\frac{1}{N} \sum_{i,j} \big(h_{ij}^2 - \frac{1}{N}\big) .
\end{align}
It describes fluctuations of the total number of edges and is asymptotically Gaussian of size $\frac{1}{\sqrt{N}q}$. Hence at $q\asymp N^{1/6}$ the fluctuations transition from Tracy--Widom to Gaussian ones~\cite{huang2020transition} down to  $q \ge N^{\bd}$~\cite{HeKnowles2021}. Yet by including refined subleading random corrections terms~\cite{lee2021higher}, Tracy--Widom fluctuations and the edge universality as in~\eqref{edge_sparse} are recovered~\cite{huang2022edge} for $q \geq N^{\bd}$.

For even sparser matrices, \eg  the critical Erd\H{o}s--Rényi graph with $Np \asymp \log N$, the extremal eigenvalue statistics were analyzed in \cite{largest_eigenvale_sparse_inhom,spectralRadiiSparse,  poissonStatisticsSparse,vu2008spectral, latal2018structure,altDucatezKnowlesLocalizedPhase}. In~\cite{extremalEigenvalueCritical,outliersSpectrumSparse} a sharp transition was established: there is $b^*\simeq 2.59$ such that when $q^2 > b^* \log N$,
the spectral norm converges to 2, while for $q^2 < b^* \log N$ the extreme eigenvalues are determined by the
largest degree. Another important model of sparse random graphs, the random $d$-regular graph ensemble, has a  strikingly different extremal eigenvalue behavior. Remarkably, edge universality for any fixed degree $d\geq 3$ was recently proved in \cite{huang2024ramanujan}, see earlier works in \cite{edgeRigidityRegularGraph, spectrumRegularGraphs, spectralGapEdgeUniversalityRandomRegularGraph, huang2023edge, spectralEdgeConstantDegree, huang2024optimal}.

Returning to Wigner matrices, it is natural to ask how fast the fluctuations of the largest eigenvalue in (\ref{edge_wigner}) converge to their Tracy--Widom limits. The exact rate of convergence $O(N^{-2/3})$ was established in \cite{goe_convergence} for the invariant Gaussian Unitary Ensemble ($\GUE$) and the Gaussian Orthogonal Ensemble ($\GOE$) with a proper rescaling for the largest eigenvalue. The first quantitative convergence rate $O(N^{-2/9+\omega})$ for Wigner matrices with a variance profile was obtained in~\cite{bourgade_quant} using optimal local relaxation estimates for the Dyson Brownian motion and a quantitative Green function comparison theorem  for short times.  Later the convergence rate was improved to $O(N^{-1/3+\omega})$ by \cite{Schnelli_Xu_2022, generalized_wigner} using a refined long-time Green function comparison theorem \cite{green_function_comparison_reference}.

In this paper, we focus on random sparse matrices $H_N$ satisfying (\ref{moment_intro}) with $ q \geq N^{1/6+\bd}$, for some $\bd>0$, and establish the first quantitative Tracy--Widom law using a Kolmogorov type distance (see Theorem \ref{main_theorem} for a formal statement): For any $r_0 \in \mathbb{R}$ and small $\omega > 0$,
\begin{equation}\label{introduction_main_result}
	\sup_{r > r_0} \abs{\Prob \left(N^{2/3}(\lambda_N - 2 - \frac{6 \kappa_4}{q^2} - \chi) \leq r\right) - \TW_1(r)} \leq N^{\omega} \Big(\frac{1}{N^{1/3}} + \frac{N^{2/3}}{q^4}\Big),
\end{equation}
for sufficiently large $N$, with $\kappa_4$ from (\ref{edge_sparse}) and $\chi$ given by (\ref{def_chi}). In the regime $q\asymp\sqrt{N}$, this recovers the convergence rate for Wigner matrices obtained in \cite{Schnelli_Xu_2022}. In the sparse case  with $q\ll \sqrt{N}$, the entries of $H_N$ from (\ref{moment_intro}) are of size $O(1/q)$, which is much bigger than in the Wigner case. This is the main technical difficulty to overcome for sparse random matrices with small $q$.

A quantitative convergence result is not only of theoretical interest. In general, bounds on the convergence rate toward limiting distributions are important tools when approximating finite $N$ distributions \cite[Chapter~11]{DasGupta}. A common application of a bound on the convergence rate is to bound the error of the p-value used in hypothesis testing. An example involving hypothesis testing with the largest eigenvalue of sparse random matrices is community detection in stochastic block models \cite{hypothesisTestingComminutyDetection, Lei2016Block}.

	Our proof of \eqref{introduction_main_result} uses the same strategy as in \cite{Schnelli_Xu_2022} for Wigner matrices, based on the Green function comparison method initiated by Erd\H{o}s, Yau and Yin \cite{green_function_comparison_reference}. The main technical result (as stated in Theorem \ref{GFCT_random_shift}) is a Green function comparison theorem (GFT) comparing the expectation of a suitably chosen functional of a sparse matrix $H_N$ with the corresponding expectation for a Wigner matrix $W_N$ with Gaussian entries. Such a functional of $H_N$ is properly chosen to approximate the cumulative distribution function of the rescaled largest eigenvalue through the {\it Green function} or {\it resolvent} of $H_N$:
	\begin{align}\label{def_G}
		G(z):=(H_N-zI_N)^{-1} \in \C^{N \times N}, \qquad z=E+\ii \eta\in \C^+,
	\end{align}
	see Lemma \ref{second_eigenval_estimate_lemma} for more details. This lemma and many estimates crucially rely on the local law for the Green function~\cite{dynamical_approach_rmt,green_function_comparison_reference}, and its form for the sparse setup~\cite{erdos_renyi_graph_I,LeeJiOon2018LlaT, huang2020transition, huang2022edge,lee2021higher}, see Theorems~\ref{entrywise_local_law} and~\ref{optimal_local_law_sparse} below for formal statements.

	To prove the GFT in Theorem \ref{GFCT_random_shift}, we consider a continuous interpolation flow between the sparse matrix $H_N$ and a Gaussian matrix $W_N$ as used in \cite{LeeJiOon2018LlaT, huang2020transition, huang2022edge} and  bound the time derivative of the corresponding expectation along the flow. The time derivative of the expectation has a $q^{-1}$-expansion,
	which is obtained using the cumulant expansion formula (stated in Lemma \ref{cumulant_expansion_lemma}).
	Since the variances of the matrix entries are constant, the second order cumulant terms vanish algebraically, hence the expansion starts with the third order terms consisting of expectations of products of Green function entries with the pre-factors $1/(Nq^{k-2})~(k\geq 3)$ from the moment condition (\ref{moment_intro}).

	To establish the edge universality for sparse matrices in \cite{huang2020transition, huang2022edge}, a GFT for short times $t\gg N^{-1/3} $ is used together with local relaxation results for the Dyson Brownian motion from \cite{landon2017edge,adhikari2020dyson}. The $N$-dependent spectral parameter $\eta$ of the Green function in (\ref{def_G}) is chosen slightly below the typical eigenvalue spacing $N^{-2/3}$.
	With such choice of $\eta$ and small $t$, the terms with off-diagonal Green function entries can be bounded effectively using the local law for the Green function with an error of size $1/(N\eta)\approx N^{-1/3}$ and the so-called Ward identity.  It is then sufficient to show a cancellation principle for the terms consisting of diagonal Green function entries only. In \cite{huang2022edge}, where edge universality was proven down to $q \geq N^{\bd}$, such a cancellation principle is established for the terms consisting of diagonal entries till arbitrarily high order in the $q^{-1}$-expansion.

	To attain the convergence rate in \eqref{introduction_main_result},  our choice of the spectral parameter $\eta$ is much smaller than $N^{-2/3}$ and our GFT is for long times $t\asymp \log N$. Using the local law with small $\eta \gg N^{-1}$, off-diagonal Green function entries are bounded by $1/(N\eta)=o(1)$, which is not sufficiently small. Hence we have to show a cancellation principle for all the terms, including the ones containing several off-diagonal entries.  In order to reduce the computational complexity, we restrict ourselves to $q \geq N^{1/6 + \bd}$ and $\eta \gg N^{-1}+q^{-4}$. Thanks to this assumption, the terms of order six and higher in the cumulant expansions are sufficiently small by the Ward identity and local law estimates.
	We also expect that an improved version of (\ref{introduction_main_result}) without the last error term $O(N^{2/3}q^{-4})$ holds for any $q \geq N^{\bd}$ (see Remark \ref{remark_intro} below), by looking at sufficiently high order terms in the $q^{-1}$-expansion and using a refined  spectral edge from \cite{huang2022edge,lee2021higher}.

	For the third and fifth order terms in the cumulant expansions, the contributions from them are shown to be negligible using the idea of unmatched indices introduced in \cite{green_function_comparison_reference}. The key observation is that these odd order cumulant terms contain at least one unmatched index (see \eg Definition \ref{def_unmatch}) and can be iteratively expanded using similar techniques as developed in \cite{Schnelli_Xu_2022}. By iterative expansions via unmatched indices, one obtains sufficiently many off-diagonal Green function entries which can bounded effectively using the local law estimates by naive power counting.

	A main novelty of this paper is a cancellation principle among all the fourth order terms in the cumulant expansions
	and the terms originating from the corrections to the spectral edge in (\ref{tilde_L_def}). In other words, the choice of the correction terms  in (\ref{tilde_L_def}) is designed to cancel with the fourth order terms in the cumulant expansion, as introduced in \cite{huang2022edge,lee2021higher, huang2020transition}. This cancellation principle is proved by generating a large system of non-trivial identities among all the involved fourth order terms. Each non-trivial identity is generated by using the resolvent identity and the cumulant expansion formula (see more details in Sections \ref{perfect_cancellation_section} and \ref{perfect_cancellation_section_general_case}). We end up using 4288 identities to see the cancellation principle needed for our proof of the GFT stated in Theorem \ref{GFCT_random_shift}.
	
	The manipulations involving non-trivial identities are carried out using computer aided symbolic computations. We present algorithms for manipulating expectations of averaged products of Green function entries, and verifying the cancellations among these terms. We describe the implementation details of our algorithms in Section \ref{implementation_details_section}. All of the code used, along with the non-trivial identities is available on GitHub~\cite{githubrepo}.

\subsection{Assumptions and main results}

 We first set up a general sparse random matrix model that covers the Erd\H{o}s--Rényi graph model introduced above.
\begin{assump}
	Let $H = H_N$ be a symmetric $N \times N$ random matrix with the following assumptions: \begin{enumerate}
		\item For each $N$, the entries $(h_{ij})_{i \leq j}$ are i.i.d. real valued random variables.
		\item The sparsity parameter $q = q(N)$ satisfies $N^{1/6 + \bd} \leq q \leq N^{1/2}$, for some $\bd > 0$. 
		\item  The entries satisfy the following assumptions on the moments:  \begin{equation}
			\EX[h_{ij}] = 0, \quad \EX[\abs{h_{ij}}^2] = \frac{1}{N}, \quad \EX[|h_{ij}|^k] \leq \frac{C^k}{N q^{k-2}}, \; \quad k\geq 3,
			\label{moment_assumptions}
		\end{equation} 
	for some constant $C$ independent of $N$. Additionally, denoting the $k$th cumulant of a random variable by $c^{(k)}(\cdot)$, we assume that
	\begin{equation}
		c^{(1)}(h_{ij})= 0, \quad c^{(2)}(h_{ij})= \frac{1}{N}, \quad c^{(k)}(h_{ij}) = \frac{(k-1)! \kappa_k}{N q^{k-2}}, \; \quad k\geq 3.
		\label{cumulant_assumptions}
	\end{equation} 
Note that the normalized cumulant $\kappa_k=\kappa_k(N)$ may depend on $N$ and is bounded uniformly in $N$.

\item
 There exists $c > 0$ such that $\kappa_4 \geq c$.
	\end{enumerate}
	\label{H_assumptions}
\end{assump}

\begin{remark}
 The last condition that $\kappa_4 \geq c$ was also assumed in previous works (\textit{e.g.}, \cite{huang2020transition}) to guarantee that the normalization for the matrix $H$ by $q$ in (\ref{cumulant_assumptions}) is correct. 
\end{remark}

Our main result is the following theorem. 
\begin{theorem}
		Let $H_N=(h_{ij})_{1 \leq i \leq j \leq N}$ satisfy Assumption \ref{H_assumptions} and denote the ordered eigenvalues of $H_N$ as $\lambda_N \geq \lambda_{N-1} \geq \ldots \geq \lambda_1$. Then for any fixed $r_0 \in \mathbb{R}$ and small $\omega > 0$, we have
		\begin{equation}\label{eq_main}
			\sup_{r > r_0} \abs{\Prob \left(N^{2/3}(\lambda_N - 2 - \frac{6 \kappa_4}{q^2} - \chi) \leq r\right) - \TW_1(r)} \leq N^{\omega} \Big(\frac{1}{N^{1/3}} + \frac{N^{2/3}}{q^4}\Big),
		\end{equation} 
	for sufficiently large $N \geq N_0(r_0, \omega)$, where $\kappa_4$ is from \eqref{moment_assumptions} and $\chi = \frac{1}{N} \sum_{i,j} \left(h_{ij}^2 - \frac{1}{N}\right)$ is a random shift term defined in (\ref{def_chi}).
		\label{main_theorem}
\end{theorem}
\begin{remark} \label{remark_intro}
	 The second condition that $q \geq N^{1/6+\bd}$ of Assumption \ref{H_assumptions} is technical. We expect that a quantitative convergence result such as Theorem \ref{main_theorem} also hold for $q \leq C N^{1/6}$ with a refined shifted edge, possibly as indicated in \cite{huang2020transition,huang2022edge,lee2021higher}. We refer to Remark \ref{lower_q_bound_remark} for more explanations on where the assumption $q \geq N^{1/6+\bd}$ is used.
\end{remark}

The proof of Theorem \ref{main_theorem} relies on a Green function comparison theorem (GFT) stated in Theorem \ref{GFCT_random_shift} below. Roughly speaking, Theorem \ref{GFCT_random_shift} compares the distribution of the rescaled largest eigenvalue of the sparse matrix $H_N$ and that of a Wigner matrix with Gaussian entries, denoted by $\GOEast$. To make the notation in the proof slightly simpler we will not compare with the usual $\GOE$ matrix, that has independent centered Gaussian entries, with variance $1/N$ off the diagonal and variance $2/N$ on the diagonal. Instead we introduce  a slightly different matrix $\GOEast$  with i.i.d. centered Gaussian entries that all have variance $1/N$. Note that $\GOEast$ is also a Wigner matrix, that satisfies conditions 1--3 in Assumption \ref{H_assumptions} with $q = N^{1/2}$. It is known from \cite{Schnelli_Xu_2022} that the convergence rate to the Tracy--Widom law for the largest eigenvalue of $\GOEast$ is bounded by $O(N^{-1/3+\omega})$.

 Before we state Theorem \ref{GFCT_random_shift}, we recall the definition of the Green function of $H_N$ in (\ref{def_G}) and its normalized trace  $m_N$, \ie
 \begin{equation}\label{def:G}
	G(z) := (H_N - zI_N)^{-1}, \qquad m_N(z) := \frac{1}{N} \Tr(G(z)), \qquad z \in \mathbb{C}^+.
\end{equation} Note that $m_N(z)$ also is the {\it Stieltjes transform} of the empirical eigenvalue distribution $\mu_N:=\frac{1}{N} \sum_{i=1}^N \delta_{\lambda_i}$; see more details in Section \ref{sec:local_law}.  
\begin{theorem}[Green function comparison theorem (GFT)]
	Recall $\bd > 0$ and $\kappa_4$ from Assumption \ref{H_assumptions}. Fix  a large $C_0>0$ and a small $\epsilon > 0$ such that $\epsilon < \min\{\bd, 1/12\}$.   Define the random edge $\widehat{L}$ by 
	\begin{align}\label{tilde_L_def} 
		\widehat{L}  := 2 +  \frac{6\kappa_4 }{ q^2} + \chi, \qquad \chi = \frac{1}{N} \sum_{i,j} \left(h_{ij}^2 - \frac{1}{N}\right).
	\end{align} 
 Then for any $\abs{\gamma_1}, \abs{\gamma_2} \leq C_0 N^{-2/3+\epsilon}$, $\eta \in \left[N^{-1+\epsilon} + N^{2\epsilon} q^{-4}, N^{-2/3-\epsilon/2}\right]$, and any smooth function $F \in C_0^{\infty}(\mathbb{R})$,  we have 
\begin{equation}
		\begin{multlined}
		\abs{\EX \left[F \left(N \int_{\gamma_1}^{\gamma_2} \im m_N(\widehat{L} + x + \ii \eta) \di x\right)\right] - \EX^{\GOEast} \left[F \left(N \int_{\gamma_1}^{\gamma_2} \im m_N\left(2 + x + \ii \eta\right) \di x\right)\right]} \\
		 \leq N^{-1/3 + 3 \epsilon} + \frac{N^{2/3 + 3 \epsilon}}{q^4},
		\end{multlined}
		\label{GFCT_random_shift_eq}
	\end{equation} 
for sufficiently large $N \geq N_0(\epsilon, C_0)$.
	\label{GFCT_random_shift}
\end{theorem}

\begin{remark}[Complex case]
	We expect that analogues of Theorem \ref{main_theorem} and Theorem \ref{GFCT_random_shift} also hold for complex Hermitian sparse matrices satisfying the moment assumptions in \eqref{moment_assumptions} and the additional requirement $\EX[h_{ij}^2] = 0$. Most of the proof could be adapted to the complex setting with slight modifications (see similar arguments in \cite{Schnelli_Xu_2022} for complex Wigner matrices), but it is outside the scope of this paper. 
\end{remark}

\subsection{Organization of the paper}
The paper is organized as follows. In Section \ref{preliminaries_section} we collect preliminary results needed in the proof, such as local laws for the Green function, rigidity estimates on the eigenvalues, and cumulant expansion formulas. In Section \ref{proof_of_main_results_section} we first prove the main result Theorem~\ref{main_theorem} using the long-time GFT in Theorem \ref{GFCT_random_shift} as the main technical result. The proof of Theorem \ref{GFCT_random_shift} is based on a dynamical approach and follows directly from Proposition \ref{derivative_observable_small}. The proof of Proposition \ref{derivative_observable_small} is postponed to Section \ref{general_F_cancellation_section}, and Sections \ref{imm_estimate_section}--\ref{imm_estimate_section_2} are devoted to the proof of Proposition \ref{imm_bound_random_edge_lemma}, which is a key ingredient to show Proposition \ref{derivative_observable_small}. In Section \ref{sec:code} we describe in detail how we implemented the symbolic algorithms used in the proofs. Finally, in Appendix \ref{proof_of_lemmas_appendix} we present the proofs of several lemmas and in Appendix \ref{details_about_imm_system} we provide some examples from our algorithms.
\\\\
\noindent\textbf{Notations and conventions:} For any matrix $A \in \mathbb{C}^{N \times N}$ we will use $\norm{A}_2$ to denote the matrix norm induced by the Euclidean vector norm. Also, we define $\norm{A}_{\text{max}}$ as $\norm{A}_{\max} := \max_{ i, j} \abs{A_{ij}}$. We use $\underline{A} := \frac{1}{N} \Tr A$ to denote the normalized trace of $A$. 

We will use $c$ and $C$ to denote positive constants that are independent of $N$, and their values may change from line to line. We use the standard Big-$O$ notation for large $N$, \ie $f(N) = O(g(N))$ means that there exists $C > 0$ such that $\abs{f(N)} \leq C \abs{g(N)}$ for sufficiently large $N$. We write $f(N)\asymp g(N)$ to say that there exists $c, C > 0$ such that $cg(N) \leq f(N) \leq Cg(N)$ for sufficiently large $N$. Further we use $f(N) \ll g(N)$ to denote that $\lim_{N \to \infty} f(N) / \abs{g(N)} = 0$. Finally, we denote the upper complex half-plane by $\C^+ := \{z \in \C : \im z > 0\}$, and the non-negative real numbers by $\R^+ := \{x \in \R : x > 0\}$.

We will say that an event $\mathfrak{E}$ in a probability space holds with \textit{high probability} if for any (large) $\Gamma > 0$, $\Prob(\mathfrak{E}) \geq 1 - N^{-\Gamma}$ for all $N \geq N_0(\Gamma)$. We will also frequently use the following notion of stochastic domination. Properties of stochastic domination can be found in \cite[Proposition 6.5]{dynamical_approach_rmt}.
\begin{definition}
	Let $X = X(N)$ and $Y = Y(N)$ be two sequences of non-negative random variables. We say that $Y$ stochastically dominates $X$, if for all $\epsilon > 0$ and $\Gamma > 0$ we have \begin{equation*}
		\Prob(X(N) > N^\epsilon Y(N)) \leq N^{-\Gamma},
	\end{equation*} for all $N \geq N_0(\epsilon, \Gamma)$. We denote this as $X \prec Y$ or $X = O_{\prec}(Y)$. 
\end{definition} 
Throughout the paper, the letter $N$ will always refer to the size of the matrix we are considering, and the $N_0(\epsilon, \Gamma)$ from the definition of stochastic domination can be chosen to only depend on $\epsilon$, $\Gamma$, and the constants $C$ and $\bd>0$ from Assumption \ref{H_assumptions}. With a slight abuse of notation, we use $x \prec y$ for deterministic sequences $x = x(N)$ and $y = y(N)$ to indicate that for any small $\epsilon > 0$, one has $x \leq N^{\epsilon}y$ for sufficiently large $N \geq N_0(\epsilon)$.

\section{Preliminaries}
\label{preliminaries_section}

\subsection{Local laws and eigenvalue rigidity}\label{sec:local_law}

 For a probability measure $\nu$ on $\R$ denote by $m_\nu$ its Stieltjes transform, \ie
\begin{align}
	m_\nu(z):= \int_\R\frac{\dd \nu(x)}{x-z}\,,\qquad z\in\C^+\,.
\end{align}
Note that $m_{\nu}\,:\C^+\rightarrow\C^+$ is analytic and can be analytically continued to the real line outside the support of $\nu$. Moreover, $m_{\nu}$ satisfies $\lim_{\eta\nearrow\infty}\ii\eta {m_{\mu}}(\ii \eta)=-1$. The Stieltjes transform of the semicircle distribution $\rho_{sc}(x):=\frac{1}{2 \pi} \sqrt{(4-x^2)_+}$ is denoted by $m_{sc}(z)$. It is well know that $ m_{sc}(z)$ is the unique solution to
\begin{equation}\label{self_eq}
	1+z m_{sc}(z)+ m_{sc}^2(z)=0\,,
\end{equation}
satisfying $\im m_{sc}(z)>0$, for $\im z>0$.

We denote the eigenvalues of $H_N$ by $\lambda_1 \leq \lambda_{2} \leq \ldots \leq \lambda_N$. The Stieltjes transform of the empirical eigenvalue measure of $H_N$, $\mu_N:=\frac{1}{N} \sum_{j=1}^N \delta_{\lambda_j}$, is then given by the normalized trace of its Green function (or resolvent) defined in (\ref{def:G}), \ie
$$ m_N(z)= \int_\R\frac{\dd \mu_N(x)}{x-z}=\frac{1}{N}\Tr(G(z)), \qquad G(z) = (H_N - zI_N)^{-1}, \qquad z \in \mathbb{C}^+.$$
We often refer to $z$ as spectral parameter and write $z=E+\ii\eta$, $E\in\R$, $\eta>0$.  \nc Define the domain of the spectral parameter $z$ as \begin{equation}
	S_0 := \left\{z = E + i \eta : \abs{E} \leq 5, 0 < \eta \leq 3\right\}.
	\label{S0_def}
\end{equation} We have the following entrywise local law for the Green function of sparse matrices $H_N$.

\begin{theorem}[Entrywise local law, Theorem 2.8 of \cite{erdos_renyi_graph_I}]  
	 Assume that $H_N$ satisfies conditions 1--3 in Assumption \ref{H_assumptions}, and recall $G(z)$ from \eqref{def:G}. Then uniformly for $z \in S_0$, 
	 \begin{equation}
		\max_{1 \leq i,j \leq N} \abs{G_{ij}(z) - \delta_{ij}m_{sc}(z)} \prec \Psi(z),
		\label{local_law}
	\end{equation} 
with the control parameter given by
\begin{equation}
	\Psi(z):=\frac{1}{q} + \sqrt{\frac{\im m_{sc}(z)}{N \eta}} + \frac{1}{N\eta}, \qquad z \in \mathbb{C}^+.
\end{equation}

\label{entrywise_local_law}
\end{theorem} 

\noindent We also have the following lemma concerning $\im m_{sc}(z)$, see \eg \cite[Lemma 6.2]{dynamical_approach_rmt} for a reference. 

\begin{lemma}
	The Stieltjes transform of the semicircular law satisfies \begin{equation}
		\im m_{sc}(z) \sim \begin{cases}
			\sqrt{\kappa + \eta}, & \text{if } E \in [-2, 2], \\
			\frac{\eta}{\sqrt{\kappa + \eta}}, & \text{otherwise}, \end{cases}
	\label{immsc_asypmtotics}
	\end{equation} uniformly in $z \in S_0$, where $\kappa := \min \{\abs{E - 2}, \abs{E + 2}\}$.
\end{lemma}

We will collect several results on optimal rigidity estimates of eigenvalues of sparse random matrices inside the bulk of spectrum \cite{huang2022edge} and near the edge regime \cite{lee2021higher}, inspired by earlier works \cite{HeKnowles2021,huang2020transition,LeeJiOon2018LlaT}. In fact these results are for any sparse matrices satisfying Assumption \ref{H_assumptions} even with small sparsity parameter $q\geq N^{\bd}$. They construct a random probability measure $\widetilde{\rho}$ whose Stieltjes transform $\widetilde{m}$ solves the equation (which can be viewed as (\ref{self_eq}) with perturbations) 
$$P(z, m) := 1 + zm + m^2 + \sum_{d=2}^D a_{2d} m^{2d} = 0,$$
 where the coefficients $a_{2d}$ are explicit random polynomials in the variables $h_{ij}$.  In particular it was shown in \cite{lee2021higher} that $\abs{a_{2d}} \prec 1/q^{d-1}$.  
  We have the following proposition regarding the behavior of $\widetilde{\rho}$ and $\widetilde{m}$. 
  \begin{prop}[Proposition 2.4 from \cite{huang2022edge}]
	There exists an algebraic function $\widetilde{m} : \mathbb{C}^+ \to \mathbb{C}^+$, that depends on the coefficients $a_2, a_4, \ldots, a_{2D}$ such that the following holds:
	\begin{enumerate}
		\item $\widetilde{m} = \widetilde{m}(z)$ is the solution of the polynomial equation, $P(z, \widetilde{m}(z)) = 0$.
		\item $\widetilde{m}$ is the Stieltjes transform of a symmetric probability measure $\widetilde{\rho}$, with $\supp \widetilde{\rho} = [-\widetilde{L}, \widetilde{L}]$, where $\widetilde{L}$ depends smoothly on the coefficients $a_2, a_4, \ldots, a_{2D}$, and its derivatives with respect to $a_2, a_4, \ldots, a_{2D}$ are uniformly bounded. Moreover, $\widetilde{\rho}$ is strictly positive on $(-\widetilde{L}, \widetilde{L})$ and has square root behavior at the edges of its support. 
		\item We have the following estimate on the imaginary part of $\widetilde{m}$,
		\begin{equation}
			\im [\widetilde{m}(E + \ii \eta )] \sim \begin{cases}
				\sqrt{\kappa + \eta}, & \mbox{if} \;  E \in [-\widetilde{L}, \widetilde{L}], \\
				\frac{\eta}{\sqrt{\kappa + \eta}}, & \mbox{if} \; E \notin [-\widetilde{L}, \widetilde{L}],
			\end{cases} 
			\label{imm_hat_estimate}
		\end{equation} and with high probability,
	\begin{equation*}
		\abs{\partial_2 P(z, \widetilde{m}(z))} \sim \sqrt{\kappa + \eta}, \quad \partial_2^2  P(z, \widetilde{m}(z)) = 1 + O(1/q),
		\end{equation*} 
	where $\kappa = \min \{|E - \widetilde{L}|, |E + \widetilde{L}|\}$ and $\partial_2 P(z,w)=\partial P(z,w)/\partial w$. 
	\end{enumerate}
\end{prop}

 With $\widetilde{m}$ one has the following optimal local law for the averaged trace of the Green function, as stated in Eq (2.76) and (2.78) of \cite{huang2022edge}, and Theorem 2.7 of \cite{lee2021higher} near the spectral edge. Note that the error below is independent of the sparse parameter $q$ (\cf Theorem \ref{entrywise_local_law}).

\begin{theorem}[Averaged local law]
	Uniformly for any $z \in S_0 \cap \{E + \ii \eta :  N \eta \sqrt{\kappa + \eta} \geq N^{\bc}\}$ where $\bc > 0$ is any positive (small) real number, we have 
	\begin{equation}
		\abs{m_N(z) - \widetilde{m}(z)} \prec \frac{1}{N \eta}. 
		\label{optimal_local_law_eq}
	\end{equation}
	\label{optimal_local_law_sparse}
\end{theorem}

As consequences of the local law above, we state the following optimal rigidity estimates of eigenvalues in the bulk \cite[Theorem 1.6]{huang2022edge} and near the edge from \cite[Theorem 2.11]{lee2021higher}. 

\begin{theorem}[Rigidity of eigenvalues]\label{thm_rigidity}
	Let $H_N$ satisfy Assumption \ref{H_assumptions} and denote the eigenvalues of $H_N$ by $\lambda_N \geq \lambda_{N-1} \geq \ldots, \geq \lambda_1$. Denote the classical eigenvalue locations of $\widetilde{\rho}$ as $\widetilde{\mu}_N \geq \widetilde{\mu}_{N-1} \geq \ldots \geq \widetilde{\mu}_1$, \begin{equation*}
		\frac{k - 1/2}{N} = \int_{-\widetilde{L}}^{\widetilde{\mu}_k} \widetilde{\rho}(x) \di x, \qquad 1 \leq k \leq N.
	\end{equation*} Then we have the following optimal rigidity estimates: \begin{equation}
		\abs{\lambda_k - \widetilde{\mu}_k} \prec \frac{1}{N^{2/3} \min\{k, N - k  + 1\}^{1/3}}, 
		\label{optimal_rigidity}
	\end{equation} holds uniformly in $1 \leq k \leq N$. 
\end{theorem}

From the eigenvalue rigidity in (\ref{optimal_rigidity}) near the edge regime, we know  \begin{equation}
	\abs{\lambda_N - \widetilde{L}} \prec N^{-2/3}.
	\label{lee_edge_rigidity}
\end{equation} 
In particular, if  $q \geq N^{1/6 + \bd}$ we have the following estimate (from \cite[Remark~2.5]{huang2022edge}) 
\begin{equation}
	\abs{\lambda_N - \widehat{L}} \prec N^{-2/3},
	\label{edge_rigidity}
\end{equation} 
with $\widehat{L} = 2 + 6 \kappa_4 /q^2 + \chi$ from \eqref{tilde_L_def}. Note that $\abs{\chi} \prec 1/(q \sqrt{N})$ from Assumption \ref{H_assumptions}. Moreover, 
from the square root behavior of $\widetilde{\rho}$ and the optimal rigidity estimates \eqref{optimal_rigidity}, we obtain that, for any fixed constants $C_1$ and $C_2$, (small) $\epsilon > 0$ and (large) $\Gamma > 0$, 
\begin{equation}
	\# \left\{  \lambda_j \in \big[\widehat{L} - C_1 N^{-2/3 + \epsilon}, \widehat{L} + C_2 N^{-2/3 + \epsilon}\big] \right\} \leq N^{2\epsilon},
	\label{total_eigenvalues_at_edge_bound}
\end{equation} with probability $\geq 1 - N^{-\Gamma}$.

\subsection{Relating extremal eigenvalue distribution to the Green function}
 This subsection closely follows Section 2.3 in \cite{Schnelli_Xu_2022}, with suitable adaptations to the sparse setting.  For $E_1 < E_2$ with $E_1, E_2 \in \mathbb{R} \cup \{\pm \infty\}$ we denote the eigenvalue counting function of $H$ by $\mathcal{N}(E_1, E_2) := \# \{j : E_1 \leq \lambda_j \leq E_2\}$. Let $E_L := \widehat{L} + 4 N^{-2/3 + \epsilon}$ with $\widehat{L}$ defined in (\ref{tilde_L_def}). Then for any $E \leq E_L$,  
\begin{equation}\label{number_ev}
	\mathcal{N}(E, E_L) = \Tr \mathcal{\chi}_E(H), \qquad \mbox{with}\quad  \mathcal{\chi}_E := \indicator_{[E, E_L]}. 
\end{equation} 
For $\eta > 0$ we define the mollifier $\theta_{\eta}$ by setting \begin{equation*}
	\theta_{\eta}(x) := \frac{\eta}{\pi (x^2 + \eta^2)} = \frac{1}{\pi} \im \frac{1}{x - \ii \eta}. 
\end{equation*} 
Choosing $\eta=\eta_N \ll 1$, the quantity below approximates the number of eigenvalues given in $[E, E_L]$.
\begin{equation}
	\Tr \chi_E \star \theta_{\eta} (H) = \frac{N}{\pi} \int \chi_E(y) \im m_N(y + \ii \eta) \di y = \frac{N}{\pi} \int_{E}^{E_L} \im m_N(y + \ii \eta) \di y.
	\label{smoothed_trace_identity}
\end{equation}
 The following lemma links the cumulative distribution function of the rescaled largest eigenvalue considered in Theorem \ref{main_theorem} to the expectation of a functional of the Green function considered in Theorem \ref{GFCT_random_shift}. The proof of Lemma \ref{second_eigenval_estimate_lemma} will be presented in Appendix \ref{app_lemma}.
  \begin{lemma}\label{second_eigenval_estimate_lemma}
	Set $l_1 = N^{3\epsilon} \eta$ and $l = N^{3 \epsilon} l_1$ such that $N^{-1 + \epsilon} \leq \eta \ll l_1 \ll l \ll E_L - E \leq C N^{-2/3 + \epsilon}$. Then for any $\Gamma > 0$, we have \begin{equation}
		\Tr \chi_{E + l} \star \theta_{\eta}(H) - N^{-\epsilon} \leq \mathcal{N}(E, \infty) \leq \Tr \chi_{E - l} \star \theta_{\eta}(H) + N^{-\epsilon},
		\label{high_prob_bound_second_eigenval_estimate_lemma}
	\end{equation} with probability bigger than $1 - N^{-\Gamma}$, for $N$ sufficiently large. 

 Let $F : \mathbb{R} \to \mathbb{R}$ be a smooth cut-off function, non-increasing for $x \geq 0$, with $F(x) = 1$ if $|x|\leq 1/9$ and $F(x)=0$ if  $|x|\geq 2/9$. Then we have 
   \begin{equation*}
		\EX \left[F\left(\Tr \chi_{E-l} \star \theta_\eta (H)\right)\right] - N^{-\Gamma} \leq \Prob\left(\mathcal{N}(E, \infty) = 0 \right) \leq \EX \left[F\left(\Tr \chi_{E+l} \star \theta_\eta (H)\right)\right] + N^{-\Gamma}.
	\end{equation*}
 \end{lemma}

\subsection{Cumulant expansion formula} 

A crucial tool of this paper is the cumulant expansion formula presented below.

\begin{lemma}[Lemma 3.1 from \cite{cumulant_expansion_ref}]
	Let $h$ be a random variable with finite moments. The $k$th cumulant of $h$ is defined as \begin{equation*}
		c^{(k)} := (-\ii)^k \left(\frac{\di^k}{\di t^k} \log \EX \e^{\ii th}\right) \bigg|_{t=0}.
	\end{equation*} Let $f : \mathbb{R} \to \mathbb{C}$ be a smooth function. Then for any fixed $l$ we have \begin{equation}
		\EX[hf(h)] = \sum_{k+1=1}^l \frac{1}{k!}c^{(k+1)} \EX[f^{(k)}(h)] + R_{l+1},
		\label{cumulant_expansion_eq}
	\end{equation}
	provided that all the expectations in \eqref{cumulant_expansion_eq} exist. The error term satisfies \begin{equation}
		\abs{R_{l+1}} \leq C_l \EX \abs{h}^{l+1} \sup_{\abs{x} \leq M} \abs{f^{(l)}(x)} + C_l \EX \left[\abs{h}^{l+1} 1_{\abs{h} >M}\right] \norm{f^{(l)}}_\infty,
		\label{cumulant_expansion_error_bound}
	\end{equation} where $M > 0$ is an arbitrary fixed cutoff. 
	\label{cumulant_expansion_lemma}
\end{lemma}

\section{Proofs of main results}
\label{proof_of_main_results_section}

In this section we first prove Theorem \ref{main_theorem} using Theorem \ref{GFCT_random_shift} as a technical input. After this, we present the proof of Theorem \ref{GFCT_random_shift}.
\subsection{Proof of Theorem \ref{main_theorem}}

We will use the GFT in Theorem \ref{GFCT_random_shift} along with Lemma \ref{second_eigenval_estimate_lemma} to show Theorem \ref{main_theorem}. The proof closely follows Section 3 in \cite{Schnelli_Xu_2022}. 

\begin{proof}[Proof of Theorem \ref{main_theorem}]
 From the rigidity of eigenvalues in (\ref{edge_rigidity}), we have for any small $\epsilon>0$ and large $\Gamma>10$,
 $$\sup_{|r| \geq N^{\epsilon}} \Prob\Big( N^{2/3} (\lambda_N- \widehat{L} ) > r \Big) \leq  N^{-\Gamma}.$$
Together with right tail asymptotics of $\TW_1(r)$ from \cite{tw_asymptotics}, it suffices to show Theorem \ref{main_theorem} for $r_0 < r < N^{\epsilon}$. We use the notations \begin{equation*}
		E = \widehat{L} + N^{-2/3}r, \quad \text{and} \quad E_L = \widehat{L} + 4N^{-2/3 + \epsilon}. 
	\end{equation*} 
Fix $\eta = N^{-1 + \epsilon} + N^{2\epsilon} q^{-4}$ and $l = N^{6 \epsilon} \eta$, as in Lemma \ref{second_eigenval_estimate_lemma}. We choose $\epsilon > 0$ sufficiently small so that $l \ll N^{-2/3}$. From \eqref{smoothed_trace_identity} and Lemma \ref{second_eigenval_estimate_lemma} we see that \begin{equation}
	\begin{multlined}
		\EX\left[F\left(\frac{N}{{\pi}} \int_{N^{-2/3} r - l}^{4N^{-2/3 + \epsilon}} \im m_N (\widehat{L} + x + \ii \eta) \di x\right)\right] - N^{-\Gamma} \\
	 \leq\Prob\left(N^{2/3}(\lambda_N- \widehat{L}) \leq r \right) =	\Prob(\mathcal{N}(E, \infty) = 0) \\
	\qquad\qquad\qquad\qquad	\qquad\qquad \leq \EX \left[F\left(\frac{N}{{\pi}} \int_{N^{-2/3} r + l}^{4N^{-2/3 + \epsilon}} \im m_N (\widehat{L} + x + \ii \eta) \di x\right)\right] + N^{-\Gamma}. 
	\end{multlined}
	\label{scaled_cdf_estimate}
	\end{equation} 
Recall the comparison statement in Theorem \ref{GFCT_random_shift}. The lower and upper expectations in (\ref{scaled_cdf_estimate}) can be bounded using corresponding expectations for $\GOEast$, \ie 
\begin{align}\label{GFCT_eq}
	&\sup_{r_0 < r < N^{\epsilon}}	\abs{\EX \left[F \left(\frac{N}{\pi} \int_{N^{-2/3}r \pm l}^{4N^{-2/3 + \epsilon}} \im m_N(\widehat{L} + x + \ii \eta) \di x\right)\right] - \EX^{\GOEast} \left[F \left(\frac{N}{\pi} \int_{N^{-2/3}r \pm l}^{4N^{-2/3 + \epsilon}} \im m_N\left(2 + x + \ii \eta\right) \di x\right)\right]} \nonumber\\
	&\qquad\qquad \qquad \leq N^{-1/3 + 3 \epsilon} + \frac{N^{2/3 + 3 \epsilon}}{q^4}. 
\end{align} 
Note that similar inequalities as in (\ref{scaled_cdf_estimate}) also hold for $\GOEast$, which is a Wigner matrix, with $\widehat L$ replaced by 2. Combining with the convergence estimates \cite[Eq.(3.4)--(3.5)]{Schnelli_Xu_2022} for Wigner matrices,  we obtain that 
\begin{equation}
		\sup_{r_0 < r < N^{\epsilon}} \abs{\EX^{\GOEast} \left[F\left(\frac{N}{\pi} \int_{N^{-2/3}r \pm l}^{4N^{-2/3 + \epsilon}}\im m_N (2 + x + \ii \eta) \di x \right)\right] - \TW_1(r)} = O(N^{2/3}l).
		\label{wigner_estimate_imm_average}
	\end{equation} 
Combining \eqref{wigner_estimate_imm_average} with \eqref{GFCT_eq} and choosing $\epsilon < \omega /9$, we obtain
that \begin{equation}\label{final}
		\sup_{ r_0 < r < N^{\epsilon}} \abs{\EX \left[F \left(\frac{N}{\pi} \int_{N^{-2/3}r \pm l}^{4N^{-2/3 + \epsilon}} \im m_N(\widehat{L} + x + \ii \eta) \di x\right)\right] - \TW_1(r)} \leq N^{-1/3 + \omega} + N^{2/3 + \omega} q^{-4}.
	\end{equation} 
Theorem \ref{main_theorem} hence follows directly from \eqref{scaled_cdf_estimate} and (\ref{final}).

\end{proof}

\subsection{Proof of Theorem \ref{GFCT_random_shift}}
\label{interpolation_setup_section}

Our proof strategy will involve the following matrix interpolating flow $H_N(t)=(h_{ij}(t))_{1\leq i,j\leq N}$ given by
\begin{equation}
	H_{N}(t) := \e^{-\frac{t}{2}} W_N + \sqrt{1 - \e^{-t}} H_N, \quad t \geq 0,
	\label{H_dynamics_new}
\end{equation} 
where $H_N=(h_{ij})_{1\leq i,j\leq N}$ is an $N \times N$ sparse matrix satisfying Assumption \ref{H_assumptions}, and $W_N=(w_{ij})_{1\leq i,j\leq N}$ is an $N \times N$ matrix sampled from $\GOEast_N$, independent of $H_N$.  We will often omit the $N$ in the notations and write $H(t)=H_N(t)$. Note that $H(0) = W$ and $H(\infty) = H$. We define the Green function of the time dependent matrix $H(t)$ and its normalized trace by
\begin{equation*}
	G(t, z) := \frac{1}{H(t) - zI}, \quad m_N(t, z) := \frac{1}{N} \sum_{i} (G(t, z))_{ii}, \quad z \in \mathbb{C}^+.
\end{equation*} 
Moreover, choose the upper spectral edge with time-dependent shift as:
\begin{equation}
	\widehat{L}_t := 2 + \frac{6 \kappa_4 \left(1 - e^{-t}\right)^2}{q^2} + \chi(t),
	\label{tilde_L_t_def}
\end{equation}
where $\kappa_4$ from (\ref{cumulant_assumptions}) is deterministic and $\chi(t)$ is the random correction term given by
\begin{equation}
	\chi(t) := \left(1 - \e^{-t}\right) \frac{1}{N} \sum_{i,j} \left(h_{ij}^2 - \frac{1}{N}\right). 
	\label{new_random_edge_interpolation}
\end{equation}
Note that $\chi(0) = 0$ and $\chi(\infty) = \chi$ given by \eqref{tilde_L_def}. We also have $\widehat{L}(0)=2$ and $\widehat{L}(\infty) = \widehat{L}$  defined in (\ref{tilde_L_def}).

To show Theorem \ref{GFCT_random_shift} we will mainly study the quantity \begin{equation*}
	\im m_N(t, \widehat{L}_t + x + \ii \eta)=\im \Big[ \frac{1}{N}  \Tr \big(G(t, \widehat{L}_t + x + \ii \eta) \big)\Big],
\end{equation*} 
where the spectral parameter of $G(t,z)$ is chosen near the time-dependent  spectral edge $\widehat{L}_t$ defined in (\ref{tilde_L_t_def}). To abbreviate notations, we introduce
\begin{equation}\label{resolvent_t}
	z(t) := \widehat{L}_t + x + \ii \eta, 
\end{equation} 
and define the following observable:
\begin{equation}
	X(t) := N \int_{\gamma_1}^{\gamma_2} \im m_N(t, z(t)) \di x=\int_{\gamma_1}^{\gamma_2} \im \big[\Tr G(t, z(t))\big] \di x,
	\label{X_observable_def}
\end{equation}
where $\gamma_1$ and $\gamma_2$ are the parameters used in Theorem \ref{GFCT_random_shift}. Note that the left hand side of \eqref{GFCT_random_shift_eq} is exactly $\abs{\EX[F(X(\infty))] - \EX[F(X(0))]}$. To ease notations we will often omit $(t, z(t))$ when writing $G(t, z(t))$ and $m_N(t, z(t))$, \ie we set $G=G(t, z(t))$ and $m_N=m_N(t, z(t))$ in the remainder of this paper. 

Our goal is now to bound $\abs{\EX[F(X(\infty))] - \EX[F(X(0))]}$, and we do this by bounding $\dF \EX[F(X(t))]$ as follows. \begin{prop}
	Uniformly for $t \in [0, 10 \log N]$ and $N^{-1 + \epsilon} + N^{2 \epsilon} q^{-4} \leq \eta \leq N^{-2/3 - \epsilon/2}$ it holds that \begin{equation}
		\dF \EX[F(X(t))] \leq N^{-1/3 + 9 \epsilon/4} + \frac{N^{2/3 + 9 \epsilon/4}}{q^4},
		\label{derivative_observable_small_eq}
	\end{equation} for $N \geq N_0(C_0, \epsilon)$.
	\label{derivative_observable_small}
\end{prop}

The proof of Proposition \ref{derivative_observable_small} is postponed to Section \ref{general_F_cancellation_section}. Using Proposition \ref{derivative_observable_small} we are ready to show Theorem \ref{GFCT_random_shift}. 

\begin{proof}[Proof of Theorem \ref{GFCT_random_shift}]
		We wish to bound \begin{equation}
		\begin{multlined}
			\abs{\EX\left[F(X(0))\right] - \EX\left[F(X(\infty))\right]} \\
			\leq  \abs{\int_{0}^{10 \log N} \dF \EX[F(X(t))] \di t} + \abs{\EX\left[F(X(10 \log N))\right] - \EX\left[F(X(\infty))\right]}
		\end{multlined}
		\label{main_gfc_inequality}
	\end{equation} From Proposition \ref{derivative_observable_small} we may bound the integral term in \eqref{main_gfc_inequality} as \begin{equation*}
		\abs{\int_{0}^{10 \log N} \dF \EX[F(X(t))] \di t} \leq N^{-1/3 + 5\epsilon/2} +  \frac{N^{2/3 + 5 \epsilon/2}}{q^4}.
	\end{equation*} 
	
	Next we use a perturbation argument to handle the second term on the second line of \eqref{main_gfc_inequality}. Using the norm inequality $\norm{A}_{\text{max}} \leq \norm{A}_2 \leq N \norm{A}_{\text{max}}$, resolvent expansion, that $\norm{G}_2 \leq 1/\eta$, and the definitions of $H(t)$ and $z(t)$ in \eqref{H_dynamics_new} and \eqref{resolvent_t} we obtain \begin{equation*}
		\begin{split}
				& \quad \, \norm{G(10 \log N, z(10 \log N)) - G(\infty, z(\infty))}_{2} \\
	& \leq \norm{G(\infty, z(\infty))}_2 \norm{H(\infty) - H(10 \log N) - (z(\infty) - z(10 \log N))I }_2 \norm{G(10 \log N, z(10 \log N))}_2 \\
	& \prec \frac{1}{N^{7/2} \eta^2}.
		\end{split} 
	\end{equation*} Using that $F$ is Lipschitz continuous we have \begin{equation*}
		 \abs{\EX\left[F(X(10 \log N))\right] - \EX\left[F(X(\infty))\right]} \leq N^{-1}.
	\end{equation*} Inserting this bound into \eqref{main_gfc_inequality} finishes the proof. \end{proof}

\section{Estimate on $\EX[\im m_N]$} 
\label{imm_estimate_section}

The main ingredient in the proof of Proposition \ref{derivative_observable_small} is the next proposition, which also acts as a simpler version of Proposition \ref{derivative_observable_small}. Proposition \ref{imm_bound_random_edge_lemma} is essentially the same as Proposition \ref{derivative_observable_small} with $F$ chosen as $F(x) = x$, and without the averaging integral in the definition of $X(t)$. The ideas we use to prove Proposition \ref{imm_bound_random_edge_lemma} will also be used to prove Proposition \ref{derivative_observable_small}.

\begin{prop}
	 Recall $\bd > 0$ from Assumption \ref{H_assumptions}. For any $\epsilon$ such that $0 < \epsilon < \min\{\bd, 1/12\}$, and constant $C_0 > 0$ we define the spectral edge regime \begin{equation}
	 	\edge(t, \epsilon, C_0) \equiv \edge(t) :=  \left\{z = E + \ii \eta : \frac{N^{2\epsilon}}{q^4} + N^{-1+\epsilon} \leq \eta \leq N^{-2/3 - \epsilon/2}, E = \widehat{L}_t + x, \abs{x} \leq C_0N^{-2/3 + \epsilon}\right\}.
	 	\label{edge_regime}
	 \end{equation} Then uniformly for $z(t) = \widehat{L}_t + x + \ii \eta \in \edge(t)$ and $t \in [0, 10 \log N]$ we have 
	 \begin{equation}
	 	\abs{\dm  \EX[\im m(t, z(t))]} \leq N^{-\epsilon/8} \EX[\im m(t, z(t))] + N^{-1/3 + \epsilon/8}.
	 	\label{time_derivative_gronwall_bound}
	 \end{equation} As a consequence of \eqref{time_derivative_gronwall_bound} there exists a constant $C > 0$ independent of $\epsilon$ such that \begin{equation}
		\EX \left[\im m(t, \widehat{L}_t + x + \ii \eta)\right] \leq C N^{-1/3 + \epsilon},
		\label{imm_bound}
	\end{equation} uniformly for $\widehat{L}_t + x + \ii \eta \in \edge(t)$ and $t \in [0, 10 \log N]$.
	\label{imm_bound_random_edge_lemma}
\end{prop}

The proof idea behind \eqref{time_derivative_gronwall_bound} is to expand $\dm \EX [\im m(t, z(t))]$ using Lemma \ref{cumulant_expansion_lemma} in terms of cumulants of the matrix entries. The main difficulty is to show that the leading terms associated to the fourth order cumulants cancel with the terms generated from the shifted edge introduced in (\ref{tilde_L_t_def}). After cancellations we bound the remaining non-leading terms as the right hand side of \eqref{time_derivative_gronwall_bound}, by using the Ward identity together with the assumptions $q \geq N^{1/6 + \bd}$ and $\eta \geq N^{-1 + \epsilon} + N^{2 \epsilon}/q^4$ from $\edge(t)$ (see also Remark \ref{lower_q_bound_remark} below). We proceed to show that \eqref{time_derivative_gronwall_bound} implies \eqref{imm_bound}, by using Grönwall's inequality and a bound on $\EX[\im m(0, z(0))]$ from \cite[Lemma 5.4]{Schnelli_Xu_2022}.

\begin{proof}[Proof that \eqref{time_derivative_gronwall_bound} implies \eqref{imm_bound}]
		We start by defining $u(t) := \EX[\im m(t, \widehat{L}_t + x + \ii \eta)]$. Then we define $$v(t) := \e^{-tN^{-\epsilon/8}} u(t) - \int_{0}^{t} \e^{-rN^{-\epsilon/8}} N^{-1/3 + \epsilon/8} \di r. $$ Differentiating $v(t)$ and using \eqref{time_derivative_gronwall_bound} we get $\dm v(t) \leq 0$, which implies that $v(t) \leq u(0)$ for all $t \geq 0$, since by the definition of $v$ we have $u(0) = v(0)$. Rewriting the definition for $v(t)$ we may express $u(t)$. \begin{equation*}
		u(t) = \e^{t N^{-\epsilon/8}} v(t) + \int_{0}^{t} \e^{(t-r)N^{-\epsilon/8}}N^{-1/3 + \epsilon/8} \di r. 
	\end{equation*} Since $t \leq 10 \log N$ we have that asymptotically $\e^{t N^{-\epsilon/8}} \to 1$ as $N \to \infty$, which means that \begin{equation*}
		u(t) \leq 2 v(t) + N^{-1/3 + \epsilon/4},
	\end{equation*} for $N$ sufficiently large. Combining this with $v(t) \leq u(0)$ and Lemma 5.4 from \cite{Schnelli_Xu_2022}, which says that $u(0) \leq C N^{-1/3 + \epsilon}$ uniformly for $z(0) \in \edge(0)$, we see that \begin{equation*}
		\begin{split}
			u(t) & \leq 2v(t) + N^{-1/3 + \epsilon/4} \\
			& \leq 2u(0) + N^{-1/3 + \epsilon/4} \\
			& \leq C N^{-1/3 + \epsilon},
		\end{split}
	\end{equation*} where the constant $C$ is independent of $\epsilon$. 
\end{proof}

In the remaining of this section, we prove (\ref{time_derivative_gronwall_bound}). We now begin expanding $\dm \EX [m(t, z(t))]$, which is a function in $z(t)$ and $\{h_{ab}(t)\}_{a\leq b}$ with $h_{ab}(t)=h_{ba}(t)$:
\begin{align}\label{expand0}
	\dm \EX [m(t, z(t))] = &\dm \EX \left[\frac{1}{N} \sum_{v} G_{vv}\right] \nonumber\\
	=& \frac{1}{N} \sum_v \sum_{a \leq b} \EX \left[\frac{\partial (G(t, z))_{vv}}{\partial h_{ab}(t)} \bigg|_{z = z(t)} \frac{\partial h_{ab}(t)}{\partial t}\right] + \frac{1}{N} \sum_{v} \EX \left[\frac{\partial {G}_{vv}}{\partial z(t)}  \frac{\partial z(t)}{\partial t}\right].
\end{align} 
Using the matrix flow $H(t)$ in (\ref{H_dynamics_new}), $z(t)$ chosen in (\ref{resolvent_t}), the definition of Green function $G=G(t, z(t))=(H(t)-z(t))^{-1}$, and that $H(t)$ is real symmetric (which implies $G=G^{\mathrm{T}}$), we obtain from (\ref{expand0}) that
\begin{equation}\label{expand}
	\begin{split}
		\dm \EX [m(t, z(t))] = 
		& = \frac{1}{N} \sum_v \sum_{a \leq b} \EX \left[ \frac{\partial h_{ab}(t)}{\partial t} (-(2 - \delta_{ab})G_{va} G_{bv}) \right] + \frac{1}{N} \sum_{v} \EX \left[\frac{\partial \widehat{L}_t}{\partial t}  \sum_{j} G_{vj} G_{jv} \right] \\ 
		& = \frac{1}{N} \sum_{v,a, b} \EX \left[ - \frac{\partial h_{ab}(t)}{\partial t} G_{va} G_{bv} \right] + \frac{1}{N} \sum_{v, j} \EX \left[\frac{\partial \widehat{L}_t}{\partial t}  G_{vj} G_{jv} \right].
	\end{split}
\end{equation} 
By taking time derivatives of $h_{ab}(t)$ from (\ref{H_dynamics_new}) and $\widehat{L}_t$ from (\ref{tilde_L_t_def}), we obtain that
\begin{equation}
	\begin{split}
		\dm \EX [m(t, z(t))] 
		& = \frac{1}{N} \sum_{v,a, b} \EX \left[ - \frac{\e^{-t}}{2\sqrt{1-\e^{-t}}} h_{ab} G_{va} G_{bv} + \frac{1}{2} \e^{-\frac{t}{2}} w_{ab} G_{va} G_{bv} \right] \\
		& \quad + \frac{1}{N} \sum_{v, j} \EX \left[\frac{12 \e^{-t}(1-\e^{-t})\kappa_4}{q^2} G_{vj} G_{jv} \right] \\
		& \quad + \frac{1}{N} \sum_{v, j} \EX \left[\e^{-t}\frac{1}{N} \sum_{a,b}\left(h_{ab}^2 - \frac{1}{N}\right) G_{vj} G_{jv} \right]. 
	\end{split}
	\label{imm_time_derivative_expansion_pre_cumulant}
\end{equation}

We notice that several of the terms in \eqref{imm_time_derivative_expansion_pre_cumulant} have the form $\EX[hf(h)]$ as in the cumulant expansion formula \eqref{cumulant_expansion_eq} of Lemma \ref{cumulant_expansion_lemma}, with $f$ chosen as a product of Green function entries. Indeed, we have the following application of Lemma \ref{cumulant_expansion_lemma}. The proof will be postponed to the appendix. Throughout this section we will use Lemma \ref{cumulant_remainder_small_lemma_imm_case} without explicitly referring to it.

\begin{lemma}\label{lemma_expand_special}
	 Let $h_{ab}$ be a fixed entry of $H$ from Assumption \ref{H_assumptions} and let $f = f(t, H, W, x, \eta)$ be
	 \begin{equation}
		f := h_{ab}^{p_h} \prod_{i = 1}^{n} G_{x_i y_i},
		\label{f_imm_case}
	\end{equation} where $p_h$ is a non-negative integer. Then we have the cumulant expansion formula
	\begin{equation}
				\EX[h_{ab}f] = \sum_{k+1=1}^l \frac{\kappa_{k+1}}{Nq^{k-1}} \EX\left[\frac{\partial^k}{\partial h_{ab}^k}f\right] + R_{l+1},
				\label{specific_cumulant_expansion_formula_imm_case}
		\end{equation} with the remainder term $R_{l+1}$ satisfying 
\begin{equation}
		\abs{R_{l+1}} \prec \frac{1}{N q^{l-1}},
		\label{cumulant_remainder_small_eq_imm_case}
	\end{equation} uniformly for $t \in [0, 10\log N]$, $\eta \geq N^{-1 + \epsilon}$, and $\abs{x} \leq 1$. Furthermore, for a fixed entry $w_{ab}$ of $W$ we have the formula \begin{equation*}
		\EX[w_{ab}f] = \frac{1}{N} \EX \left[\frac{\partial}{\partial w_{ab}} f\right].
		\label{gaussian_cumulant_expansion}
	\end{equation*}
		\label{cumulant_remainder_small_lemma_imm_case}
\end{lemma}

 Performing the above cumulant expansions to sufficiently large order $l$ on all expectations with an $h_{ab}$ or $w_{ab}$, we obtain 
\begin{equation}
	\begin{split}
	\dm \EX [m(t, z(t))] = & \frac{1}{2N^2} \sum_{v,a, b} - \frac{\e^{-t}}{\sqrt{1-\e^{-t}}} \sum_{k+1=1}^l \frac{\kappa_{k+1}}{q^{k-1}} \EX \left[\frac{\partial^k (G_{va} G_{bv})}{\partial h_{ab}^k}  \right]  \\ 
		&  + \frac{1}{2N^2} \sum_{v,a, b}  \e^{-\frac{t}{2}} \EX \left[ \frac{\partial (G_{va} G_{bv})}{\partial w_{ab}}  \right]  \\
		&  + \frac{1}{N} \sum_{v, j} \frac{12 \e^{-t}(1-\e^{-t})\kappa_4}{q^2} \EX \left[ G_{vj} G_{jv} \right]  \\
		&  + \frac{1}{N^3} \sum_{v, j, a, b} \e^{-t} \sum_{k+1 = 1}^{l} \frac{\kappa_{k+1}}{q^{k-1}} \EX \left[\frac{\partial^k (h_{ab} G_{vj} G_{jv})}{\partial h_{ab}^k}  \right]  - \frac{1}{N} \sum_{v, j} \e^{-t} \EX \left[ G_{vj} G_{jv} \right]  \\
		&  + O_{\prec}\left(\frac{N}{q^{l-1}}\right) .	%
	\end{split}
	\label{imm_time_derivative_expansion}
\end{equation}
Fixing and $D > 0$, we can choose the truncating order, $l$ in \eqref{imm_time_derivative_expansion} large enough such that $N/q^{l-1} \leq N^{-D}$.  We remark that the first two lines on the right side of (\ref{imm_time_derivative_expansion}) are from the matrix flow $H(t)$ in (\ref{H_dynamics_new}), while the third and fourth line are from the time dependent correction terms of $z(t)$ in (\ref{resolvent_t}), corresponding to the last two terms in (\ref{tilde_L_t_def}), respectively.

To compute the terms in \eqref{imm_time_derivative_expansion} precisely, we present the following differentiation rules for the Green function, which follow directly from the definition of the resolvent and (\ref{H_dynamics_new})--(\ref{resolvent_t}). 

\begin{lemma} We have the following differentiation rules for $G_{ij} = \left(G(t, z(t))\right)_{ij}$:  
		\begin{align}
		\begin{split}
			\frac{\partial G_{ij}}{\partial h_{ab}} & =  - \left(1 - \e^{-t}\right)^{1/2} \Big(G_{ia} G_{bj} + G_{ib} G_{aj}\Big) \\
			& \quad \, + \left(1 - \e^{-t}\right)^{1/2} \delta_{ab} G_{ia} G_{bj} \\
			& \quad \, + \left(1 - \e^{-t}\right) 4 h_{ab} \frac{1}{N} \sum_{u} G_{iu} G_{uj} \\
			& \quad \, - \left(1 - \e^{-t}\right) 2 \delta_{ab} h_{ab} \frac{1}{N} \sum_{u} G_{iu} G_{uj},
		\end{split} \label{G_h_derivative} \\
		\begin{split}
			\frac{\partial G_{ij}}{\partial w_{ab}} & =
			- \e^{-t/2} \Big(G_{ia} G_{bj} + G_{ib} G_{aj}\Big) + 
			\e^{-t/2} \delta_{ab} G_{ia} G_{bj}.
		\end{split} \label{G_w_derivative}
	\end{align} 
	\label{differentiation_lemma_imm_case}
\end{lemma}

Using Lemma \ref{differentiation_lemma_imm_case}, we aim to compute the right side of (\ref{imm_time_derivative_expansion}) and bound the resulting terms by $N^{-c\epsilon}\EX[\im m(t,z(t))]$ for some $c>0$. The remaining proof is outlined below. 

\smallskip

{\bf 1) Second order terms}:
The first observation is that the second order terms ($k=1$) in the cumulant expansions (\ref{imm_time_derivative_expansion}) cancel mostly with the terms on the second line of \eqref{imm_time_derivative_expansion}, with the error terms listed in (\ref{oringal}) below.  Such cancellations are consequences of the fact that the entries of $H(t)$ in (\ref{H_dynamics_new}) have invariant second cumulants. More precisely, the corresponding terms from the first two lines of the differentiation rule (\ref{G_h_derivative}) will cancel with those from (\ref{G_w_derivative}), so the remaining terms are from the last two lines of (\ref{G_h_derivative}) with the prefactor $h_{ab}$:
\begin{align}
	(*):=&\frac{1}{2N^2} \sum_{v,a, b} - \frac{\e^{-t}}{\sqrt{1-\e^{-t}}}  \EX \left[\frac{\partial (G_{va} G_{bv})}{\partial h_{ab}}  \right] +\frac{1}{2N^2} \sum_{v,a, b}  \e^{-\frac{t}{2}} \EX \left[ \frac{\partial (G_{va} G_{bv})}{\partial w_{ab}}  \right]\nonumber\\
	=&\frac{1}{N^2} \sum_{v,a, b} - \e^{-t}\sqrt{1+\e^{-t}} (2-\delta_{ab}) \EX \left[h_{ab}   \left(  \frac{1}{N}\sum_{u} \Big( G_{vu}G_{ua} G_{bv}+G_{va} G_{bu}G_{uv} \Big) \right) \right].\label{oringal}
\end{align}
Further expanding these terms with centered random variables $h_{ab}$ using Lemma \ref{lemma_expand_special} as in \eqref{imm_time_derivative_expansion}, we obtain
\begin{equation}
	\begin{split}
			(*)=&\frac{1}{N^3} \sum_{v,a, b} -\e^{-t}\sqrt{1+\e^{-t}} (2-\delta_{ab})  \sum_{k+1=1}^l \frac{\kappa_{k+1}}{q^{k-1}} \EX \left[\frac{1}{N} \sum_{u} \frac{\partial^k \Big( G_{vu}G_{ua} G_{bv}+G_{va} G_{bu}G_{uv} \Big) }{\partial h^k_{ab}}  \right]\\
		&\qquad +O_{\prec}\left(\frac{1}{q^{l-1}}\right),\label{111}
	\end{split}
\end{equation}
 with sufficiently large $l$ so that $N/q^{l-1} \leq N^{-D}$ for $D > 0$. Note that we have gained an additional $N^{-1}$ factor (from the variances of $h_{ab}$) compared to the original summations $\frac{1}{N^2}\sum_{v,a,b}[\cdots]$ in (\ref{oringal}).  We thus say the terms in (\ref{oringal}) with the prefactor $h_{ab}$ are {\it non-leading terms}, that will be collected in Definition \ref{def_gronwall_term} below. Thanks to the index $v$ in (\ref{imm_time_derivative_expansion}), all these terms have at least two off-diagonal Green function entries. Hence they can be controlled effectively using the Ward identity:
 \begin{align}\label{ward_identity}
    \frac{1}{N} \sum_{j,k} |G_{jk}(z)|^2=\frac{\im m_N(z)}{\eta}, \qquad z=E+\ii \eta, \quad \eta>0.
 \end{align}
In general for any non-leading term belonging to Definition \ref{def_gronwall_term}, one can obtain the desired upper bound as in \eqref{time_derivative_gronwall_bound} using Lemma \ref{lemma_lemma_tau_g_ineq} stated in the next section. 

Moreover, for the second order terms ($k=1$) on the fourth line of (\ref{imm_time_derivative_expansion}), the leading term from $\partial h_{ab}$ acting on $h_{ab}$ will cancel precisely with the last term on the fourth line of (\ref{imm_time_derivative_expansion}). This cancellation is due to the fact that the random correction $\chi(t)$ in (\ref{new_random_edge_interpolation}) is centered. Thus we obtain non-leading terms with the prefactor $h_{ab}$ given by
\begin{equation}
	\begin{split}
		&\frac{1}{N^3} \sum_{v, j, a, b} \e^{-t}   \EX \left[\frac{\partial (h_{ab} G_{vj} G_{jv}) }{\partial h_{ab}}  \right]  - \frac{1}{N} \sum_{v, j} \e^{-t} \EX \left[ G_{vj} G_{jv} \right] \\
	&= \frac{1}{N^3} \sum_{v, j, a, b} \e^{-t}   \EX \left[h_{ab} \frac{\partial (G_{vj} G_{jv}) }{\partial h_{ab}}   \right],
	\label{222}
	\end{split}
\end{equation}

where we can gain an additional factor $N^{-1}$ from further expanding $h_{ab}$ as in (\ref{111}). In fact these non-leading terms belong to case \ref{pheq1_vocc2} of Definition \ref{def_gronwall_term}, and can be bounded using Lemma \ref{lemma_lemma_tau_g_ineq} below.

\smallskip

{\bf 2) Fourth order terms}:
We next consider the most critical terms with the fourth cumulant $\kappa_4$, \ie fourth order terms with $k=3$ from the first and fourth line of (\ref{imm_time_derivative_expansion}), together with the third line of \eqref{imm_time_derivative_expansion}. %
Note that in the fourth line of (\ref{imm_time_derivative_expansion}), if no partial derivative $\partial h_{ab}$ acts on $h_{ab}$, then the resulting terms are indeed non-leading terms with the prefactor $h_{ab}$ (included in case \ref{pheq1_dqge1} of Definition \ref{def_gronwall_term}) from which one gains an additional $N^{-1}$ as in (\ref{111}). Hence we collect all the fourth order terms with the prefactor  $\kappa_4$ as
\begin{equation}
	\begin{split}
		\mbox{fourth order terms with } \kappa_4=&\frac{1}{2N^2} \sum_{v,a, b} - \frac{\e^{-t}}{\sqrt{1-\e^{-t}}}  \frac{\kappa_{4}}{q^{2}} \EX \left[\frac{\partial^3  (G_{va} G_{bv})}{\partial h_{ab}^3}  \right]\\
		&+\frac{1}{N} \sum_{v, j} \frac{12 \e^{-t}(1-\e^{-t})\kappa_4}{q^2} \EX \left[ G_{vj} G_{jv} \right]\\
		&+\frac{1}{N^3} \sum_{v, j, a, b} \e^{-t}  \frac{\kappa_{4}}{q^{2}} \EX \left[\frac{\partial^2 (G_{vj} G_{jv})}{\partial h_{ab}^2}   \right]+\mbox{non-leading terms}.
	\end{split}
	\label{term1}
\end{equation}
 We remark that the first line on the right side above is the collection of fourth order terms from the randomness of $H(t)$ in (\ref{H_dynamics_new}),  while the last two lines contain the fourth order terms originating from the time-dependent correction terms of $z(t)$ introduced in  (\ref{tilde_L_t_def}). The key observation is the cancellations between the above fourth order terms with $\kappa_4$ up to an desired error term. These cancellations crucially rely on the choice of the shifted edge in (\ref{tilde_L_t_def}), as introduced in \cite{huang2020transition,lee2021higher,huang2022edge}. 
In Section \ref{imm_estimate_section_2} below, we will show in Lemma \ref{time_derivative_bound_lemma} that the fourth order terms in (\ref{term1}) can be expanded into finitely many non-leading terms that can be bounded effectively by Lemma \ref{lemma_lemma_tau_g_ineq}. Such cancellations are highly non-trivial and will be proved with computer aided symbolic computations; see Section \ref{perfect_cancellation_section} for details. 

\smallskip

{\bf 3) Third and fifth order terms}: We claim that all the third and fifth order terms in (\ref{imm_time_derivative_expansion}) with $k=2,4$ are also non-leading terms (in particular they belong to cases \ref{pheq1_dqge1}--\ref{unmatched} of Definition \ref{def_gronwall_term}). 
More precisely, if the term has the prefactor $h_{ab}$ (case \ref{pheq1_dqge1}), then we gain an additional $N^{-1}$ from expanding $h_{ab}$ as in (\ref{111}). If the term comes with the prefactor $\delta_{ab}$ (case \ref{ddge1_dqge1}), one gains an additional $N^{-1}$ from $\delta_{ab}$ when summing over the indices $1\leq a,b\leq n$ (see \eg the argument in (\ref{example}) below). The remaining third (or fifth) order terms are the so-called {\it unmatched terms} (see case (\ref{unmatched}) and Definition \ref{def_unmatch} below), \eg 
$$2 \left(1 - \e^{-t}\right)^{1/2} \frac{N}{q N^{3}} \sum_{v, a, b} \EX \left[G_{va} G_{av} G_{bb} G_{ab}\right],$$
where the index $a$ or $b$ occurs three (or five) times as the row/column indices of the Green function entries. One could gain an additional $N^{-1}$ for such an unmatched term, by iterative expansions via the index $a$ (or $b$) that occurs an odd number of times (see the formal statement in Lemma \ref{prop_unmatched_term_bound} below). Similar iterative expansions were also used in \cite{Schnelli_Xu_2022} for Wigner matrices.

\smallskip

{\bf 4) Sixth and higher order terms}: For all the higher order terms in \eqref{imm_time_derivative_expansion} with $k\geq 5$, we gain enough smallness from the normalized cumulants (\ref{cumulant_assumptions}) using that $q\geq N^{1/6+\bd}$ and $\eta \geq N^{-1 + \epsilon} + N^{2\epsilon}/q^4$ from (\ref{edge_regime}). In fact they are non-leading terms belonging to case \ref{dqge4_dge2} of Definition \ref{def_gronwall_term}. 

\smallskip

To sum up: After expanding the right side of (\ref{imm_time_derivative_expansion}), the leading terms originating from the fourth order cumulant expansions together with the terms from the shifted edge $\widehat{L}_t $ chosen in (\ref{tilde_L_t_def}) cancel precisely, while the remaining non-leading terms can be bounded by $N^{-c\epsilon}\EX[\im m(t,z(t)))]$ using Lemma \ref{lemma_lemma_tau_g_ineq}. This sketches the proof of \eqref{time_derivative_gronwall_bound} and the details will be presented in the next section.

\section{Proof of \eqref{time_derivative_gronwall_bound} in Proposition \ref{imm_bound_random_edge_lemma}}
\label{imm_estimate_section_2}

Following the strategy outlined in the previous section, we now present the formal proof of \eqref{time_derivative_gronwall_bound} in Proposition \ref{imm_bound_random_edge_lemma}. Recall that we have expanded $\dm \EX [m(t, z(t))]$ in (\ref{imm_time_derivative_expansion}).

\subsection{General form of terms in (\ref{imm_time_derivative_expansion})}

  To study the terms on the right side of (\ref{imm_time_derivative_expansion}) computed using the differentiation rules in \eqref{G_h_derivative} and \eqref{G_w_derivative},  we introduce the following general form of averaged products of Green function entries. 
\begin{definition}[General term notation]
	Let $\mathcal{I} := \{v_j\}_{j=1}^m$ be a set of summation indices, and denote its cardinality as $\# \mathcal{I}$. Further, let $d_N$, $d_q$, $d_{\delta}$, and $p_h$ be non-negative integers, and let $\alpha : \mathbb{R}^+ \to [0, 1]$ be a smooth function denoting a time factor. We use the following notation to denote a general term:
	\begin{equation}
		c \, \alpha(t) \frac{N^{d_N}}{q^{d_q}N^{\# \mathcal{I}}}  \sum_{\mathcal{I}} \delta_{ab}^{d_{\delta}} \EX \left[h_{ab}^{p_h} \prod_{i=1}^{n} G_{x_i y_i}\right].
		\label{very_general_term}
	\end{equation} The notation $\sum_{\mathcal{I}}$ means that all indices in $\mathcal{I}$ are iterated from $1$ to $N$. The row index $x_i$ and column index $y_i$ of Green function entries are chosen from the index set $\{v_j\}$. Also, in \eqref{very_general_term} we use the convention $0^0 =1$  with $\delta_{ab}^0$. 
	
	The number of off-diagonal Green function entries in the product $\prod_{i=1}^n G_{x_i y_i}$ in \eqref{very_general_term} is referred to as the degree of the term, usually denoted $d$, i.e.,
	\begin{equation}\label{def_degree}
		d := \# \{1 \leq i \leq n : x_i \neq y_i\}.
	\end{equation} The notation $T$ will often be used to denote a general term on the form \eqref{very_general_term}.
	\label{general_term_notation_def}
\end{definition}

Using the differentiation rules in \eqref{G_h_derivative} and \eqref{G_w_derivative}, we see that all the terms obtained in \eqref{imm_time_derivative_expansion} are of the form (\ref{very_general_term}) with $d_N = 1$. Additionally, thanks to the index $v$ in \eqref{imm_time_derivative_expansion}, all the terms have degree $\geq 2$. In other words, each term contains at least two off-diagonal Green function entries.

\begin{example} 
	One example of a third order term in (\ref{imm_time_derivative_expansion}) with $k=2$ is given by
	\begin{equation}\label{example_unmatch}
		- 2 \e^{-t} \left(1 - \e^{-t}\right)^{1/2} \frac{N}{q N^{3}} \sum_{v, a, b} \EX \left[G_{va} G_{av} G_{bb} G_{ab}\right],
	\end{equation}
	which is of the form described in Definition \ref{general_term_notation_def} with $\mathcal{I} = \{v, a, b\}$, $d_N = 1$, $d_q = 1$, $d_\delta = 0$, $p_h = 0$, $c = -2$, $\alpha(t) = \e^{-t} \left(1 - \e^{-t}\right)^{1/2}$, $n = 4$, $(x_1, y_1) = (v, a)$, $(x_2, y_2) = (a, v)$, $(x_3, y_3) = (b, b)$ and $(x_4, y_4) = (a, b)$. It has degree $d = 3$.
	
	Moreover, if we instead have $d_\delta=1$, we can write
	\begin{equation}\label{example}
			-2 \e^{-t} \left(1 - \e^{-t}\right)^{1/2} \frac{N}{q N^{3}} \sum_{v, a, b} \delta_{ab}\EX \left[G_{va} G_{av} G_{bb} G_{ab}\right]= -2 \e^{-t} \left(1 - \e^{-t}\right)^{1/2} \frac{1}{q N^{2}} \sum_{v, a} \EX \left[G_{va} G_{av} G_{aa} G_{aa}\right].
	\end{equation}
\end{example}

We will often but not always have the additional restrictions $d_{\delta} = p_h = 0$ on the general form \eqref{very_general_term}, which yields a simpler form to manipulate. We remark that a term with $d_{\delta} \geq 1$ or $p_h \geq 1$ is indeed smaller than the same term setting $d_{\delta}= p_h=0$. In the former case $d_{\delta}\geq 1$, this means $a=b$ in the summation, so only $a$ or only $b$ is an effective index. We hence gain an additional $N^{-1}$ summing over $a,b$ to reduce $d_N$ by one; see (\ref{example}) for instance. In the latter case $p_h\geq 1$, we gain additional $N^{-1}$ from further expanding in the centered random variable $h_{ab}$ with variance $N^{-1}$ as explained in (\ref{111}).

A crucial component of this paper are algorithms for symbolic computations used to compute with terms of the form \eqref{very_general_term}. We note that the representation of a term \eqref{very_general_term} with $d_{\delta} = p_h = 0 $ is not unique. 
If one permutes the indices in $\mathcal{I}$, the value of the sum is invariant. So we say the terms obtained by permuting indices are {\it equivalent}; see the formal definition below. Detecting whether two terms are equivalent or not is one of the main algorithmic problems. We describe how this and other implementational issues are handled in Section \ref{implementation_details_section}. 
\begin{definition}\label{def_equiv}
	We say that two terms of the form \eqref{very_general_term} with $d_{\delta} = p_h = 0$ are equivalent if they have the same $d_N$, $d_q$, $\alpha(t)$, and the indices of the first term can be permuted to obtain the second term. 
\end{definition}

\begin{example} [Equivalent terms]
	The terms \begin{equation*}
		\frac{1}{N^3} \sum_{v, a, b} \EX[G_{va}^2G_{aa}G_{bb}^2], 
	\end{equation*} and \begin{equation*}
		\frac{1}{N^3} \sum_{v, a, b} \EX[G_{vb}^2G_{bb}G_{aa}^2],
	\end{equation*} are equivalent since the latter may be obtained from the former by the map $(v, a, b) \mapsto (v, b, a)$. 
\end{example}

\subsubsection{Definition of non-leading terms}
As outlined in Section \ref{imm_estimate_section}, the main step to prove \eqref{time_derivative_gronwall_bound} is to show that the leading terms cancel on the right side of \eqref{imm_time_derivative_expansion}. We collect all the remaining non-leading terms in a set, denoted by $\Tau_G$. 
A substantial subclass of non-leading terms (indeed case 5 of Definition \ref{Tau_G_def} below) are introduced in the following definition:  
 \begin{definition}[Unmatched terms]\label{def_unmatch}
	We call a term of the form \eqref{very_general_term} with $d_{\delta} = p_h = 0$ an unmatched term if there is an unmatched index $v_j \in \mathcal{I}$ such that \begin{equation*}
		\# \{1 \leq i \leq n : x_i = v_j\} + \# \{1 \leq i \leq n : y_i = v_j\}
	\end{equation*} is odd. In other words, a term is unmatched if there is an index occurring an odd number of times in the Green function entries. Furthermore, denote the collection of unmatched terms of degree $\geq d$ by $\Tau_d^{o}$. 
\end{definition}
We remark that the third order terms with $k+1=3$ (in general with any odd $k+1$) on the right side of (\ref{imm_time_derivative_expansion}) are unmatched terms, see \eg (\ref{example_unmatch}).  Note that an unmatched term always has at least two indices occurring an odd number of times, since there are an even number of positions for the indices to appear. Now we are ready to define the set of non-leading terms $\Tau_G$:
\begin{definition} [Non-leading terms]\label{def_gronwall_term}
	Let $\Tau_{G1}$ be the set of terms on the form \eqref{very_general_term}, \textit{i.e.},
	\begin{equation}
		c \, \alpha(t) \frac{N^{d_N}}{q^{d_q}N^{\# \mathcal{I}}}  \sum_{\mathcal{I}} \delta_{ab}^{d_{\delta}} \EX \left[h_{ab}^{p_h} \prod_{i=1}^{n} G_{x_i y_i}\right],
	\end{equation}
 with $d_N \leq 1$ and that additionally satisfy at least one of the following conditions: 
	\begin{enumerate}
		\item $d_q \geq 4$ and $d \geq 2$, 
		\label{dqge4_dge2}
		\item $p_h \geq 1$ and there is an index $v \in \mathcal{I}$ distinct from $a$ and $b$ occurring at least twice in the off-diagonal Green function entries of $\prod_{i=1}^n G_{x_i y_i}$, 
		\label{pheq1_vocc2}
		\item $p_h \geq 1$ and $d_q \geq 1$, 
		\label{pheq1_dqge1}
		\item $d_{\delta} \geq 1$ and $d_q \geq 1$,
		\label{ddge1_dqge1}
		\item $p_h = d_{\delta} = 0$, $d_q \geq 1$ and the term is unmatched (recall Definition \ref{def_unmatch}). %
		\label{unmatched} 
	\end{enumerate} 
Moreover, let $\Tau_{G2}$ be the set of terms of the form  \begin{equation}
	c \alpha(t) \frac{N^{d_N}}{q^{d_q}N^{\# \mathcal{I}}}  \sum_{\mathcal{I}} \EX \left[ \big( x + \ii \eta+\chi(t)\big) \prod_{i=1}^{n} G_{x_i y_i}\right], \qquad d_N \leq 1, ~ d_q \geq 2,
	\label{general_term_chi}
\end{equation}
where $x + \ii \eta+\chi(t)=2-z(t)$, with $z(t)$ given by (\ref{resolvent_t}). Such terms will appear due to the prefactor $2-z(t)$ from using the identity (\ref{ggg}) below.

Finally define the set of non-leading terms by
\begin{equation*}
	\Tau_G := \Tau_{G1} \cup \Tau_{G2}.
\end{equation*} 
With a slight abuse of notation, we will use  $O(\Tau_G)$ to indicate a sum of finitely many non-leading terms in $\Tau_G$, which may vary from line to line.
\label{Tau_G_def}
\end{definition}

\begin{example}
	Two examples of terms in $\Tau_G$ are: \begin{equation*}
		\e^{-t} \frac{N}{q N^2} \sum_{a, b} \EX\left[G_{aa}G_{ab}G_{bb}\right], \quad \text{and,} \quad  \e ^{-t} \left(1 - \e^{-t}\right)^2 \frac{N}{q^4 N^3} \sum_{v, a, b} \EX\left[G_{va}G_{av}G_{aa}^2G_{bb}^3\right].
	\end{equation*} The first term is unmatched since both $a$ and $b$ occurs an odd number of times, and hence it belongs to case \ref{unmatched} of $\Tau_{G1}$ because $d_q = 1$ and $d_N = 1$. The second term belongs to case \ref{dqge4_dge2} of $\Tau_{G1}$ since $d_q = 4$, $d_N = 1$, and the degree is 2. 
\end{example}

	Now we are ready to analyze the expansion of $\dm \EX [m(t, z(t))]$ in (\ref{imm_time_derivative_expansion}) following the sketch outlined in Section \ref{imm_estimate_section}. Computing the right side of (\ref{imm_time_derivative_expansion}) using the differentiation rules in Lemma \ref{differentiation_lemma_imm_case}, the resulting terms are all in the general form (\ref{very_general_term}).  
	
	Recall the second order terms in \eqref{imm_time_derivative_expansion} with $k = 1$ that have been computed in (\ref{oringal}), \eqref{111}, and (\ref{222}). There we observe perfect cancellations among the leading terms, and the remaining terms in (\ref{oringal}) and (\ref{222}) have $p_h\geq 1$. Thanks to the index $v$ in (\ref{imm_time_derivative_expansion}), all these terms have two off-diagonal Green function entries.  Hence the terms in (\ref{oringal}) and (\ref{222}) with $p_h\geq 1$ are non-leading terms belonging to case \ref{pheq1_vocc2} of $\Tau_{G1}$. 
	
	As for the third and fifth order terms ($k = 2,4$)  in \eqref{imm_time_derivative_expansion}, when taking each partial derivative $\partial/\partial h_{ab}$ using \eqref{G_h_derivative}, we split the discussion into three cases: 1)  as in the first line of \eqref{G_h_derivative}, an extra $G$ is added to the Green function of entries, with one $a$ and one $b$ added to the set of the row and column indices; 2) one prefactor $\delta_{ab}$ is added as in the second and fourth line of \eqref{G_h_derivative}; 3) one prefactor $h_{ab}$ is added as in the third line of \eqref{G_h_derivative}. The crucial thing to note is that, if the resulting term from taking the $k$th partial derivative $\partial^{k}/\partial h^k_{ab}$ satisfies $p_h = d_{\delta} = 0$, then 
	the number of $a$'s and $b$'s has the same parity as the number of  derivatives $k$. Hence the resulting terms with $d_{\delta} = p_h = 0$ are unmatched terms  belonging to case \ref{unmatched} of $\Tau_{G1}$, those with $d_\delta \geq 1$ belong to case \ref{ddge1_dqge1} of $\Tau_{G1}$, and the remaining ones with $p_h \geq 1$ belong to case \ref{pheq1_dqge1} of $\Tau_{G1}$.

	In addition, all the sixth and higher order terms in \eqref{imm_time_derivative_expansion} with $k \geq 5$ will belong to case \ref{dqge4_dge2} of  $\Tau_{G1}$, since they all have $d_q \geq 4$ and have degree $\geq 2$ thanks to the index $v$.

	We next focus on the most critical fourth order terms with $k=3$ in (\ref{imm_time_derivative_expansion}). From the differentiation rule \eqref{G_h_derivative}, they are all of the form (\ref{very_general_term}) with $d_N=1$, $d_q = 2$. Terms with the prefactor $h_{ab}$ or $\delta_{ab}$ are clearly non-leading terms belonging to cases \ref{pheq1_dqge1}--\ref{ddge1_dqge1} of $\Tau_{G1}$. Hence the leading terms in \eqref{imm_time_derivative_expansion} are the terms with $p_h = d_{\delta} = 0$, together with the term on the third line of \eqref{imm_time_derivative_expansion}, which is from the deterministic edge shift in (\ref{tilde_L_t_def}).  All of these terms have $\alpha(t) = \e^{-t}(1-\e^{-t})$. More precisely, we have (with potentially renamed summation indices) 
\begin{equation} 
	\begin{split} \dm \EX[m(t, z(t))] & =
		12 \kappa_4  \e^{-t}(1 - \e^{-t}) \frac{N^1}{q^2N^2} \sum_{a,b} \EX \left[G_{ab}^2 \right] \\ 
		& \quad +24 \kappa_4 \e^{-t}(1 - \e^{-t}) \frac{N^1}{q^2N^3} \sum_{a,b,v} \EX \left[G_{ab}G_{av}G_{bv} \right]\\ 
		& \quad +48 \kappa_4 \e^{-t}(1 - \e^{-t}) \frac{N^1}{q^2N^4} \sum_{a,b,v,j} \EX \left[G_{aa}G_{bv}G_{bj}G_{vj} \right] \\ 		
		& \quad +72 \kappa_4 \e^{-t}(1 - \e^{-t}) \frac{N^1}{q^2N^4} \sum_{a,b,v,j} \EX \left[G_{ab}G_{aj}G_{bv}G_{vj} \right] \\ 
        & \quad +24 \kappa_4 \e^{-t}(1 - \e^{-t}) \frac{N^1}{q^2N^4} \sum_{a,b,v,j} \EX \left[G_{aj}^2G_{bv}^2 \right] \\ 		
		& \quad  +12 \kappa_4 \e^{-t}(1 - \e^{-t}) \frac{N^1}{q^2N^3} \sum_{a,b,v} \EX \left[G_{aa}^2G_{bb}G_{bv}^2 \right]\\ 
		& \quad +36 \kappa_4 \e^{-t}(1 - \e^{-t}) \frac{N^1}{q^2N^3} \sum_{a,b,v} \EX \left[G_{aa}G_{ab}G_{av}G_{bb}G_{bv} \right] \\ 
		& \quad +36 \kappa_4 \e^{-t}(1 - \e^{-t}) \frac{N^1}{q^2N^3} \sum_{a,b,v} \EX \left[G_{aa}G_{ab}^2G_{bv}^2 \right] \\ 
		& \quad +12 \kappa_4 \e^{-t}(1 - \e^{-t}) \frac{N^1}{q^2N^3} \sum_{a,b,v} \EX \left[G_{ab}^3G_{av}G_{bv} \right] \\
		& \quad + O(\Tau_G) + O_{\prec}(N^{-D}),
	\end{split} 
	\label{imm_diff_expansion}
\end{equation}
where the notation $O(\Tau_G)$ denotes a sum of finitely many non-leading terms from the set $\Tau_G$. 

The following key lemma states that the above leading terms in \eqref{imm_diff_expansion} cancel precisely up to a collection of finitely many non-leading terms in $\Tau_G$. We will accomplish this step with computer aided symbolic computations, to be explained in Section \ref{perfect_cancellation_section} below. 
\begin{lemma}[Leading terms cancel]
	Assume that $z(t) \in \edge(t)$ and fix any $D>0$. Then leading terms on the right side of (\ref{imm_diff_expansion}) can be expanded into a finite sum of non-leading terms in $\Tau_G$ up to an $O_{\prec}(N^{-D})$-error. Precisely  
	\begin{equation*}
		\dm \EX[m(t, z(t))] = O(\Tau_G) + O_{\prec}(N^{-D}),
	\end{equation*} uniformly for $z(t) \in \edge(t)$ and $t \in [0, 10 \log N]$. Here the notation $O(\Tau_G)$ denotes a sum of finitely many non-leading terms belonging to the set $\Tau_G$ (which is a  different sum than that in (\ref{imm_diff_expansion})).
	\label{time_derivative_bound_lemma}
\end{lemma}

To estimate the non-leading terms in $\Tau_G$, we introduce the following lemma. 
 \begin{lemma} 
	For any term $T \in \Tau_G$, we have 
	\begin{equation}
		\abs{\im T} \leq N^{-\epsilon/4} \EX [\im m] + \frac{N^{\epsilon}}{N},
		\label{lemma_tau_g_ineq}
	\end{equation} 
uniformly for $t \in [0, 10 \log N]$ and $z(t) \in \edge(t)$, for sufficiently large $N$. 
	\label{lemma_lemma_tau_g_ineq}
\end{lemma} 

The proof is postponed to Section \ref{unmatched_terms_section}. Combining Lemmas \ref{time_derivative_bound_lemma} and \ref{lemma_lemma_tau_g_ineq} we have proved \eqref{time_derivative_gronwall_bound} and hence finished the proof of Proposition \ref{imm_bound_random_edge_lemma}.

\begin{remark}
		The assumption $q \geq N^{1/6 + \bd}$ is used to prove the estimate in (\ref{lemma_tau_g_ineq}) along with the restriction $\eta \geq N^{-1+\epsilon}+ N^{2\epsilon}/q^4$ in the regime $\edge(t)$ from \eqref{edge_regime}.  We remark that the restriction of $\eta \geq N^{2\epsilon}/q^4$ yields the additional error term $N^{2/3+\omega}q^{-4}$ in (\ref{eq_main}).
	More precisely, we need this assumption to bound the non-leading terms belonging to case \ref{dqge4_dge2} of $\Tau_{G1}$;
	see arguments used in \eqref{q_bound_used_eq} below. Note that one main class of terms in case \ref{dqge4_dge2} of $\Tau_{G1}$ are the sixth order cumulant terms in (\ref{imm_time_derivative_expansion}), \eg 
	\begin{equation*}
		120 \kappa_6 e^{-t}(1 - e^{-t})^2 \frac{N}{q^4N^3} \sum_{a,b,v} \EX \left[G_{aa}^3G_{bb}^2G_{bv}^2 \right].
	\end{equation*} 
	We expect that by refining the corrections to the spectral edge (\ref{tilde_L_t_def}) as in \cite{huang2022edge,lee2021higher}, these high order terms can also be canceled up to a sufficiently large order.
	\label{lower_q_bound_remark}
\end{remark}

 \subsection{Proof of Lemma \ref{time_derivative_bound_lemma}}
\label{perfect_cancellation_section}

This subsection is devoted to prove Lemma \ref{time_derivative_bound_lemma}. The proof contains two steps: \begin{itemize}
	\item[1)] {\bf Step 1: generate identities.} The first step is to introduce different expansion rules to generate non-trivial identities among the leading terms occurring in \eqref{imm_diff_expansion}, and in general, terms of the form \eqref{very_general_term} with $d_q = 2$, $d_N = 1$, and $d_{\delta} = p_h = 0$. These expansions rules are presented in details as Rules 1--3 below in Section \ref{sec:step1}.
		
	\item[2)]	
		 {\bf Step 2: show cancellation.} The second step is to check whether the leading terms in \eqref{imm_diff_expansion} will be canceled using the identities obtained in the first step. This can be done by encoding the identities into a large linear equation system and checking if a solution exists, as explained in Section \ref{sec:step2} below. 

\end{itemize}
The proof relies on symbolic computer-aided computations. In Section \ref{sec:code} we present details on implementing the algorithms needed in the proof.

\subsubsection{ \bf Step 1: generate identities.} \label{sec:step1}
We now describe how to generate non-trivial identities among the leading terms in (\ref{imm_diff_expansion}) with the common factor $\kappa_4$. Note that these leading terms are (up to constants) all in the following general form (\cf (\ref{very_general_term}) with $d_q = 2$, $d_N = 1$, $d_{\delta} = p_h = 0$):
\begin{equation}
	c\alpha(t)\frac{N}{q^{2}N^{\# \mathcal{I}}} \sum_{\mathcal{I}} \EX \left[\prod_{i=1}^{n} G_{x_i y_i}\right], \qquad  d=\# \{i : x_i \neq y_i\}\geq 2,
	\label{start_term}
\end{equation} 
with $c=1$, $\alpha(t) = \e^{-t}(1 - \e^{-t})$. The identities among the leading terms of the general form (\ref{start_term}) will be obtained using the expansion rules explained below. All of the rules use the resolvent identity:
\begin{equation}
	z(t)G_{ab} = \sum_{i} h_{ai}(t) G_{ib} - \delta_{ab}.
	\label{resolvent_identity}
\end{equation}

\textbf{Rule 1. Expanding an off-diagonal Green function entry} \\ 
Given a term in (\ref{start_term}), without loss of generality, we may assume the first Green function entry $G_{x_1y_1}$ is off-diagonal with $x_1=a,y_1=b$. Using $z(t) = 2 + (z(t) - 2)$, we rewrite \eqref{resolvent_identity} as 
\begin{align}\label{ggg}
	G_{ab} = \frac{1}{2} \sum_{j} h_{aj}(t) G_{jb} - \frac{\delta_{ab}}{2} + \frac{2 - z(t)}{2} G_{ab}.
\end{align}
Replacing the first entry $G_{ab}$ by (\ref{ggg}), the resulting terms from  $-\delta_{ab}/2$ and 
\begin{equation}\label{zzz}
	\frac{2-z(t)}{2} = - \frac{3 \kappa_4 (1 - e^{-t})^2}{q^2} - \frac{ x + \ii \eta + \chi(t)}{2},
\end{equation} 
are clearly non-leading terms in $\Tau_G$. In particular, the term with $-\delta_{ab}/2$ belongs to case \ref{ddge1_dqge1} of $\Tau_{G1}$, while the terms from (\ref{zzz}) belong to case \ref{dqge4_dge2} of $\Tau_{G1}$ and $\Tau_{G2}$, respectively. Hence we obtain
\begin{equation}
	\label{expansion_off_diagonal0}
	\begin{split}
		&c\alpha(t) \frac{N}{q^{2}N^{\# \mathcal{I}}} \sum_{\mathcal{I}} \EX \left[ G_{ab}\prod_{i=2}^{n} G_{x_i y_i}\right] \\
		& =	\frac{c}{2}  \alpha(t) \frac{N}{q^{2}N^{\# \mathcal{I}}} \sum_{\mathcal{I} \cup \{j\}} \EX \left[h_{a j}(t) G_{j b}\prod_{i=2}^{n} G_{x_i y_i}\right]+O(\Tau_{G}),
	\end{split}	
\end{equation}
with $j$ a fresh index outside the index set $\mathcal{I}$, where $O(\Tau_{G})$ denotes a sum of the non-leading terms obtained from the last two terms in (\ref{ggg}).

We next perform cumulant expansions on the terms with $h_{aj}(t)$ on the right side of \eqref{expansion_off_diagonal0}. It is not hard to check that the third and higher order terms in the cumulant expansions are non-leading terms belonging to $\Tau_{G}$, since the term on the left side of (\ref{expansion_off_diagonal0}) has degree $\geq 2$. Hence we obtain from (\ref{expansion_off_diagonal0}) that
\begin{equation}
	\begin{split}
		\label{expansion_off_diagonal}
		&c\alpha(t) \frac{N}{q^{2}N^{\# \mathcal{I}}} \sum_{\mathcal{I}} \EX \left[ G_{ab}\prod_{i=2}^{n} G_{x_i y_i}\right] \\
		& = \frac{c}{2} \alpha(t)\sqrt{1 - e^{-t}} \frac{N}{q^{2}N^{\# \mathcal{I}+1}} \sum_{\mathcal{I} \cup \{j\}} \EX \left[ \frac{\partial \big(G_{j b}\prod_{i=2}^{n} G_{x_i y_i}\big)}{\partial h_{a j}}  \right]\\
		& \quad +\frac{c}{2} \alpha(t)\e^{-\frac{t}{2}} \frac{N}{q^{2}N^{\# \mathcal{I}+1}} \sum_{\mathcal{I} \cup \{j\}}  \EX \left[  \frac{\partial \big(G_{j b}\prod_{i=2}^{n} G_{x_i y_i}\big)}{\partial w_{a j}}\right]+O(\Tau_{G})+O_{\prec}(N^{-D}),
	\end{split}
\end{equation}
where $O(\Tau_{G})$ is a collection of the non-leading terms from (\ref{expansion_off_diagonal0}) and from higher order cumulant expansions, the last error term $O_{\prec}(N^{-D})$ for any large $D>0$ is from truncating the cumulant expansions at sufficiently high order using \eqref{cumulant_remainder_small_eq_imm_case}. 

We recall the differentiation rules (\ref{G_h_derivative})-(\ref{G_w_derivative}) and use them to compute (\ref{expansion_off_diagonal}). Note that terms with the prefactor $\delta_{ab}$ or $h_{ab}$ are non-leading terms belonging to case \ref{pheq1_dqge1} and \ref{ddge1_dqge1} of $\Tau_{G_1}$. By direct computations and grouping similar terms, we have 
\begin{equation}
	\begin{split}
		\label{expansion_off_diagonal2}
		&c\alpha(t) \frac{N}{q^{2}N^{\# \mathcal{I}}} \sum_{\mathcal{I}} \EX \left[ G_{ab}\prod_{i=2}^{n} G_{x_i y_i}\right] \\
		& = \frac{c}{2} \alpha(t) \frac{N}{q^{2}N^{\# \mathcal{I}+1}} \sum_{\mathcal{I} \cup \{j\}} \EX \left[ \frac{\ppartial \big(G_{j b}\prod_{i=2}^{n} G_{x_i y_i}\big)}{\ppartial h_{a j}}  \right] + O(\Tau_G) + O_{\prec}(N^{-D}),
	\end{split}
\end{equation}
where we define the new ``differentiation rule'' below to keep leading terms only
\begin{align}
		\frac{\ppartial G_{ij}}{\ppartial h_{ab}} & =  - 
		\Big(G_{ia} G_{bj} + G_{ib} G_{aj}\Big), \label{new_rules_1}
\end{align}
and include in $O(\Tau_{G})$ the remaining non-leading terms from the original rules \eqref{G_h_derivative} and \eqref{G_w_derivative}. It is straightforward to check that the leading terms obtained from the new rule \eqref{new_rules_1} must have degree $\geq 2$. Hence we have obtained a linear combination of leading terms of the form \eqref{start_term} with degree $\geq 2$, plus a collection of finitely many non-leading terms belonging to $\Tau_{G}$ from Definition \ref{def_gronwall_term}. 

\medskip

\textbf{Rule 1'.}
Note that $G$ is symmetric and in particular $G_{ab} = G_{ba}$ so we can also rewrite \eqref{resolvent_identity} as \begin{equation*}
	G_{ab} = \frac{1}{2} \sum_{j} h_{bj}(t) G_{ja} - \frac{\delta_{ab}}{2} + \frac{2 - z(t)}{2} G_{ab}.
\end{equation*} Using this identity for $G_{ab}$ and repeating the arguments as in (\ref{expansion_off_diagonal0})-(\ref{expansion_off_diagonal2}), we get 
\begin{equation}
	\begin{split}
			&c\alpha(t) \frac{N}{q^{2}N^{\# \mathcal{I}}} \sum_{\mathcal{I}} \EX \left[G_{ab} \prod_{i=2}^{n} G_{x_i y_i}\right] \\
		=& \frac{c}{2} \alpha(t) \frac{N}{q^{2}N^{\# \mathcal{I}+1}} \sum_{\mathcal{I} \cup \{j\}} \EX \left[ \frac{\ppartial \big(G_{j a}\prod_{i=2}^{n} G_{x_i y_i}\big)}{\ppartial h_{b j}}  \right]
		+O(\Tau_{G})+O_{\prec}(N^{-D}),
		\label{expansion_off_diagonal_reflected}
	\end{split}
\end{equation}
using the new differentiation rule in (\ref{new_rules_1}). It might seem like \eqref{expansion_off_diagonal_reflected} gives exactly the same identity as \eqref{expansion_off_diagonal2}, but if for example the index $a$ occurs four times in the product of Green function entries and the index $b$ appears only twice, then we get a non-trivial identity by comparing \eqref{expansion_off_diagonal2} and \eqref{expansion_off_diagonal_reflected}. We also note that the expansion in Rule 1 (or Rule 1') is not unique. One could expand any off-diagonal Green function entry $G_{x_i y_i}$ in \eqref{start_term}.

\bigskip

\textbf{Rule 2. Inserting diagonal Green function entries} \\ 
Given any term in (\ref{start_term}), one could insert a factor 1 by the following identity:
\begin{equation}\label{relation}
	1=\frac{1}{N} \sum_{k, j} h_{jk}(t) G_{kj} - \frac{2}{N} \sum_{j} G_{jj} + \frac{2 - z(t)}{N} \sum_{j} G_{jj},
\end{equation} 
with fresh indices $j,k$. As explained in the first rule, the resulting terms obtained from the last term in (\ref{relation}) with the prefactor $2-z(t)$ are non-leading terms in $\Tau_G$.  Inserting a factor $1$ into the term in \eqref{start_term}  and performing cumulant expansions with respect to $h_{jk}(t)$ as in the first rule, we obtain
 \begin{equation}
	\begin{split}
		& 	c\alpha(t)	\frac{N}{q^{2}N^{\# \mathcal{I}}} \sum_{\mathcal{I}} \EX \left[1 \cdot  \prod_{i=1}^{n} G_{x_i y_i}\right] \\
	= &	c\alpha(t)  \frac{N}{q^{2}N^{\# \mathcal{I} + 2}} \sum_{\mathcal{I} \cup \{j, k\}}  \EX \left[\frac{\ppartial \big(G_{kj}\prod_{i=1}^{n} G_{x_i y_i}\big)}{\ppartial h_{jk}}\right] \\
		& \quad \, - 	2c\alpha(t)\frac{N}{q^{2}N^{\# \mathcal{I} + 1}} \sum_{\mathcal{I} \cup \{j\}}  \EX \left[G_{jj} \prod_{i=1}^{n} G_{x_i y_i}\right]+ O(\Tau_G)+O_{\prec}(N^{-D}),
	\end{split}
	\label{expansion_off_1}
\end{equation} 
 with the new differentiation rule \eqref{new_rules_1}, and where $j$ and $k$ are fresh indices outside the index set $\mathcal{I}$. The term on the last line is clearly of the form \eqref{start_term}. The analysis of the first two lines on the right side above is quite similar to that of \eqref{expansion_off_diagonal2}, so we omit it. Hence we have obtained a linear combination of leading terms of the form  \eqref{start_term} with degrees $\geq 2$ as in the first rule, while the other terms are non-leading belonging to $\Tau_G$.

\bigskip
\textbf{Rule 3. Inserting diagonal Green function entries with existing index} \\ Rule 3 is similar to Rule 2. Given any term in (\ref{start_term}), we may assume the index $a$ was used as a summation index in $\mathcal{I}$.   
Instead of (\ref{relation}), we use 
\begin{equation}\label{relation2}
	1 = \sum_{j} h_{aj}(t) G_{j a} - 2 G_{aa} + (2-z(t))G_{aa} ,
\end{equation} 
where $a$ is an existing index in $\mathcal{I}$ and $j$ is a fresh one. As in the second rule, we insert a factor $1$ into the term in \eqref{start_term} and apply cumulant expansions with respect to $h_{aj}(t)$. That is,
\begin{equation}
	\begin{split}
		& 	c\alpha(t)\frac{N}{q^{2}N^{\# \mathcal{I}}} \sum_{\mathcal{I}}  \EX \left[1 \cdot  \prod_{i=1}^{n} G_{x_i y_i}\right] \\
		& = 	c\alpha(t) \frac{N}{q^{2}N^{\# \mathcal{I}}} \sum_{\mathcal{I} \cup \{j\}}  \EX \left[ \frac{\ppartial \big(G_{a j}\prod_{i=1}^{n} G_{x_i y_i}\big)}{\ppartial h_{a j}}\right] \\
		& \quad \, - 	2 c\alpha(t) \frac{N}{q^{2}N^{\# \mathcal{I}}} \sum_{\mathcal{I}}  \EX \left[G_{aa} \prod_{i=1}^{n} G_{x_i y_i}\right]+O(\Tau_{G})+O_{\prec}(N^{-D}),
	\end{split}
	\label{expansion_off_1_reused_index}
\end{equation} 
with $a\in \mathcal{I}$. The analysis of the terms in \eqref{expansion_off_1_reused_index} is very similar to that of \eqref{expansion_off_1}, so we omit it.

\subsubsection{ \bf  Step 2: show cancellations.\nc}\label{sec:step2} 
Before we show the method to check the cancellations in (\ref{imm_diff_expansion}), we first introduce the concept of type-0 terms, type-A terms and type-AB terms, 
that were introduced in \cite{Schnelli_Xu_2022}.

\begin{definition}\label{def_type}	
	A type-0 term is a term of the form \eqref{start_term} where every index in $ \mathcal{I}$ occurs exactly twice in the product $\prod_{i=1}^{n}G_{x_i y_i}$ as the row or column index of Green function entries. 
	
	A type-A term is also a term of the form \eqref{start_term}, but where one index, which we usually call $a \in \mathcal{I}$ occurs exactly four times and all the other indices occur exactly twice in the product $\prod_{i=1}^{n}G_{x_i y_i}$. 
	
	Finally, a type-AB term is a term of the form \eqref{start_term} where there are two indices, usually called $a, b \in \mathcal{I}$ that occur exactly four times in the product $\prod_{i=1}^{n}G_{x_i y_i}$, and all other indices occur exactly twice. 
\end{definition}

 The simplest terms are type-0 terms. Since every index is occurring exactly twice, every type-0 term can be written on the form $c\alpha(t) \frac{N}{q^2N^{\sum p_i}} \EX \left[\prod_{i = 1}^{n} \Tr G^{p_i}\right]$. Then they are completely characterized by the finite sequence $(p_1, p_2, \ldots, p_n)$ (which can be assumed to be ordered). 
Note that Rules 1 and 2 both express a term of a certain type as a linear combination of terms of the same type, \eg we can use them to express a type-0 term as a linear combination of type-0 terms. 

But the problem is that the expansion of $\dm \EX[m(t, z(t))]$ in \eqref{imm_diff_expansion} contains both type-0 and type-AB terms. So we need a way to transform type-0 terms into type-AB terms to see cancellations. This is the purpose of introducing an additional Rule 3 using the existing indices $a$ and $b$ for expansions.
 If the left side of \eqref{expansion_off_1_reused_index} is a type-0 term, then the terms on the right will all be type-A terms. Similarly if we expand a type-A term via another existing index, say $b$, using Rule 3 again, then we obtain a linear combination of type-AB terms. Hence by using Rule 3 several times we can express any type-0 term as a linear combination of type-AB terms.

\bigskip

 We are now ready to show the cancellations in \eqref{imm_diff_expansion}.
The idea is to generate enough identities between type-0, type-A and type-AB terms through the expansion rules 1--3 presented above, and then encode the problem of checking the cancellations among the leading terms in \eqref{imm_diff_expansion} into the problem of solving a large linear equation system.

It works as follows.  Starting with a term of the form \eqref{start_term} with $c = 1$, $\alpha(t) = \e^{-t}(1 - \e^{-t})$, denoted by $T_{\text{start}}$, we may use any of the expansion rules 1--3 introduced in Section \ref{sec:step1} to write the identity in the following form:
\begin{equation}
	T_{\text{start}} = \sum_{k=1}^{K} c_k T_k + O(\Tau_G) + O_{\prec}\left(N^{-D}\right),
	\label{identity_generation_1}
\end{equation} 
for some $K\in \mathbb{N}$, with $c_k \in \mathbb{Q}$ being rational numbers, and the $T_k$s are of the form \eqref{start_term} with $c= 1$, $\alpha(t) = e^{-t}(1 - e^{-t})$ and $d\geq 2$. We remark that the order of the expansion terms $T_k$ $(1 \leq k \leq K)$ is irrelevant. Here $O(\Tau_G)$ denotes a finite sum of non-leading terms from the set $\Tau_G$. We will not track these non-leading terms carefully in the following.
We now move $T_{\text{start}}$ to the right side in \eqref{identity_generation_1} and obtain 
\begin{equation}
	0 = -T_{\text{start}} + \sum_{k=1}^{K} c_k T_k + O(\Tau_G) + O_{\prec}\left(N^{-D}\right).
	\label{zero_identity}
\end{equation} 

Next we aim to expand sufficiently many starting terms as in (\ref{zero_identity}).
The set of starting terms is denoted by $\Tau^{\text{start}}$. Then we  choose $\Tau^{\text{start}}:=\Tau_4^{\text{start}}$, with $\Tau_n^{\text{start}} (n\in \mathbb{N})$ introduced in Definition \ref{def_start} below. For each term $T_{\text{start}} \in \Tau^{\text{start}}$ we use the mechanism introduced below Definition \ref{def_start} step by step to generate a system of identities on the form \eqref{zero_identity}, \ie
\begin{equation}
	0 = -T^{(m)}_{\text{start}} + \sum_{k=1}^{K_m} c^{(m)}_k T^{(m)}_k + O(\Tau_G) + O_{\prec}\left(N^{-D}\right), \qquad 1\leq m\leq M,
	\label{all_identity}
\end{equation} 
with $c^{(m)}_k$ being rational numbers, where $M\in \mathbb{N}$ is the (finite) number of generated identities. We remark that the order of these identities is irrelevant.
From the list of identities \eqref{all_identity} we define the set of basis terms, denoted by $\Tau^{\text{basis}}$, as the set containing the starting terms  $T^{(m)}_{\text{start}}\in \Tau^{\text{start}}$  and the leading expansion terms $\{T^{(m)}_k\}_{k=1}^{K_m}$ in (\ref{all_identity}) for all $1\leq m\leq M$. The terms in $\Tau^{\text{basis}}$ are considered unique only up to equivalence (recalling the definition of equivalent terms from Definition \ref{def_equiv} 
). Again the order of the basis terms in $\Tau^{\text{basis}}=\{T_l\}_{l=1}^{L}$ with $L:=\#\Tau^{\text{basis}}$ is irrelevant.
Numerical details about (\ref{all_identity}) can be found in Section \ref{sec:detail} below. In particular, to help the reader learn the mechanism to derive (\ref{list_of_identities}), we write out one example of the identities obtained above in Appendix \ref{details_about_imm_system} and we also list the first couple basis terms that we obtained in $\Tau^{\text{basis}}$.

We encode the list of identities obtained above into the following list of linear equations:
\begin{equation}
	\left(0 = \sum_{l=1}^{L} \widetilde{c}_{ml} T_{l} + O(\Tau_G) +O_{\prec}\left(N^{-D}\right) \right)_{m=1}^{M}, \qquad \Tau^{\text{basis}}=\{T_l\}_{l=1}^{L},
	\label{list_of_identities}
\end{equation}
where $\widetilde{c}_{ml}$ are rational numbers, and $\Tau^{\text{basis}}$ is the set of basis terms (up to equivalence).

Now we proceed with using  (\ref{list_of_identities}) to verify the cancellations among the leading terms in (\ref{imm_diff_expansion}), that are also basis terms in $\Tau^{\text{basis}}$. The corresponding positions of these terms appearing in the list $\Tau^{\text{basis}}=(T_l)_{l=1}^{L}$ are denoted by $\mathcal{L}_0\subset \{1,2,\ldots,L\}$. Then we rewrite \eqref{imm_diff_expansion} as
 \begin{equation}
	\dm \EX[m(t, z(t))] = \kappa_4 \sum_{l \in \mathcal{L}_0} b_l T_l + O(\Tau_G) + O_{\prec}(N^{-D}),
	\label{imm_expansion_leading_terms_notation}
\end{equation} 
with $b_l \in \mathbb{Q}~(l\in \mathcal{L}_0)$ being the rational coefficients in \eqref{imm_diff_expansion}. 
Our goal is to show that there exists $(x_{m})_{m = 1}^{M} \in \mathbb{Q}^M$ such that
\begin{equation}
	\sum_{l \in \mathcal{L}_0} b_l T_l = \sum_{m = 1}^{M} x_{m} \left( \sum_{l=1}^{L} \widetilde{c}_{ml} T_l \right), \qquad \Tau^{\text{basis}}=\{T_l\}_{l=1}^{L}.
	\label{goal_equation_system} 
\end{equation} 
holds. Such $(x_{m})_{m = 1}^{M} \in \mathbb{Q}^M$ can be found by solving the following linear system:
\begin{equation}
	\left(\sum_{m = 1}^{M}  \widetilde{c}_{ml} x_{m} = \widetilde{b}_{l}\right)_{l=1}^{L}, \qquad   \widetilde{b}_l:=\mathbf{1}_{l\in \mathcal{L}_0} b_l.
	\label{goal_linear_eq_sys}
\end{equation} 
The verification of (\ref{goal_linear_eq_sys}) is done symbolically with the aid of a computer, as explained in Section \ref{sec:solution}. We hence obtain from  \eqref{imm_expansion_leading_terms_notation}, (\ref{goal_equation_system}), and \eqref{list_of_identities} that 
\begin{equation}
	\begin{split}
		\label{final_step}
		\dm \EX[m(t, z(t))]=&\kappa_4 \sum_{l \in \mathcal{L}_0} b_l T_l + O(\Tau_G) + O_{\prec}(N^{-D}) \\ =& \kappa_4 \sum_{m = 1}^{M} x_{m} \left(O(\Tau_G) +O_{\prec}\left(N^{-D}\right)\right)\\
		=& O(\Tau_G) + O_{\prec}(N^{-D}).
	\end{split}
\end{equation}

So the proof of Lemma \ref{time_derivative_bound_lemma} is finished.

\subsubsection{\bf Derivation of (\ref{list_of_identities}) and solution of (\ref{goal_linear_eq_sys}).\nc }\label{sec:detail}
In this section, we define the set of starting terms $\Tau^{\text{start}}=\Tau_4^{\text{start}}$ as defined below, and explain the general mechanism used to generate the identities listed in (\ref{list_of_identities}) and the set of basis terms $\Tau^{\text{basis}}$.

\begin{definition}\label{def_start}
	Let $n$ be a positive integer and recall Definition \ref{def_type}. We define $\Tau_n^{\text{start}}$ as the set of terms consisting of all type-0, type-A, and type-AB terms of the form \eqref{start_term} with $c = 1$, $\alpha(t) = \e^{-t}(1 - \e^{-t})$, degree $d \geq 2$, and the number of summation indices $\# \mathcal{I} \leq n$, up to equivalence (see Definition \ref{def_equiv}). Without loss of generality, we assume the indices $a$ and $b$ will be the ones occurring more than twice in the product of Green function entries.
\end{definition}

We choose $\Tau^{\text{start}}=\Tau_4^{\text{start}}$, and start with $\Tau^{\text{basis}} = \emptyset$. For each term $T_{\text{start}} \in \Tau^{\text{start}}$, we apply the following steps to generate the identities listed in \eqref{list_of_identities} and add the corresponding basis terms to $\Tau^{\text{basis}}$.

{\bf Mechanism to generate identities from each $T_{\text{start}} \in \Tau^{\text{start}}$:}
\begin{enumerate}
	\item For each off-diagonal Green function entry of $T_{\text{start}}$, we apply Rule 1 and 1' to add two non-trivial identities as in \eqref{zero_identity} into the list \eqref{list_of_identities} and also add the corresponding basis terms (up to equivalence) to $\Tau^{\text{basis}}$.
	
	\item If $T_{\text{start}}$ satisfies $\# \mathcal{I} \leq 3$, then we use Rule 2 on $T_{\text{start}}$ to add one identity to the list \eqref{list_of_identities} as well as adding new basis terms (up to equivalence) to $\Tau^{\text{basis}}$.
	\item If the index $a$ occurs less than four times as the row or column indices of the product of Green function entries in $T_{\text{start}}$, we use Rule 3 with the existing index $a$ to add one identity to the list \eqref{list_of_identities} and add new basis terms to $\Tau^{\text{basis}}$. Moreover, if the index $b$ occurs less than four times in $T_{\text{start}}$, then we use Rule 3 with the existing index $b$ to add one more identity to \eqref{list_of_identities} and add new basis terms to $\Tau^{\text{basis}}$.
\end{enumerate}

Below in Section \ref{sec:basis} we describe in detail how we in practice add the basis terms to $\Tau^{\text{basis}}$.

Using the above procedure to generate identities with starting terms in $\Tau_4^{\text{start}}$ we get $M=138$ unique identities in \eqref{list_of_identities}, after removing obvious duplicates. Additionally, we get $L = 110$ unique basis terms in $\Tau^{\mathrm{basis}}$ up to equivalence. Consequently, the coefficient matrix $\widetilde{C}:=\left(\widetilde{c}_{ml}\right)_{1\leq m\leq M,1\leq l\leq L}$ of the system in \eqref{goal_linear_eq_sys} has the dimensions $110 \times 138$. The solution to the system, $(x_{m})_{m=1}^{M}$ contains 62 non-zero entries, meaning that in total we use 62 identities to see the cancellation of the leading terms in the expansion of $\frac{\di}{\di t} \EX[m (t, z(t))]$ in \eqref{imm_diff_expansion}. The rank of the coefficient matrix $\widetilde{C}$ is 98, so among the 138 identities we started with, 40 of them turned out to be superfluous. However, the rank is comparatively of the same size as the number of basis terms, \ie 98 is quite close to 110. This suggests that there are many non-trivial relations between the basis terms, so observing the cancellations between them requires effort. Additionally, the solution $(x_{m})_{m=1}^{M}$ we find consists only of integers. The entire linear system, its solution, and the set of basis terms, $\Tau^{\mathrm{basis}}$ are available in the accompanying GitHub repo \cite{githubrepo}.

\begin{remark}
Starting with $\Tau_4^{\text{start}}$ is non-essential. One could start with any $\Tau_n^{\text{start}}~(n\geq 4)$ to generate identities as explained above. We want as little redundancy as possible in our identities, which means that we aim to choose $n$ as small as possible, but still large enough to generate enough identities to see the cancellations. And it turns out that $n = 4$ was the smallest $n$ for which we could see the cancellation. \end{remark}

\subsection{Proof of Lemma \ref{lemma_lemma_tau_g_ineq}}
\label{unmatched_terms_section} 

This subsection is devoted to proving Lemma \ref{lemma_lemma_tau_g_ineq}, which gives upper bounds on the imaginary parts of non-leading terms in $\Tau_G$ from Definition \ref{def_gronwall_term}. 

 To bound terms of the form \eqref{very_general_term} we will need the following entrywise local law estimates for $G(t,z)$, uniformly for $t \in [0, 10 \log N]$ and $z \in S_0$:
	\begin{equation}\label{time_dependent_local_law}
	\sup_{t \in [0, 10 \log N],z \in S_0}\Big\{ \max_{1 \leq i,j \leq N} \big| (G(t, z))_{ij} - \delta_{ij}m_{sc}(z) \big|\Big\}\prec \frac{1}{q} + \sqrt{\frac{\im m_{sc}(z)}{N \eta}} + \frac{1}{N\eta} = \Psi(z).
	\end{equation}  
\begin{proof}[Proof of (\ref{time_dependent_local_law})]
	Recall $H(t)$ in (\ref{H_dynamics_new}). It is easy to check that $H(t)$ satisfies conditions 1--3 in Assumption \ref{H_assumptions} with the same constant $C$ in \eqref{moment_assumptions} for all $t \in [0, 10 \log N]$.  Then from Theorem \ref{entrywise_local_law}, for each fixed $t\in [0,10\log N]$, the local law estimates (\ref{time_dependent_local_law}) hold with probability bigger than $1-N^{-\Gamma}$ for any large $\Gamma>100$. We next choose a uniform mesh of $[0, 10 \log N]$ containing $N^{10}$ grid points such that the local law estimates (\ref{time_dependent_local_law}) hold for all the grid points with probability bigger than $1-N^{-\Gamma/2}$. 
	Together with the following (stochastic) Lipschitz  continuity of $H(t)$ in time, \ie $\|H(t)-H(s)\|_2 \prec N^2 |t-s|^{1/2}$, we have extended (\ref{time_dependent_local_law}) to all $t\in [0,10 \log N] $ simultaneously.
	  \end{proof}

Recalling from (\ref{edge_regime}), we have $\edge(t) \subset S_0 \cap \{E + \ii \eta : \eta \geq N^{-1 + \epsilon}\}$, with high probability. As an consequence of (\ref{time_dependent_local_law}) and  the boundedness of $\im m_{sc}(z)$ in \eqref{immsc_asypmtotics}, we have for $0 < \epsilon < \min\{1/12, \bd\}$ that
 \begin{equation}\label{entries_small}
 	\sup_{t \in [0, 10 \log N],z \in \edge(t)}\Big\{ \max_{1 \leq i,j \leq N} \big|(G(t, z))_{ij}-\delta_{ij}m_{sc}(z)\big| \Big\}\prec \Psi(z) \leq 3 N^{-\epsilon/2}.
 \end{equation} 
In the rest of this section we will often omit the argument when writing $\Psi(z)$ and will just write $\Psi$.

For any term of the form \eqref{very_general_term} with degree $d\geq 0$ defined in (\ref{def_degree}), we can bound it naively using the entrywise local law estimates in \eqref{time_dependent_local_law} and that $\abs{h_{ab}} \prec 1/q$, \ie
 \begin{equation}
	\abs{T} = \abs{c \, \alpha(t) \frac{N^{d_N}}{q^{d_q}N^{\# \mathcal{I}}}  \sum_{\mathcal{I}} \delta_{ab}^{d_{\delta}} \EX \left[h_{ab}^{p_h} \prod_{i=1}^{n} G_{x_i y_i}\right]}  \prec \frac{N^{d_N}}{q^{d_q+p_h}} \left(\Psi^d + \frac{1}{N}\right),
	\label{naive_local_law_bound}
\end{equation} 
with non-negative integers $d_N, d_q, p_h$, where the error $1/N$ comes from the diagonal cases when the values of two different indices in $\mathcal{I}$ coincide in the summations. 

Note that for terms with low degree $d \geq 0$, the above estimates using (\ref{entries_small}) are not enough. In the following lemma, we obtain an upper bound for the imaginary part $|\im T|$ using the Ward identity (\ref{ward_identity}). The proof will be postponed till Appendix \ref{sec:proof_unmatched}. 
\begin{lemma}
	Fix any small $\xi>0$.	Let $T$ be a term in the general form of \eqref{very_general_term}. Then we have \begin{equation}
		\abs{\im T} \leq N^{\xi}\frac{N^{d_N}}{q^{d_q}}\left(\EX \left[\im m \right] + \frac{1}{N}\right),
		\label{imm_lemma_eq}
	\end{equation} 
with non-negative integers $d_N, d_q$ from \eqref{very_general_term}, uniformly in $z(t) \in \edge(t)$ and $t \in [0, 10 \log N]$. 
	\label{lemma_general_imm_bound}
\end{lemma}

One major difficulty in showing Lemma \ref{lemma_lemma_tau_g_ineq} is to improve Lemma \ref{lemma_general_imm_bound} for unmatched terms from Definition \ref{def_unmatch}. 	Comparing with Lemma \ref{lemma_general_imm_bound}, we aim to gain an additional $N^{-1}$ for the unmatched terms. The proof follows the strategy used to handle unmatched terms in \cite[Proposition 4.3]{Schnelli_Xu_2022} and is adapted to the setting of sparse matrices in Appendix \ref{sec:proof_unmatched}. 

	\begin{lemma}
		Fix any small $\xi>0$. Then any unmatched term $T$ from Definition \ref{def_unmatch} in the general form of \eqref{very_general_term} satisfies 
		\begin{equation}
			\abs{\im T} \leq N^{\xi} \frac{N^{d_N}}{q^{d_q}} \left(\frac{\EX [\im m]}{N} + \frac{1}{N^2} \right),
			\label{unmatched_improved_bound}
		\end{equation} 
	with non-negative integers $d_N, d_q$ from \eqref{very_general_term},
	uniformly in $z(t) \in \edge(t)$ and $t \in [0, 10 \log N]$.
		\label{prop_unmatched_term_bound}
	\end{lemma}

Armed with Lemma \ref{lemma_general_imm_bound} and Lemma \ref{prop_unmatched_term_bound}, we are ready to prove Lemma \ref{lemma_lemma_tau_g_ineq}.

\begin{proof}[Proof of Lemma\ref{lemma_lemma_tau_g_ineq}]

	We will treat each of the possible conditions on $T \in \Tau_G$ separately. We begin by assuming that $T \in \Tau_{G1}$ and treat cases \ref{dqge4_dge2}--\ref{unmatched} one by one. 
	\begin{enumerate}[label=Case \arabic*:]
		\item 
		Using the Ward identity, power counting and assuming that $G_{x_1 y_1}$ and $G_{x_2 y_2}$ are off-diagonal, we can bound the imaginary part of the following averaged product of Green function entries as
		\begin{equation}
			\begin{split}
				& \abs{\im \frac{c}{N^{\# \mathcal{I}}} \sum_{\mathcal{I}} \EX \left[\prod_{i=1}^{n} G_{x_i y_i}\right]} \\
				& \prec \frac{1}{N^{\# \mathcal{I}}} \sum_{\mathcal{I}} \EX \left[\left(\abs{G_{x_1 y_1}}^2 + \abs{G_{x_2 y_2}}^2\right) \Psi^{d-2} \right] + \frac{\EX[\im m]}{N} + \frac{1}{N^2} \\
				& = \frac{\Psi^{d-2}}{N \eta} \EX \left[\im m \right] + \frac{\EX[\im m]}{N} + \frac{1}{N^2}. \\
				\label{ward_identity_bound}	
			\end{split}
		\end{equation}
		For details on the first step in \eqref{ward_identity_bound}, see the proof of Lemma \ref{prop_unmatched_term_bound} in the Appendix. Using \eqref{ward_identity_bound}, the bound on $\Psi$ in \eqref{entries_small}, $\delta_{ab} \leq 1$, and the fact that $\abs{h_{ab}} \prec q^{-1}$, any term $T$ of the form \eqref{very_general_term} with degree $d \geq 2$ can be bounded by
		\begin{equation}
			\abs{\im T} \leq C N^{\epsilon/2} \frac{N^{d_N}}{q^{d_q + p_h}}\left( \frac{N^{-(d-2) \epsilon/2 }}{N \eta} \EX[\im m] + \frac{\EX[\im m]}{N} + \frac{1}{N^2}\right).
			\label{very_general_term_naive_bound}
		\end{equation} 
		
		Since $z(t) \in \edge(t)$ we have that $\eta \geq N^{2 \epsilon}/q^{4}$, which implies that $1/(N\eta) \leq q^4/N^{1 + 2 \epsilon}$. Plugging this bound into \eqref{very_general_term_naive_bound} we obtain \begin{equation}
			\begin{split}
				\abs{\im T} & \leq 
				C N^{\epsilon / 2} \frac{N^{d_N}}{q^{d_q + p_h}} \left( \left(\frac{q^4 N^{- (d-2)\epsilon/2}}{N^{1 + 2 \epsilon}} + \frac{1}{N} \right) \EX [\im m] + \frac{1}{N^2} \right)\\
				& \leq \frac{C N^{-3 \epsilon / 2} N^{d_N - 1}}{q^{d_q - 4 + p_h}} \EX[\im m] + \frac{C N^{\epsilon/2} N^{d_N - 1}}{q^{d_q + p_h}} \EX[\im m] + \frac{C N^{\epsilon/2} N^{d_N - 2}}{q^{d_q + p_h}}   \\
				& \leq C N^{-\epsilon / 2} \EX[\im m] + C \frac{N^{\epsilon/2}}{N}. 
			\end{split}
			\label{q_bound_used_eq}
		\end{equation} 
		This shows \eqref{lemma_tau_g_ineq} under assumption \ref{dqge4_dge2} on $T$. 
		\item 
		Since $p_h \geq 1$ we can perform an additional cumulant expansion with respect to $h_{ab}$, and gain an additional factor of $1/N$ for the resulting terms;
		see (\ref{111}) for instance. 
		The assumption on a distinct index $v$ implies that all the resulting terms in the cumulant expansion must have $d \geq 2$. We may bound the resulting terms with \eqref{very_general_term_naive_bound} and use that $\eta \geq N^{-1 + \epsilon}$ for $z(t) \in \edge(t)$: \begin{equation*}
			\begin{split}
				\abs{\im T} & \leq  C N^{\epsilon/2} \frac{N^{d_N-1}}{q^{d_q + p_h}}\left( \frac{N^{-(d-2) \epsilon/2 }}{N \eta} \EX[\im m] + \frac{\EX[\im m]}{N} + \frac{1}{N^2}\right) \\
				& \leq C N^{-\epsilon/2} \EX[\im m] + C\frac{N^{\epsilon/2}}{N^2}.
			\end{split}
		\end{equation*}
		\item As in the previous case, since $p_h\geq 1$, we can perform a cumulant expansion to obtain a linear combination of terms on the form \eqref{very_general_term} with an additional factor of $1/N$. 
		So we may use Lemma \ref{lemma_general_imm_bound} to bound $T$ as 
		\begin{equation}
			\abs{\im T} \leq N^{\epsilon / 2} \frac{N^{d_N-1}}{q^{d_q}} \left(\EX[\im m] + \frac{1}{N}\right). 
			\label{low_degree_bound}
		\end{equation}  
		Using \eqref{low_degree_bound} and $q \geq N^{\epsilon}$ (which follows from $\epsilon < 1/12$) we get \eqref{lemma_tau_g_ineq}.
		\item Since $d_{\delta} \geq 1$ this means that we can set $a = b$ in the summation of \eqref{very_general_term} and obtain another term of the form \eqref{very_general_term} with an additional factor of $1/N$ (see (\ref{example}) for instance). The bound then follows from \eqref{low_degree_bound} and $q \geq N^{\epsilon}$, as in the previous case.
		\item Note that any unmatched term $T$ of the form \eqref{very_general_term} satisfies (from Lemma \ref{prop_unmatched_term_bound})
		\begin{equation}
			\abs{\im T} \leq N^{\epsilon/2}\frac{N^{d_N}}{q^{d_q}}\left( \frac{\EX[\im m]}{N} + \frac{1}{N^2}\right) = N^{\epsilon/2} \frac{N^{d_N-1}}{q^{d_q}}\left( \EX[\im m] +
			\frac{1}{N}\right).
			\label{very_general_term_unmatched_bound}
		\end{equation}
		We thus use \eqref{very_general_term_unmatched_bound} and $q \geq N^{\epsilon}$ to obtain \eqref{lemma_tau_g_ineq}. 
	\end{enumerate}

	We next deal with the final case $T \in \Tau_{G2}$. In this case the main fact we need to use is $\chi(t) \prec \frac{1}{\sqrt{n}q}$ with $q\geq N^{1/6+\bd}$, and that $\abs{x + \ii \eta} \leq C N^{-2/3 + \epsilon}$. So we may use the bound \eqref{low_degree_bound} and the fact that $d_q \geq 2$ to obtain \begin{equation*}
		\abs{\im T} \leq  C \frac{N^{1 - 2/3 + 3\epsilon/2 }}{q^2} \left(\EX [\im m] + \frac{1}{N}\right) = C \frac{N^{1/3 + 3\epsilon/2 }}{q^2} \left(\EX [\im m] + \frac{1}{N}\right),
	\end{equation*} which implies \eqref{lemma_tau_g_ineq} under the assumption $\epsilon < \bd$. This finishes the proof of Lemma \ref{lemma_lemma_tau_g_ineq}.
\end{proof}

\section{Proof of Proposition \ref{derivative_observable_small}}
\label{general_F_cancellation_section}
The proof of Proposition \ref{derivative_observable_small} is very similar to the proof of Proposition~\ref{imm_bound_random_edge_lemma}. We will expand $\dF \EX[F(X(t))]$ and observe cancellations using a similar strategy as we used in Section~\ref{perfect_cancellation_section}, and then bound the remaining terms as the right hand side of \eqref{derivative_observable_small_eq}. To highlight the similarities, we shall structure this section in a similar manner as Sections \ref{imm_estimate_section}--\ref{imm_estimate_section_2}. However, unlike in the proof of Proposition~\ref{imm_bound_random_edge_lemma}, we will crucially use \eqref{imm_bound} as an input to show \eqref{derivative_observable_small_eq}.

Recall $X(t)$ defined in (\ref{X_observable_def}). We begin by computing $\dF \EX [F(X(t))]$ (\cf \ref{expand}) and obtain that
\begin{equation}
	\begin{split}
		& \dF \EX[F(X(t))] \\
		& = \EX\left[F'(X(t)) N \int_{\gamma_1}^{\gamma_2} \frac{1}{N} \sum_{v} \frac{\partial}{\partial t} \im [G_{vv}] \di x \right] \\
		& = \EX\left[F'(X(t)) \im \sum_{v, a, b} - \frac{\partial h_{ab}(t)}{\partial t} \int_{\gamma_1}^{\gamma_2} \im[G_{va} G_{bv}] \di x \right] + \EX\left[F'(X(t))  \sum_{v, j} \frac{\partial \widehat{L}_t}{\partial t} \int_{\gamma_1}^{\gamma_2} \im[G_{vj} G_{jv}] \di x\right],
	\end{split}
	\label{diff_X_t_initial_eq}
\end{equation} 
where we write for short $G=G(t,z(t))$ defined in (\ref{resolvent_t}) with $z(t)=\widehat{L}_t + x + \ii \eta$. Note that 
\begin{align}
	\sum_{v}\int_{\gamma_1}^{\gamma_2} \im[G_{va} G_{bv}] \di x =&\int_{\gamma_1}^{\gamma_2} \im[ (G^2)_{ab}] \di x\nonumber\\
	=&\im[G_{ab}(t,\widehat{L}_t + \gamma_2 + \ii \eta)]-\im[G_{ab}(t,\widehat{L}_t + \gamma_1 + \ii \eta)], \nonumber
\end{align}
follows by using that $\frac{\dd}{\dd x} G(t,\widehat{L}_t + x + \ii \eta)=G^2(t,\widehat{L}_t + x + \ii \eta)$.
To ease the notation, we introduce the following: Let $\Dim$ be the operator transforming any function (\eg $G=G(t,z(t))$)
$$P: \mathbb{R}^+ \times \mathbb{C}^+ \longrightarrow  \mathbb{C}$$  to the function $\Dim [P] : \mathbb{R}^+ \times \mathbb{R}^+ \longrightarrow  \mathbb{C}$ given by
\begin{equation}
\Dim [P] (t, \eta) := \im [P(t, \widehat{L}_t + \gamma_2+\ii\eta)] - \im [P(t, \widehat{L}_t + \gamma_1+\ii \eta)].
\label{Dim_def}
\end{equation} 
Using the $\Dim$-notation we can neatly rewrite \eqref{diff_X_t_initial_eq} as:
\begin{equation*}
	\dF \EX[F(X(t))] = \EX\left[F'(X(t))  \sum_{a, b} - \frac{\partial h_{ab}(t)}{\partial t} \Dim[G_{ab}] \right] + \EX\left[F'(X(t))  \sum_{v} \frac{\partial \widehat{L}_t}{\partial t} \Dim[G_{vv}] \right].
\end{equation*} Computing the time derivatives we get (\cf the expansion in \eqref{imm_time_derivative_expansion_pre_cumulant})
\begin{equation}
	\begin{multlined}
		\dF \EX[F(X(t))] =  \EX\left[F'(X(t))  \sum_{a, b} \left(-\frac{1}{2} e^{-t/2} w_{ab} + \frac{1}{2} \frac{e^{-t}}{\sqrt{1 - e^{-t}}} h_{ab}\right) \Dim[G_{ab}] \right] \\
		+ \EX\left[F'(X(t))  \sum_{v} \left(\frac{12 e^{-t}(1 - e^{-t}) \kappa_4}{q^2} + e^{-t} \frac{1}{N} \sum_{a,b} \left(h_{ab}^2 - \frac{1}{N}\right)\right) \Dim[G_{vv}]\right].
	\end{multlined}
	\label{diff_X_t_second_eq}
\end{equation} 
To continue expanding $\dF \EX[F(X(t))]$ we need a lemma for cumulant expansions (\cf Lemma \ref{cumulant_remainder_small_lemma_imm_case}), whose proof will be postponed to the appendix. 

\begin{lemma}
	 Let $h_{ab}$ be a fixed entry of $H$ and let $f = f(t, H, W, x, \eta)$ be 
	\begin{equation}
		f := h_{ab}^{p_h} F^{(i_0)}(X(t)) \prod_{i=1}^{i_0}\Dim \left(\prod_{l=1}^{n_i} G_{x_i^{(l)} y_i^{(l)}}\right),
		\label{f_general_F_case}
	\end{equation} where $p_h$ is a non-negative integer and $i_0$ is a positive integer. Then we have the cumulant expansion formula
	\begin{equation}
				\EX[h_{ab}f] = \sum_{k+1=1}^l \frac{\kappa_{k+1}}{Nq^{k-1}} \EX\left[\frac{\partial^k}{\partial h_{ab}^k}f\right] + R_{l+1},
				\label{specific_cumulant_expansion_formula}
		\end{equation} with the remainder term $R_{l+1}$ satisfying 
\begin{equation}
		\abs{R_{l+1}} \prec \frac{1}{N q^{l-1}},
		\label{cumulant_remainder_small_eq}
	\end{equation} uniformly for $t \in [0, 10\log N]$, $\eta \geq N^{-1 + \epsilon}$, and $\abs{x} \leq 1$. Furthermore, for a fixed entry $w_{ab}$ of $W$ we have the formula \begin{equation*}
	\EX[w_{ab}f] = \frac{1}{N} \EX \left[\frac{\partial}{\partial w_{ab}} f\right].
	\label{gaussian_cumulant_expansion_general_case}
	\end{equation*}
	\label{cumulant_remainder_small_lemma}
\end{lemma}

Next, we use the cumulant expansion formula from Lemma \ref{cumulant_remainder_small_lemma} to obtain (\cf (\ref{imm_time_derivative_expansion}))
\begin{equation}
	\begin{multlined}
		\dF \EX[F(X(t))] =  \frac{1}{N} \sum_{a, b} - \frac{1}{2}e^{-t/2} \EX\left[\frac{\partial( F'(X(t)) \Dim G_{ab})}{\partial w_{ab}} \right] \\
		\quad \, + \frac{1}{N} \sum_{a, b} \frac{1}{2} \frac{e^{-t}}{\sqrt{1 - e^{-t}}} \sum_{k+1=1}^{K} \frac{\kappa_{k+1}}{q^{k-1}} \EX \left[\frac{\partial^k( F'(X(t)) \Dim G_{ab})}{\partial h_{ab}^k} \right] \\
		\quad \, + \frac{1}{q^2} \sum_{v} 12 e^{-t}(1 - e^{-t}) \kappa_4 \EX \left[F'(X(t)) \Dim G_{vv}\right] \\
		\quad \, + \frac{1}{N^2} \sum_{v, a, b} e^{-t} \sum_{k+1 = 1}^{K} \frac{\kappa_{k+1}}{q^{k-1}} \EX \left[\frac{\partial^k(h_{ab} F'(X(t)) \Dim G_{vv})}{\partial h_{ab}^k}  \right] - \sum_{v} e^{-t} \EX \left[F'(X(t)) \Dim G_{vv} \right] \\
		\quad \, + O_{\prec} \left(\frac{N}{q^{K-1}}\right).
	\end{multlined}
	\label{F_time_derivative_expansion}
\end{equation} 

To compute the terms on the right side above precisely, we derive the following rules using Lemma \ref{differentiation_lemma_imm_case}. 
\begin{lemma}
	We have the following formulas for differentiating $X(t)$: \begin{align}
		\begin{split}	
			\frac{\partial X(t)}{\partial h_{ab}} & = -2 \sqrt{1 - e^{-t}}  \Dim G_{ab} + \sqrt{1 - e^{-t}} \delta_{ab} \Dim G_{ab} \\
			& \, \quad + 4 (1-e^{-t}) h_{ab} \Dim G_{vv} - 2 (1-e^{-t})  \delta_{ab} h_{ab} \Dim G_{vv},
		\end{split} \label{X_t_h_derivative} \\
		\begin{split}
			\frac{\partial X(t)}{\partial w_{ab}} & = -2 e^{-t/2} \Dim G_{ab} + e^{-t/2} \delta_{ab} \Dim G_{ab}.
		\end{split} \label{X_t_w_derivative}
	\end{align} Furthermore, the operators $\frac{\partial}{\partial h_{ab}}$ and $\Dim$ commute, as do $\frac{\partial}{\partial w_{ab}}$ and $\Dim$. In particular we have \begin{align}
		\frac{\partial}{\partial h_{ab}} \Dim \prod_{i=1}^{n} G_{x_i y_i} & = 	\Dim \frac{\partial}{\partial h_{ab}} \prod_{i=1}^{n} G_{x_i y_i}.
		\label{Dim_commute_diff} \\
		\frac{\partial}{\partial w_{ab}} \Dim \prod_{i=1}^{n} G_{x_i y_i} & = 	\Dim \frac{\partial}{\partial w_{ab}} \prod_{i=1}^{n} G_{x_i y_i}.
		\label{Dim_commute_diff_w} 
	\end{align}
	\label{differentiation_rules_X_lemma}
\end{lemma}

It is now possible to compute the derivatives in \eqref{F_time_derivative_expansion} using Lemma \ref{differentiation_lemma_imm_case} and Lemma \ref{differentiation_rules_X_lemma}. All the resulting terms will have the following general form, which should be considered analogue to Definition \ref{general_term_notation_def}. \begin{definition}[General term notation with $F$]
	Let $\mathcal{I} := \{v_j\}_{j=1}^m$ be a set of summation indices, and denote its cardinality as $\# \mathcal{I}$. Further, let $d_N$, $d_q$, $d_{\delta}$, and $p_h$ be non-negative integers, let $i_0$ be a positive integer, and let $\alpha : \mathbb{R}^+ \to [0, 1]$ be a smooth function denoting a time factor. We use the following notation to denote a general term:
	\begin{equation}
		c \alpha(t) \frac{N^{d_N}}{q^{d_q} N^{\# \mathcal{I}}} \sum_{\mathcal{I}} \delta_{ab}^{d_{\delta}} \EX \left[h_{ab}^{p_h} F^{(i_0)}(X(t)) \prod_{i=1}^{i_0} \Dim \left(\prod_{l=1}^{n_i} G_{x_l^{(i)} y_l^{(i)}}\right)\right].
		\label{F_very_general_term}
	\end{equation} 
	The notation $\sum_{\mathcal{I}}$ means that all indices in $\mathcal{I}$ are iterated from $1$ to $N$. The $x_l^{(i)}$ and $y_l^{(i)}$ take on the $v_j$ as values. Also, in \eqref{very_general_term} we use the convention $0^0 =1$ with $\delta_{ab}^0$. 
	
	We define the degree $d$ of a term on the form \eqref{F_very_general_term} as the number of off-diagonal Green function entries in the products, i.e. $$d := \#\{(l, i) : x_{l}^{(i)} \neq y_{l}^{(i)}, 1 \leq i \leq i_0, 1\leq l \leq n_i\}. $$ 
	The notation $T^{F}$ will often be used  to denote a general term of the form (\ref{start_term_F}), to distinguish it from the form in (\ref{start_term}) without $F$.
\end{definition} 

After computing the derivatives in \eqref{F_time_derivative_expansion} using the derivative rules \eqref{G_h_derivative}, \eqref{G_w_derivative}, \eqref{X_t_h_derivative}, and \eqref{X_t_w_derivative}, all terms will have the form \eqref{F_very_general_term} with $d_N = 1$ and $i_0 \geq 1$. 
 
 As before we will need the concept of equivalent terms for terms of type \eqref{F_very_general_term}. 
 
 \begin{definition}
	We will say that two terms of the form \eqref{F_very_general_term} with $d_{\delta} = p_h = 0$ are equivalent if they have the same $d_N$, $d_q$, $\alpha(t)$, and the indices of the first term can be permuted to obtain the second term.
 \end{definition}
 
 \subsection{Leading term cancellation}
 Before stating the key cancellation lemma, which will be the analogue of Lemma \ref{time_derivative_bound_lemma}, we have to define the corresponding set of non-leading terms, as analogue to Definitions \ref{def_unmatch} and \ref{def_gronwall_term}. 
  \begin{definition}[Unmatched terms]
 	Let $T^F$ be a term of the form \eqref{F_very_general_term} with $p_h = d_{\delta} = 0$. We call $T^F$ unmatched if there exists an index $v_j \in \mathcal{I}$ such that \begin{equation*}
 		\sum_{i=1}^{i_0} \left(\# \{1 \leq l \leq n_i : x_{l}^{(i)} = v_j\} + \# \{1 \leq l \leq n_i : y_{l}^{(i)} = v_j\}\right),
 	\end{equation*} is odd.
 \end{definition}

\begin{definition}[Non-leading terms]
	Let $\Tau_{F1}$ be the set of terms $T^F$ on the form \eqref{F_very_general_term} with $d_N \leq 1$ and for which at least one of the following conditions hold.
	\begin{enumerate}
		\item $p_h \geq 1$, \label{phge1}
		\item $d_{\delta} \geq 1$, \label{ddge1}
		\item $d_q \geq 4$, \label{dqge4}
		\item $p_h = d_{\delta} = 0$ and $T^F$ is unmatched. \label{Funmatched}
	\end{enumerate} Additionally, let $\Tau_{F2}$ be the set of terms on the form \begin{equation}
	c \alpha(t) \frac{N^{d_N}}{q^{d_q} N^{\# \mathcal{I}}} \sum_{\mathcal{I}}  \EX \left[F^{(i_0)}(X(t)) \Dim \left((x + \ii \eta + \chi(t))\prod_{l=1}^{n_1} G_{x_l^{(1)} y_l^{(1)}}\right) \prod_{i=2}^{i_0} \Dim \left(\prod_{l=1}^{n_i} G_{x_l^{(i)} y_l^{(i)}}\right)\right].
	\label{F_very_general_term_x} 
	\end{equation} with $d_N \leq 1$ and $d_q \geq 2$. Finally, define \begin{equation*}
		\Tau_F := \Tau_{F1} \cup \Tau_{F2}.
	\end{equation*} 
We use the notation $O(\Tau_F)$ to denote a linear combination of finitely many terms in $\Tau_F$, which may vary line from line with a slight abuse of notation.
\end{definition} 

We are now ready to state the key cancellation lemma stating that all the leading terms in the expansion of $\dF \EX[F(X(t))]$ cancel up to a collection of finitely many non-leading terms in $\Tau_{F}$. This is the last major step toward proving Theorem \ref{GFCT_random_shift}. 
\begin{lemma} [Leading terms cancel]
	Fix any $D > 0$. Uniformly for $t \in [0, 10 \log N]$ and $N^{-1 + \epsilon} + N^{2 \epsilon} q^{-4} \leq \eta \leq N^{-2/3 - \epsilon/2}$ it holds that \begin{equation}
		\dF \EX[F(X(t))] = O(\Tau_F) + O_{\prec}(N^{-D}).
		\label{F_term_cancellation_eq}
	\end{equation} 
	\label{F_time_derivative_bound_lemma}
\end{lemma}
\begin{proof}
	The proof of Lemma \ref{F_time_derivative_bound_lemma} is very similar to the proof of Lemma \ref{time_derivative_bound_lemma}. We will first show that all terms obtained by expanding the derivatives in \eqref{F_time_derivative_expansion} except the ones with $d_q = 2$, $p_h = d_{\delta} = 0$ belong to $\Tau_F$. Then we can follow a similar strategy as in Section \ref{perfect_cancellation_section} to generate identities and solve the resulting linear equation system to see cancellation up to $O(\Tau_F)$ among the terms in $\dF \EX[F(X(t))]$.
	
	We analyze the terms we get by computing the derivatives in \eqref{F_time_derivative_expansion}. First we handle the terms coming from $k = 1$ in the cumulant expansions. As in the expansion of $\dm \EX[m(t, z(t))]$ we see that the terms coming from the second line of \eqref{F_time_derivative_expansion} with $p_h = 0$ perfectly cancels with the first term on the first line of \eqref{F_time_derivative_expansion}. Also we see that the terms from $k=1$ in the fourth line in \eqref{F_time_derivative_expansion} with $p_h = 0$ cancel with the second term on the fourth line. And as before, the remaining terms from the second line of \eqref{F_time_derivative_expansion} all have $p_h \geq 1$, so they belong to $\Tau_F$, by case \ref{phge1} of $\Tau_{F1}$.
	
	Next we look at the terms with $k = 2$ and $k = 4$. Similarly to the expansion of $\dm \EX[m(t, z(t))]$ all the terms with $p_h = d_{\delta} = 0$ will be unmatched. This follows from considering the terms with $d_{\delta} = 0$ and studying the differentiation rules \eqref{G_h_derivative} and \eqref{X_t_h_derivative}. When applying $\partial/\partial h_{ab}$ to a term, it is either inserted with one $a$ and one $b$, or a $h_{ab}$ is added, or one $h_{ab}$ is removed. So two derivatives are consumed for each step of adding and removing $h_{ab}$, and if we are to end up with $p_h = 0$ we need to remove as many $h_{ab}$'s as we add. For the second line of \eqref{F_time_derivative_expansion} we differentiate $F'(X(t)) \Dim G_{ab}$, so we start with one $a$ and one $b$, i.e. an odd number. And we perform an even number of derivatives, $k = 2, 4$. So an even number of $a$'s and $b$'s will be added, resulting in an unmatched term. For the fourth line of \eqref{F_time_derivative_expansion} we differentiate $h_{ab} F'(X(t)) \Dim G_{vv}$, so we start with zero $a$'s and $b$'s, but we also start with one $h_{ab}$, meaning that one derivative needs to hit this $h_{ab}$ in order for us to have $p_h = 0$ in the end. Hence, an odd number of $a$'s and $b$'s will be added since $k$ is even. So the terms from $k = 2$ or $k = 4$ either has $p_h \geq 1$ or $d_{\delta} \geq 1$ and falls under case \ref{phge1} or \ref{ddge1} of $\Tau_{F1}$, or they are unmatched and fall under case \ref{Funmatched} of $\Tau_{F1}$.
	
	The terms from $k \geq 5$ all have $d_q \geq 4$ and therefore fall under case \ref{dqge4} of $\Tau_{F1}$.
	
	Hence the only terms in the expansion of $\dF \EX[F(X(t))]$  not contained in $\Tau_F$ are the terms with $d_q = 2$ and $d_{\delta} = p_h = 0$. These terms come from $k = 3$ in the cumulant expansions and the third line of \eqref{F_time_derivative_expansion}. We will show that these terms cancel up to $O(\Tau_F) + O_{\prec}(N^{-D})$ and thereby finish the proof in Section \ref{perfect_cancellation_section_general_case}.\end{proof}

Next we present the analogue of Lemma \ref{lemma_lemma_tau_g_ineq} to bound the remaining non-leading terms in (\ref{F_term_cancellation_eq}), and combine Lemmas \ref{F_time_derivative_bound_lemma} and \ref{tau_F_lemma} to finish the proof of Proposition \ref{derivative_observable_small}.
 \begin{lemma}
	Assume $T^F \in \Tau_F$, then \begin{equation}
		\abs{T^F} \leq N^{-1/3 + 2 \epsilon} + \frac{N^{2/3 + 2 \epsilon}}{q^4},
		\label{ineq_tau_F_lemma}
	\end{equation} 
for $N \geq N_0(\epsilon, C_0)$, uniformly for $t \in [0, 10 \log N]$ and $N^{-1 + \epsilon} + N^{2 \epsilon} q^{-4} \leq \eta \leq N^{-2/3 - \epsilon/2}$.
	\label{tau_F_lemma}
\end{lemma}

\begin{proof}[Proof of Proposition \ref{derivative_observable_small}]
	We bound each of the terms in $O(\Tau_F)$ in \eqref{F_term_cancellation_eq} using \eqref{ineq_tau_F_lemma}. By bounding the number of terms in $O(\Tau_F)$ by $N^{\epsilon/16}$ we obtain \eqref{derivative_observable_small_eq} and finish the proof.
\end{proof}

\subsection{Generating identities and obtaining cancellation: Proof of Lemma \ref{F_time_derivative_bound_lemma}}
\label{perfect_cancellation_section_general_case}
In this section we will finish the proof of Lemma \ref{F_time_derivative_bound_lemma} by generating identities and thereby showing cancellation between the terms obtained by expanding the derivatives in \eqref{F_time_derivative_expansion}. The differences from the strategy used in Section \ref{perfect_cancellation_section} are minor, and we will slightly abuse notation. We will again use the resolvent identity \eqref{resolvent_identity} in  different ways to generate identities. Instead of expanding terms of the form \eqref{start_term} we expand terms of the following form (\cf (\ref{F_very_general_term}) with $d_q = 2$, $d_N = 1$, $d_{\delta} = p_h = 0$):
\begin{equation}
	c\alpha(t) \frac{N}{q^{2} N^{\# \mathcal{I}}} \sum_{\mathcal{I}} \EX \left[F^{(i_0)}(X(t)) \prod_{i=1}^{i_0} \Dim \left(\prod_{l=1}^{n_i} G_{x_l^{(i)} y_l^{(i)}}\right)\right].
	\label{start_term_F}
\end{equation} 
We remark that by direct computations using Lemma \ref{differentiation_rules_X_lemma}, the leading terms on the right side of (\ref{F_time_derivative_expansion}) are of the above form with $c=1$, $\alpha(t)=e^{-t}(1-e^{-t})$, and $i_0\geq 1$.

We then use expansion rules 1--3 similarly as in Subsection \ref{perfect_cancellation_section} 
for general terms in the form \eqref{start_term_F}. We will get the equations analogue to \eqref{expansion_off_diagonal}, \eqref{expansion_off_diagonal_reflected}, \eqref{expansion_off_1}, and \eqref{expansion_off_1_reused_index}, where the leading terms will be of the form \eqref{start_term_F} up to multiplication by a constant and the non-leading terms will be in $\Tau_F$. We omit writing out these analogue equations.

We can extend the definition of type-0, type-A, and type-AB term to terms on the form \eqref{start_term_F} (\cf Definition \ref{def_type}). 
\begin{definition}
	Let $T^F$ be a term on the form (\ref{start_term_F}).
	Further, let $\# j$ denote the number of times the index $j \in \mathcal{I}$ occurs in the sequences \begin{equation*}
		\left(x_{l}^{(i)}\right)_{1 \leq i \leq i_0, 1 \leq l \leq n_i}, \left(y_{l}^{(i)}\right)_{1 \leq i \leq i_0, 1 \leq l \leq n_i}.
	\end{equation*} 
We call $T^F$ a type-0 term if $\# j = 2$ for all $j \in \mathcal{I}$. Next we call $T^F$ a type-A term if $\mathcal{I} = \{a\} \cup \widehat{\mathcal{I}}$, and $\# a = 4$, and $\# j = 2$ for all $j \in \widehat{\mathcal{I}}$. Finally, we call $T^F$ a type-AB term if $\mathcal{I} = \{a,b\} \cup \widehat{\mathcal{I}}$, and $\# a = \# b = 4$, and $\# j = 2$ for all $j \in \widehat{\mathcal{I}}$.
\end{definition}

 We next generate identities using an analogue strategy as explained Sections \ref{sec:step2}--\ref{sec:detail} to obtain a list of identities as in (\ref{list_of_identities}). In analogy to $\Tau_4^{\mathrm{start}}$ defined in Definition \ref{def_start}, we define $\Tau_4^{F\text{start}}$ to be the set of all type-0, type-A, and type-AB terms on the form \eqref{start_term_F} with $0 \leq \# \mathcal{I} \leq 4$. In the identities we will get terms with $\# \mathcal{I} \leq 5$, and we will denote the indices of $\mathcal{I}$ as $a, b, c, d, e$, and the only indices that occur more than twice are $a$ and $b$. Also, we allow the terms in $\Tau_4^{F\text{start}}$ to have at most one empty $\Dim$-factor, i.e. $\Dim 1$. Of course $\Dim 1 = 0$, so all the terms with an empty $\Dim$-factor will be 0, but we may expand the $1$ in the empty factor using expansion rules 2 and 3 to obtain non-trivial identities. The term \begin{equation*}
	\e^{-t}(1 - \e^{-t}) \frac{N}{q^{2}} \EX \left[F^{(1)}(X(t)) \Dim \left(1\right)\right],
\end{equation*} is an example of a term with an empty factor that also satisfies $\# \mathcal{I} = 0$.

Similarly to the mechanism introduced below Definition \ref{def_start}, we first use expansion rules 1 and 1', on all off-diagonal Green function entries for each $T^F \in \Tau_4^{F\text{start}}$. Next, we use expansion rule 2 on each of the $\Dim$-factors for terms with $\# \mathcal{I} \leq 3$. Finally, we use expansion rule 3 to expand a 1 in each of the $\Dim$-factors for $T^F \in \Tau_4^{F\text{start}}$ with the existing index $a$ if $\# a \leq 2$ in $T^F$, and we also use expansion rule~3 with the existing index $b$ if $\# b \leq 2$ in $T^F$. Note that if the start term has an empty factor, we can only generate a non-trivial identity if we expand the empty factor using expansion rules 2 or 3. Otherwise all the terms in the identity will contain an empty factor, creating the identity $0=0$. So, when generating the identities as described above, we skip generating these trivial identities. By following this procedure we generate a list of identities analogue to \eqref{list_of_identities}, \ie
\begin{equation}
	\left(0 = \sum_{l=1}^{L_F} \widetilde{c}^{F}_{ml} T^F_{l} + O(\Tau_F) +O_{\prec}\left(N^{-D}\right) \right)_{m=1}^{M_F}, \qquad \Tau^{F\text{basis}}:=(T^{F}_l)_{l=1}^{L_F},
	\label{list_of_identities_F}
\end{equation}
with $M_F=14246$, $L_F=13852$, and $\widetilde{c}^F_{ml}\in \mathbb{Q}$ being rational numbers, where $\Tau^{F\text{basis}}$ is the set of unique basis terms (up to equivalence) that we collected along the way. Here we use the superscript $F$ to distinguish it from \eqref{list_of_identities}.

Similarly to (\ref{imm_expansion_leading_terms_notation}), the leading terms in the expansion of $\dF \EX[F(X(t))]$ in (\ref{F_time_derivative_expansion}) also belong to the basis set $\Tau^{F\text{basis}}$. Hence (\ref{F_time_derivative_expansion}) can be written as
  \begin{equation}
 	\dF \EX[F(X(t))]= \kappa_4 \sum_{l \in \mathcal{L}^F_0} b^F_l T^F_l + O(\Tau_F) + O_{\prec}(N^{-D}),
 	\label{imm_expansion_leading_terms_notation_F}
 \end{equation} 
 with $b^F_l \in \mathbb{Q}~(l\in \mathcal{L}^F_0 \subset \{1,2,\cdots,L_F\})$ being the rational coefficients. Repeating the same arguments used in (\ref{imm_expansion_leading_terms_notation})-(\ref{final_step}), it then suffices to verify the existence of 
 a rational solution $(x^F_{m})_{m = 1}^{M_F} \in \mathbb{Q}^{M_F}$ to the following linear equation
 \begin{equation}
 	\left(\sum_{m = 1}^{M_F}  \widetilde{c}^F_{ml} x^F_{m} = \widetilde{b}^F_{l}\right)_{l=1}^{L_F}, \qquad   \widetilde{b}^F_l:=\mathbf{1}_{l\in \mathcal{L}^F_0} b^F_l.
 	\label{goal_linear_eq_sys_F}
 \end{equation} 
Note that the coefficient matrix $\widetilde{C}^F:=(\widetilde{c}^{F}_{ml})$ for the system (\ref{list_of_identities_F}) has dimensions $13852 \times 14246$. Additionally it has 137616 non-zero entries, which is only $\approx 0.07\%$ of the entries. The rank of the coefficient matrix $\widetilde{C}^F$ is 7893, and the solution vector $(x^F_{m})$ contains 4288 non-zero entries, meaning that 4288 of the identities generated are used to express the expansion of $\dF \EX[F(X(t))]$. The linear equation system and the solution are available at \cite{githubrepo}. Hence we have finished the proof of Lemma \ref{F_time_derivative_bound_lemma}.

\begin{remark}
	$\Tau_4^{F\text{start}}$ is considerably larger than $\Tau_4^{\text{start}}$ from Definition \ref{def_start} as one can generate $\Tau_4^{F\text{start}}$ by taking each term in $\Tau_4^{\text{start}}$ and then partition the Green function entries into different $\Dim$-factors. This means that the list of identities is much longer now, than it was in Subsection \ref{perfect_cancellation_section}, which makes the linear system in \eqref{goal_linear_eq_sys} much more time consuming to solve. The time required to solve the system increased from a few seconds to roughly one hour. 
	\label{difficult_to_solve_lin_eq_remark}. 
\end{remark}

We provide some examples of identities in the list of identities (\ref{list_of_identities_F}).
\begin{example}[Expanding term using expansion rule 1]
	We expand $G_{cd}$ in the outer factor $\Dim (G_{aa}^{2}G_{bb}^{2}G_{cd})$ in the term \begin{equation*}
		e^{-t}(1-e^{-t}) \frac{N}{q^2N^4}  \sum_{a,b,c,d} \EX \left[F^{(2)}(X(t))\Dim (G_{aa}^{2}G_{bb}^{2}G_{cd})\Dim (G_{cd})\right],
	\end{equation*} using expansion rule 1 to obtain the following identity. \begin{equation*}
		\begin{split}
			0 =
			& -\frac{1}{2} e^{-t}(1-e^{-t}) \frac{N}{q^2N^5}  \sum_{a,b,c,d,e} \EX \left[F^{(2)}(X(t))\Dim (G_{aa}^{2}G_{bb}^{2}G_{cd}G_{ee})\Dim (G_{cd})\right] \\
			& -\frac{1}{2} e^{-t}(1-e^{-t}) \frac{N}{q^2N^5}  \sum_{a,b,c,d,e} \EX \left[F^{(2)}(X(t))\Dim (G_{aa}^{2}G_{bb}^{2}G_{de})\Dim (G_{cc}G_{de})\right] \\
			& -e^{-t}(1-e^{-t}) \frac{N}{q^2N^4}  \sum_{a,b,c,d} \EX \left[F^{(2)}(X(t))\Dim (G_{aa}^{2}G_{bb}^{2}G_{cd})\Dim (G_{cd})\right] \\
			& -e^{-t}(1-e^{-t}) \frac{N}{q^2N^5}  \sum_{a,b,c,d,e} \EX \left[F^{(3)}(X(t))\Dim (G_{aa}^{2}G_{bb}^{2}G_{de})\Dim (G_{cd})\Dim (G_{ce})\right] \\
			& -\frac{1}{2} e^{-t}(1-e^{-t}) \frac{N}{q^2N^5}  \sum_{a,b,c,d,e} \EX \left[F^{(2)}(X(t))\Dim (G_{aa}^{2}G_{bb}^{2}G_{ce}G_{de})\Dim (G_{cd})\right] \\
			& -\frac{1}{2} e^{-t}(1-e^{-t}) \frac{N}{q^2N^5}  \sum_{a,b,c,d,e} \EX \left[F^{(2)}(X(t))\Dim (G_{aa}^{2}G_{bb}^{2}G_{de})\Dim (G_{cd}G_{ce})\right] \\
			& -4 e^{-t}(1-e^{-t}) \frac{N}{q^2N^5}  \sum_{a,b,c,d,e} \EX \left[F^{(2)}(X(t))\Dim (G_{aa}G_{ac}G_{ae}G_{bb}^{2}G_{de})\Dim (G_{cd})\right] \\
			& + O({\Tau_F}) + O_{\prec}\left(N^{-1/3}\right). 
		\end{split}
	\end{equation*}
\end{example}

\begin{example}[Expanding term using expansion rule 2]
	We expand a $1$ in the factor $\Dim (G_{aa})$ in the term \begin{equation*}
		e^{-t}(1-e^{-t}) \frac{N}{q^2N^1}  \sum_{a} \EX \left[F^{(2)}(X(t))\Dim (G_{aa})\Dim (G_{aa})\right],
	\end{equation*} to obtain the following identity, \begin{equation*}
		\begin{split}
			0 = 
			& -2 e^{-t}(1-e^{-t}) \frac{N}{q^2N^2}  \sum_{a,b} \EX \left[F^{(2)}(X(t))\Dim (G_{aa})\Dim (G_{aa}G_{bb})\right] \\
			& -1 e^{-t}(1-e^{-t}) \frac{N}{q^2N^3}  \sum_{a,b,c} \EX \left[F^{(2)}(X(t))\Dim (G_{aa})\Dim (G_{aa}G_{bb}G_{cc})\right] \\
			& -1 e^{-t}(1-e^{-t}) \frac{N}{q^2N^1}  \sum_{a} \EX \left[F^{(2)}(X(t))\Dim (G_{aa})\Dim (G_{aa})\right] \\
			& -2 e^{-t}(1-e^{-t}) \frac{N}{q^2N^3}  \sum_{a,b,c} \EX \left[F^{(3)}(X(t))\Dim (G_{aa})\Dim (G_{aa}G_{bc})\Dim (G_{bc})\right] \\
			& -1 e^{-t}(1-e^{-t}) \frac{N}{q^2N^3}  \sum_{a,b,c} \EX \left[F^{(2)}(X(t))\Dim (G_{aa})\Dim (G_{aa}G_{bc}^{2})\right] \\
			& -2 e^{-t}(1-e^{-t}) \frac{N}{q^2N^3}  \sum_{a,b,c} \EX \left[F^{(2)}(X(t))\Dim (G_{aa})\Dim (G_{ab}G_{ac}G_{bc})\right] \\
			& -2 e^{-t}(1-e^{-t}) \frac{N}{q^2N^3}  \sum_{a,b,c} \EX \left[F^{(2)}(X(t))\Dim (G_{aa}G_{bc})\Dim (G_{ab}G_{ac})\right] \\
			& + O({\Tau_F}) + O_{\prec}\left(N^{-1/3}\right). 
		\end{split}
	\end{equation*}
\end{example}

\begin{example}[Expanding term using expansion rule 3]
	We expand $1$ in the factor $\Dim (1)$ in the term \begin{equation*}
		e^{-t}(1-e^{-t}) \frac{N}{q^2N^2}  \sum_{a,b} \EX \left[F^{(2)}(X(t))\Dim (1)\Dim (G_{aa}^{2}G_{bb})\right],
	\end{equation*} using expansion rule 3 with the existing index $b$ to obtain the following identity, \begin{equation*}
		\begin{split}
			0 =
			& -2 e^{-t}(1-e^{-t}) \frac{N}{q^2N^2}  \sum_{a,b} \EX \left[F^{(2)}(X(t))\Dim (G_{aa}^{2}G_{bb})\Dim (G_{bb})\right] \\
			& - e^{-t}(1-e^{-t}) \frac{N}{q^2N^3}  \sum_{a,b,c} \EX \left[F^{(2)}(X(t))\Dim (G_{aa}^{2}G_{bb})\Dim (G_{bb}G_{cc})\right] \\
			& -2 e^{-t}(1-e^{-t}) \frac{N}{q^2N^3}  \sum_{a,b,c} \EX \left[F^{(3)}(X(t))\Dim (G_{aa}^{2}G_{bb})\Dim (G_{bc})\Dim (G_{bc})\right] \\
			& -2 e^{-t}(1-e^{-t}) \frac{N}{q^2N^3}  \sum_{a,b,c} \EX \left[F^{(2)}(X(t))\Dim (G_{aa}^{2}G_{bb}G_{bc})\Dim (G_{bc})\right] \\
			& -1 e^{-t}(1-e^{-t}) \frac{N}{q^2N^3}  \sum_{a,b,c} \EX \left[F^{(2)}(X(t))\Dim (G_{aa}^{2}G_{bb})\Dim (G_{bc}^{2})\right] \\
			& -4 e^{-t}(1-e^{-t}) \frac{N}{q^2N^3}  \sum_{a,b,c} \EX \left[F^{(2)}(X(t))\Dim (G_{aa}G_{ab}G_{ac}G_{bb})\Dim (G_{bc})\right] \\
			& + O({\Tau_F}) + O_{\prec}\left(N^{-1/3}\right). 
		\end{split}
	\end{equation*}
\end{example}

\subsection{Bounding non-leading terms in $\Tau_F$: Proof of Lemma \ref{tau_F_lemma}}

We begin by the following lemma analogous to Lemma \ref{lemma_general_imm_bound}: \begin{lemma}
	Fix any small $\xi > 0$. Let $T^F$ be a general term on the form \eqref{F_very_general_term}. Then \begin{equation}
			\abs{T^F} \leq N^{\xi} \frac{N^{d_N}}{q^{d_q}} \left(\EX[\im m(t, \widehat{L}_t + \gamma_1 + \ii \eta) + \im m(t, \widehat{L}_t + \gamma_2 + \ii \eta)] + \frac{1}{N}\right),
			\label{imm_lemma_F_eq}
		\end{equation} uniformly in $t \in [0, 10 \log N]$ and $N^{-1 + \epsilon} + N^{2 \epsilon} q^{-4} \leq \eta \leq N^{-2/3 - \epsilon/2}$. 
		\label{lemma_F_imm_bound}
\end{lemma} 
Similarly to Lemma \ref{prop_unmatched_term_bound} we have the following bound on unmatched terms on the form \eqref{F_very_general_term}. The proofs of Lemmas \ref{lemma_F_imm_bound}--\ref{prop_unmatched_term_bound_F} are postponed to Appendix \ref{appendix_F_proofs}.
\begin{lemma}
	Fix any small $\xi > 0$. Let $T^F$ be an unmatched term on the form \eqref{F_very_general_term}. Then \begin{equation}
		\abs{T^F} \leq N^{\xi}  \frac{N^{d_N}}{q^{d_q}} \left(\frac{\EX[\im m(t, \widehat{L}_t + \gamma_1 + \ii \eta) + \im m(t, \widehat{L}_t + \gamma_2 + \ii \eta)]}{N} + \frac{1}{N^2}\right),
		\label{prop_unmatched_term_bound_F_eq}
	\end{equation} uniformly in $t \in [0, 10\log N]$ and $N^{-1 + \epsilon} + N^{2 \epsilon} q^{-4} \leq \eta \leq N^{-2/3 - \epsilon/2}$.
	\label{prop_unmatched_term_bound_F}
\end{lemma}

 \begin{proof} [Proof of Lemma \ref{tau_F_lemma}]
 	We handle each of the cases for in the definition of $\Tau_F$. We begin by assuming $T^F \in \Tau_{F1}$, then we have four cases to handle: \begin{enumerate}[label=Case \arabic*:]
 		
 		\item If $p_h \geq 1$ we may perform one cumulant expansion on $T^F$ to get a linear combination of terms of the form \eqref{F_very_general_term} but with an additional factor of $1/N$. Then \eqref{ineq_tau_F_lemma} follows from \eqref{imm_lemma_F_eq} and the bound on $\EX[\im m]$, \eqref{imm_bound}.
 		
 		\item If $d_{\delta} \geq 1$, we may rewrite $T^F$ into another term with an additional factor of $1/N$ as in the previous case.
 		
 		\item Using \ref{imm_lemma_F_eq} and \eqref{imm_bound}, we obtain \begin{equation*}
 			\abs{T^F} \leq \frac{N^{1 + \epsilon/2}}{q^4} \left(2 C N^{-1/3 + \epsilon} + N^{-1}\right) \leq \frac{N^{2/3 + 2 \epsilon}}{q^4},
 		\end{equation*} for sufficiently large $N$. 
 		
 		\item To show this case we use Lemma \ref{prop_unmatched_term_bound_F} and Proposition \ref{imm_bound_random_edge_lemma}. Combined we obtain \begin{equation*}
 			\abs{T^F} \leq N^{d_N - 1 + \epsilon/2} \left(2 C N^{-1/3 + \epsilon} + \frac{1}{N}\right) \leq  N^{-1/3 + 2 \epsilon}.
 		\end{equation*}
 	\end{enumerate} Next we handle the case $T^F \in \Tau_{F2}$, which follows from essentially the same argument as the $\Tau_{G2}$-case in Lemma \ref{lemma_lemma_tau_g_ineq}. We have that $\abs{\chi(t)} \prec 1/(\sqrt{N} q)$, $\abs{\gamma_1} \leq CN^{-2/3 + \epsilon}$ and $\abs{\gamma_2} \leq CN^{-2/3 + \epsilon}$. Expanding the product of $\Dim$-factors as we obtain 
\begin{equation*}
 		\begin{split}
 			\abs{T^F} & \leq N^{\epsilon/2} \frac{N}{q^2} N^{-2/3 + \epsilon} \left(\EX[\im m(t, \widehat{L}_t + \gamma_{1} + \ii \eta )] + \EX[\im m(t, \widehat{L}_t + \gamma_{2} + \ii \eta )] + \frac{1}{N}\right).
 		\end{split}
 	\end{equation*} As in the previous cases we now use the bound on $\EX[\im m]$ (Proposition \ref{imm_bound_random_edge_lemma}) along with the assumption $q \geq N^{1/6 + \bd}$ to obtain \eqref{ineq_tau_F_lemma}. 
 \end{proof}

 \section{Implementation details}\label{sec:code}
 \label{implementation_details_section}

 As we explained in Sections \ref{perfect_cancellation_section} and \ref{perfect_cancellation_section_general_case}, one can find rational solutions to the linear equation systems \eqref{goal_linear_eq_sys} and (\ref{goal_linear_eq_sys_F}), showing that all the leading terms in the expansions of $\dm \EX[m(t, z(t))]$ and $\dF \EX[F(X(t))]$ cancel up to an error of $O(\Tau_G) + O_{\prec}(N^{-D})$ and $O(\Tau_F) + O_{\prec}(N^{-D})$, respectively. Also, we have explained how to generate the identities that were used to obtain the cancellations, but there are some crucial details in the implementation that considerably speeds up the running time. We only describe how to represent and manipulate terms of the form \eqref{F_very_general_term} with a general function $F$, since the set of terms on the form \eqref{very_general_term} has a trivial bijection with the set of terms of the form \eqref{F_very_general_term} with $i_0 = 1$. Further we only represent terms with $d_{\delta} = 0$, since all the terms with $d_{\delta} > 0$ used in this paper belongs to $\Tau_G$ or $\Tau_F$. All or the computations are done using Python except for the Gaussian elimination used for solving the linear systems \eqref{goal_linear_eq_sys} and \eqref{goal_linear_eq_sys_F}, which is implemented in C++. 
 \subsection{Representing general terms in (\ref{F_very_general_term})}
Recall the form \eqref{F_very_general_term}:
  	\begin{equation}
  	c \alpha(t) \frac{N^{d_N}}{q^{d_q} N^{\# \mathcal{I}}} \sum_{\mathcal{I}} \delta_{ab}^{d_{\delta}} \EX \left[h_{ab}^{p_h} F^{(i_0)}(X(t)) \prod_{i=1}^{i_0} \Dim \left(\prod_{l=1}^{n_i} G_{x_l^{(i)} y_l^{(i)}}\right)\right],
  	\label{F_very_general_term7}
  \end{equation}
with $i_0 \geq 1$, $d_N$, $d_q$, $d_{\delta}$, $n_i$, and $p_h$ being non-negative integers. We represent general terms of the above form as a class called \code{Term}. 
The defining data of a \code{Term}-object are stored in the fields \code{F\_derivative}, \code{prod\_dict}, \code{q\_deg}, \code{N\_deg}, \code{coeff}, and \code{h\_list}. 
Firstly, \code{F\_derivative}, \code{q\_deg}, and \code{N\_deg} are integers representing $i_0$, $d_q$, and $d_N$, respectively. Next, \code{prod\_dict} is a list of dictionaries, with each dictionary having sorted pairs of integers as keys and integers as values. Each of the dictionaries represents a product $\prod_{l=1}^{n_i}G_{x_l^{(i)}y_l^{(i)}}$, each of the summation indices in $\mathcal{I}$ are assigned a unique integer, and the key $(x_l^{(i)},y_l^{(i)})$ having value $p$ means that $G_{x_l^{(i)}y_l^{(i)}}$ occurs exactly $p$ times in the product $\prod_{l=1}^{n_i}G_{x_l^{(i)}y_l^{(i)}}$.  Moreover, the field \code{coeff} is an instance of a class we implemented named \code{CumulantPoly}, which represents a commutative polynomial. This commutative polynomial is used to represent $c \, \alpha(t)$. Finally, \code{h\_list} is a list of ordered pairs of integers, each of the pairs representing a $h$ in the product $h_{ab}^{p_h}$, using the same index mapping as \code{prod\_dict}. 
 
 Using the representation of a general term just described, it is straightforward to implement the differentiation rules \eqref{G_h_derivative}, \eqref{G_w_derivative}, \eqref{new_rules_1}, \eqref{X_t_h_derivative}, \eqref{X_t_w_derivative}. Also, as mentioned above, we may neglect all the terms containing a $\delta_{ab}$ arising from the differentiation rules since these will all belong to $\Tau_G$ or $\Tau_F$, respectively. 
 
 \begin{example}
	Consider the term \begin{equation*}
		3 e^{-t}(1-e^{-t}) \kappa_4 \frac{N}{q^2N^4}  \sum_{a,b,c,d} \EX \left[F^{(2)}(X(t))\Dim (G_{aa}^{2}G_{bb}^{2}G_{cd})\Dim (G_{cd})\right]. 
	\end{equation*} An instance of the class \code{Term} representing this term with the index mapping $(a, b, c, d) \mapsto (0, 1, 2, 3)$ has the following values in its fields:
	\begin{itemize}
		\item $\code{F\_derivative} = 2$,
		\item $\code{prod\_dict} = \bigg[\Big\{(0,0): 2, (1, 1) : 2, (2, 3) : 1\Big\}, \Big\{ (2, 3) : 1\Big\} \bigg]$
		\item $\code{q\_deg} = 2$,
		\item $\code{N\_deg} = 1$,
		\item $\code{coeff}  = \code{CumulantPoly}\big(3 \cdot \kappa_4 \, \e^{-t}(1 - \e^{-t})\big)$,
		\item \code{h\_list} = [ ].
	\end{itemize} 
	The list and dictionary notation used is standard for Python. The field \code{coeff} stores as instance of the class \code{CumulantPoly}. In this example, the instance of \code{CumulantPoly} is just a monomial in the variables $\kappa_4$ and $\e^{-t}(1 - \e^{-t})$, with coefficient 3. 
 \end{example}
 
 \begin{remark}
	Storing the field \code{coeff} as an instance of \code{CumulantPoly} is a bit unnecessary for the uses in this paper, since $c \, \alpha(t)$ will be $c \kappa_4 \e^{-t}(1 - \e^{-t}), c \in \mathbb{Q}$ for all the terms in our identities, but implementing a class for cumulative polynomials makes our program more capable of handling other problems. 
 \end{remark}
 
 \subsection{Equivalent terms}\label{sec:basis}
 
 The main implementational difficulty stems from the fact that if one permutes the summation indices in $\mathcal{I}$ of a general term as in \eqref{F_very_general_term7}, the value is invariant. Therefore, the terms we store and manipulate in our Python script should be thought of as representatives of equivalent classes. When we use the derivative rules \eqref{G_h_derivative} and \eqref{G_w_derivative} we will usually generate several terms with different representation in our script, but some of which are equivalent to each other. So after computing derivatives we need to group all of the equivalent terms together. To do this we require an effective way to check if two terms are equivalent or not, \ie if one may permute the indices in the first representation and obtain the second representation. This is a very non-trivial problem. For example, when $i_0 = 1$ in (\ref{F_very_general_term7}) this is equivalent to the graph isomorphism problem, for which no polynomial time algorithm is known \cite{graphIsomorphism}.
  However, we only deal with terms with a sufficiently small number of indices, and have a heuristic technique that makes our program run fast enough.
 
 The idea behind our algorithm for checking if two terms are equivalent is to generate an ``almost unique identifier'', \ie a permutation invariant hash value that is guaranteed to be the same for terms that in fact are equivalent but it should be quite rare for two terms that are not equivalent to have the same ``almost unique identifier''. Similar strategies can be used for the graph isomorphism problem. The tactic we use when constructing our identifier is to count the number of times indices occur together in the products $\prod_{l=1}^{n_i}G_{x_l^{(i)} y_l^{(i)}}$. More precisely, given a term $T$ of the form \eqref{F_very_general_term7} with $d_{\delta} = p_h = 0$ we define the \textit{almost unique identifier}, denoted $\auid(T)$ as the nested tuple of integers 
 \begin{equation}\label{auid}
	\auid(T) := \code{sort}\left(\code{sort}\left(\code{sort}\left(\left(\# \Big\{1 \leq i \leq n_i : \{v, v'\} = \big\{x_l^{(i)}, y_l^{(i)}\big\}\Big\}, n_i\right)\right)_{i=1}^{i_0}\right)_{v \in \mathcal{I}_T}\right)_{v' \in \mathcal{I}_T},
 \end{equation} 
where we used $\mathcal{I}_T$ to denote the set of summation indices $\mathcal{I}$ in (\ref{F_very_general_term7}).
The function $\code{sort}$ takes a tuple and returns the a sorted version of the tuple, using lexicographical ordering. Thanks to the sorting, we see that $\auid(T)$ is invariant under permutation of the indices, since it only depends on how often the indices occur together. Also, we see that $\auid(T)$ contains quite a lot of information about $T$, so at least intuitively it should be rare that two terms that are not equivalent get the same $\auid$. 
 
 Now we are ready to describe how we check whether two terms are equivalent. Assume we are given two terms $T$ and $S$ of the form \eqref{F_very_general_term7} with $d_{\delta} = p_h = 0$ and want to check if $T$ and $S$ are equivalent of not. We begin by comparing $\auid(T)$ with $\auid(S)$. If we have $\auid(T) \neq \auid(S)$ we can say for sure that $T$ and $S$ are not equivalent. On the other hand, if we have $\auid(T) = \auid(S)$ we need to perform further computations. This is handled in two different ways depending on the integer $i_0$ from (\ref{F_very_general_term7}). 
 
 \textbf{Case 1:} ($i_0 = 1$) Since $i_0 = 1$ our problem of checking whether $T$ and $S$ are equivalent is an instance of the Graph isomorphism problem. There has been extensive research on this problem with efficient algorithms. 
 In this paper we use the algorithm VF2++ introduced in \cite{JUTTNER201869}. The algorithm is available in the Python package NetworkX, introduced in \cite{networkx}. 
  
 \textbf{Case 2:} ($i_0 > 1$) This case is handled by recursively trying all mappings of the indices from $\mathcal{I}_T$ to the indices in $\mathcal{I}_S$ and comparing if we get the same \code{prod\_dict} (after sorting the list). However, when trying to map an index in $\mathcal{I}_T$ to another index in $\mathcal{I}_S$, we check that the indices occurs the same number of times in the products of Green function entries of $T$ and $S$, hence speeding up the process, by ruling out index mappings we know will fail early in the recursion. 

\subsubsection{\bf Adding and finding basis terms in $\Tau^{\mathrm{basis}}$}
\label{basis_term_implementation_section}
In order to build the linear equation system in \eqref{goal_linear_eq_sys_F} we need to compute the set of basis terms, $\Tau^{\mathrm{basis}}$. To do this we need a data structure to store terms and their equivalence classes. We also need fast lookup and insertion times. We implement the class \code{Ab\_term\_container} for this purpose. It works essentially in the same way as a hash table and we use $\auid$ as the hash-function. 

To store a new instance of \code{Term} in an instance of \code{Ab\_term\_container} we first check if this exact representation of the term is already stored in the container, if this is the case we do nothing more. If the exact representation is not already stored  we proceed to check if there is some other term already stored in the \code{Ab\_term\_container} with the same $\auid$ defined in (\ref{auid}) above. If not, we store the term under this $\auid$ in the instance of \code{Ab\_term\_container}. If there already are terms in the container with identical $\auid$ we check for equivalence between the term we wish to insert and each of the equivalence classes already stored with the same $\auid$. If we find an equivalent term already stored, we store that the representation of the term we inserted is equivalent to the equivalent class already stored. Else if we do not find an equivalent term already stored in the container, we store it as a new equivalence class with only one representative. 

To check if an instance of \code{Term} is stored in an instance of \code{Ab\_term\_container} we follow the same procedure as when inserting a term, as described above, except that we do not insert the term if we find it not being contained in the container. 

\subsection{Solving linear equation systems in rational numbers}\label{sec:solution}
As pointed out in Remark \ref{difficult_to_solve_lin_eq_remark}, solving the linear equation system in \eqref{goal_linear_eq_sys_F} needed for the proof of Lemma \ref{F_time_derivative_bound_lemma} turns out to be the most computationally demanding part of this paper. The problems arises from the fact that we need to find an exact solution for \eqref{goal_linear_eq_sys_F} in order for our proof to work. In practice this means finding a rational solution to \eqref{goal_linear_eq_sys_F}, since for a rational solution we can verify symbolically that the solution we find is an exact solution, whilst a floating point solution only is guaranteed to be an approximate solution. We note that crucially, the systems we need to solve are very sparse, since typically an identity includes at most 10 different terms. So each column in the coefficient matrices have around 10 non-zero entries. 
 
The available software for the type of problem we are faced with, \ie solving a sparse rectangular system over the rational numbers is limited. Most available exact solvers, such as SPEX left LU \cite{spexLU} only handle square systems of full rank. 

As is pointed out in \cite{spexLU}, it is usually best to pick the smallest possible pivot when using fractions, since this helps to keep the denominators small. 

For the small system ($110 \times 138$) we need to solve for the proof of Lemma \ref{time_derivative_bound_lemma} we search through the entire non-triangulated part of the matrix and pick the smallest pivot possible. Doing this we find a solution consisting entirely of integers. For the proof of Lemma \ref{F_time_derivative_bound_lemma}, the system we need to solve has the dimensions $13852 \times 14246$. For a $n \times m$ matrix Gaussian elimination requires $O(nm \min\{n, m\})$ arithmetic operations, which is too much with these dimensions. So we need to use the sparseness of the system to obtain a better time complexity. We resolve it as follows: when reducing the matrix to triangular form, most entries will already be zero, so for the rows beginning with zero we do not have to do anything. For the pivots we simply choose the first non-zero pivot we can find.

Our implementation is available at \cite{githubrepo}.

\appendix
\section{Proofs of several lemmas}
\label{proof_of_lemmas_appendix}
In Appendix \ref{proof_of_lemmas_appendix}, we present the proofs of several lemmas stated before.

\subsection{Proof of Lemma \ref{second_eigenval_estimate_lemma}}\label{app_lemma}

We start with the proof of Lemma \ref{second_eigenval_estimate_lemma}, using the same arguments as in the proof of \cite[Lemma 2.5]{Schnelli_Xu_2022}. 
\begin{proof} [Proof of Lemma \ref{second_eigenval_estimate_lemma}]
	The proof is identical to the proof of Lemma 2.5 in \cite{Schnelli_Xu_2022}. The only missing ingredient is an optimal local law for $\im m$ that holds for $\eta \gg N^{-1}$, \ie \begin{equation}
		\abs{\im m(E + \ii \eta) - \im \widetilde{m} (E + \ii \eta)} \prec \frac{1}{N \eta}
		\label{imm_local_law_eq}
	\end{equation}	holding uniformly for $\eta \gg N^{-1}$. In fact, \eqref{imm_local_law_eq} can be extended to all $E + \ii \eta \in S_0$ by following the arguments presented in Section 10 of \cite{benaychgeorges2018lectureslocalsemicirclelaw}: Pick the $\bc$ of Theorem \ref{optimal_local_law_sparse}  to be $< \epsilon/2$, then define $\eta_0$ through $N\eta_0 \sqrt{\kappa + \eta_0} = N^{\bc}$. From Theorem \ref{optimal_local_law_sparse} we have \eqref{imm_local_law_eq} for $\eta_0 \leq \eta \leq 5$. To extend it to $\eta > 0$, pick an $\eta < \eta_0$ and bound \begin{equation*}
		\begin{split}
			\abs{\im m(E + \ii \eta) - \im \widetilde{m} (E + \ii \eta)}
			\leq & \abs{\im m(E + \ii \eta) - \im m(E + \ii \eta_0)} \\
			& + \abs{\im m(E + \ii \eta_0) - \im \widetilde{m} (E + \ii \eta_0)} \\
			& + \abs{\im \widetilde{m} (E + \ii \eta_0) - \im \widetilde{m} (E + \ii \eta)}.
		\end{split}
	\end{equation*} The second term  on the right is $\prec 1/(N \eta_0)$ by Theorem \ref{optimal_local_law_sparse}, and the third is $\leq 2 N^{\bc}/(N \eta_0)$ with high probability, by \eqref{imm_hat_estimate}. The first term can be bounded as \begin{equation*}
		\begin{split}
			\abs{\im m(E + \ii \eta) - \im m(E + \ii \eta_0)} 
			& = \abs{\int_{\eta}^{\eta_0} \frac{\di}{\di y} \im m(E + \ii y) \di y} \\
			& = \abs{\int_{\eta}^{\eta_0} \im \frac{\ii}{N} \Tr G(E + \ii y)^2 \di y} \\
			& \leq \int_{\eta}^{\eta_0} \frac{y\im m(E + \ii y)}{y^2} \di y.
		\end{split} 
	\end{equation*} In the last inequality we used the Ward identity. Next we use that $y \mapsto y \im m(E + \ii y)$ is strictly increasing with $y$, so the last term can be bounded as \begin{equation*}
		\int_{\eta}^{\eta_0} \frac{\eta_0 \im m(E + \ii \eta_0)}{y^2} \di y \leq \frac{1}{\eta} \eta_0 \im m(E + \ii \eta_0).
	\end{equation*} But now we can use Theorem \ref{optimal_local_law_sparse} and \eqref{imm_hat_estimate} on this last term to bound it with high probability as \begin{equation*}
		\frac{\eta_0}{\eta} \left(\frac{N^{\epsilon/2}}{N \eta_0} + \frac{N^{\bc}}{N \eta_0}\right) \leq 2 \frac{N^{\epsilon/2}}{N \eta}.
	\end{equation*} This concludes the extension argument. 
	 
	 Using \eqref{imm_local_law_eq}, the edge rigidity \eqref{edge_rigidity}, and the bound in \eqref{total_eigenvalues_at_edge_bound} we can follow the same argument as in the proofs of Lemmas 2.4 and 2.5 in \cite{Schnelli_Xu_2022}.  
\end{proof}

\subsection{Proof of Lemmas \ref{cumulant_remainder_small_lemma_imm_case} and \ref{cumulant_remainder_small_lemma}}

	Next we will prove Lemma \ref{cumulant_remainder_small_lemma_imm_case} and Lemma \ref{cumulant_remainder_small_lemma} using the cumulant expansion formula in Lemma \ref{cumulant_expansion_lemma}. 
	
\begin{proof}[Proof of Lemmas \ref{cumulant_remainder_small_lemma_imm_case} and \ref{cumulant_remainder_small_lemma}]
	We follow the arguments from the proof of Lemma 3.1 in \nc\cite{Schnelli_Xu_2022}. The expansions \eqref{specific_cumulant_expansion_formula_imm_case} and \eqref{specific_cumulant_expansion_formula} follows directly from Lemma \ref{cumulant_expansion_lemma} since both the $f$ in \eqref{f_imm_case} and \eqref{f_general_F_case} are smooth for $\eta > 0$. So what we have to show is that the bound on the remainder term in \eqref{cumulant_expansion_error_bound} implies the bounds in \eqref{cumulant_remainder_small_eq_imm_case} and \eqref{cumulant_remainder_small_eq}. We will use the cutoff $M := q^{-3/4}$. All the inequalities and stochastic dominations in this proof are meant uniformly in $t \in [0, 10 \log N]$, $\eta \geq N^{-1 + \epsilon}$, and $\abs{x} \leq 1$. The first step to show is that \begin{equation}
		\sup_{\abs{h_{ab}} \leq q^{-3/4}} \left\{\max_{i,j} \abs{G_{ij}}\right\} \prec 1.
		\label{cumulant_sup_bound}
	\end{equation} Let $E^{kl}$ denote the matrix with a one at position $kl$, i.e.\ $E^{kl} := (\delta_{ki}\delta_{lj})_{i,j=1}^{N}$. Then we can define $G^{(ab)}$ as $G$ but with $h_{ab}$ set to zero. We can express $G^{(ab)}$ as \begin{equation*}
		G^{(ab)} = \left(H(t) - z(t) - \frac{\sqrt{1 - e^{-t}} h_{ab}}{1+\delta_{ab}}(E^{ab} + E^{ba})  + \frac{2(1 - e^{-t}) h_{ab}^2}{N(1 + \delta_{ab})} I \right)^{-1}.
	\end{equation*} From this we get \begin{equation*}
		G^{(ab)}_{ij} = G_{ij} + \left(G^{(ab)} \left( \frac{\sqrt{1 - e^{-t}} h_{ab}}{1 + \delta_{ab}} (E^{ab} + E^{ba}) - \frac{2(1 - e^{-t}) h_{ab}^2}{N(1 + \delta_{ab})} I\right) G\right)_{ij}. 
	\end{equation*} Now we can take the absolute value, use the triangle inequality, and take $\max_{i,j}$ on both sides to obtain \begin{equation}
		\max_{i,j} \abs{G_{ij}^{(ab)}} \leq \max_{i,j} \abs{G_{ij}} + 2 \abs{h_{ab}} \max_{i,j} \abs{G_{ij}^{(ab)}} \max_{i,j} \abs{G_{ij}} + 2 \abs{h_{ab}}^2 \max_{i,j} \abs{G_{ij}^{(ab)}} \max_{i,j} \abs{G_{ij}}.
		\label{max_G_ab_bound}
	\end{equation} From \eqref{time_dependent_local_law} we know that $\max_{i,j} \abs{G_{ij}} \prec 1$. By Assumption \ref{H_assumptions} we know that $\abs{h_{ab}} \prec 1/q$. So together with \eqref{max_G_ab_bound} we obtain that also $\max_{i,j} \abs{G_{ij}^{(ab)}} \prec 1$ holds. Similarly, we can now express $G_{ij}$ as \begin{equation*}
		G_{ij} = G^{(ab)}_{ij} - \left(G \left( \frac{\sqrt{1 - e^{-t}} h_{ab}}{1 + \delta_{ab}} (E_{ab} + E_{ba}) - \frac{2(1 - e^{-t}) h_{ab}^2}{N(1 + \delta_{ab})} I\right) G^{(ab)}\right)_{ij}. 
	\end{equation*} Again, we can take absolute values, use the triangle inequality, and take $\max_{i,j}$ on both sides. Then we also take $\sup_{\abs{h_{ab}} \leq q^{-3/4}}$ on both sides. Using that $G_{ij}^{(ab)}$ does not depend on $h_{ab}$ we then obtain \eqref{cumulant_sup_bound}.
	
	Next, we see from the differentiation rules \eqref{G_h_derivative}, \eqref{X_t_h_derivative}, and \eqref{Dim_commute_diff} that $f^{(l)}$ will be a linear combination of functions of the form \eqref{f_imm_case} or \eqref{f_general_F_case}, but with a time factor $\alpha(t)$ satisfying $0 \leq \alpha(t) \leq 1$ in front. Furthermore, we have that $\abs{h_{ab}} \prec 1/q$ and by assumption the derivatives of $F$ are bounded. So from \eqref{cumulant_sup_bound} we get that $\sup_{\abs{h_{ab}} \leq q^{-3/4}} \abs{f^{(l)}(h_{ab})} \prec 1$. From Assumption \ref{H_assumptions} we have that $\EX[h_{ab}^{l+1}] \prec 1 / (N q^{l-1})$. So we are done with bounding the first term in \eqref{cumulant_expansion_error_bound}. 
	
	The remaining step is to bound the second term in \eqref{cumulant_expansion_error_bound}, i.e.\ $\EX[\abs{h_{ab}}^{l+1} \indicator_{\abs{h_{ab}} > q^{-3/4}}] \norm{f^{(l)}}_{\infty}$. For $\norm{f^{(l)}}_{\infty}$ we use that $\max_{i, j} \abs{G_{ij}} \leq 1/\eta \leq N$. At most $l$ Green function entries are added to the products in \eqref{f_imm_case} and \eqref{f_general_F_case} when differentiating $l$ times. So if we denote $n = \sum_{i=1}^{i_0} n_i$ we have $\norm{f^{(l)}}_{\infty} \prec N^{n + l}$. For the other factor, $\EX[\abs{h_{ab}}^{l+1} \indicator_{\abs{h_{ab}} > q^{-3/4}}]$ we use Cauchy-Schwarz to bound \begin{equation*}
		\EX_{h_{ab}} \left[\abs{h_{ab}}^{l+1} 1_{\abs{h_{ab}} >q^{-3/4}}\right] \leq \sqrt{\EX_{h_{ab}} \left[\abs{h_{ab}}^{2(l+1)}\right] \EX_{h_{ab}}  \left[1_{\abs{h_{ab}} >q^{-3/4}}\right]} \leq \sqrt{\frac{C}{Nq^{2l}} \Prob \left(\abs{h_{ab}} > q^{-3/4}\right)}.
	\end{equation*}
	From $h_{ab} \prec 1/q$ we see that $\Prob(\abs{h_{ab}} > q^{-3/4}) \leq N^{-\Gamma}$ for any $\Gamma > 0$, for $N$ large enough. So by choosing $\Gamma = 2(n + l + 1)$ we get the bound $\EX[\abs{h_{ab}}^{l+1} \indicator_{\abs{h_{ab}} > q^{-3/4}}] \norm{f^{(l)}}_{\infty} \prec 1/ (N^{3/2}q^{l}) \prec 1/(N q^{l-1})$. So the proof is done. 
\end{proof}

\subsection{Proof of Lemmas \ref{lemma_general_imm_bound}--\ref{prop_unmatched_term_bound}}\label{sec:proof_unmatched}

Before we prove Lemmas \ref{lemma_general_imm_bound}--\ref{prop_unmatched_term_bound}, we begin by deriving a preliminary estimate on unmatched terms from Definition \ref{def_unmatch} (without taking the imaginary part). Comparing to the a priori bound in \eqref{naive_local_law_bound}, we gain a factor of $1/N$ for the unmatched terms. 

\begin{lemma}
	Fix any small $\xi>0$. Then	any unmatched term $T$ from Definition \ref{def_unmatch} in the general form of \eqref{very_general_term} satisfies
	\begin{equation*}
		\abs{T} \leq N^{\xi}\frac{N^{d_N - 1}}{q^{d_q}},
	\end{equation*} 
	with non-negative integers $d_N, d_q$ from \eqref{very_general_term}, uniformly in $z(t) \in \edge(t)$ and $t \in [0, 10 \log N]$. 
	\label{preliminary_unmatched_bound}
\end{lemma}

\begin{proof}[Proof of Lemma \ref{preliminary_unmatched_bound}]
	The proof is similar to the proof of \cite[Proposition 4.3]{Schnelli_Xu_2022}. The strategy consists of expanding an unmatched term into a linear combination of unmatched terms of higher degree. 
	Repeating this procedure sufficiently many times, the unmatched term one starts with can be written as a linear combination of terms with arbitrarily high degrees. In the final step we use the a priori bound \eqref{naive_local_law_bound} to estimate these terms with sufficiently high degrees. 
	
	For ease of notation we assume that $d_q = d_N = 0$, $c = 1$, and $\alpha(t) = 1$, and show that $\abs{T} \prec 1/N$. 
	We shall utilize the following identity to expand the Green function entries:
	 \begin{equation}
		G_{ij} = \delta_{ij} \underline{G} + G_{ij} \underline{H(t)G} - \underline{G}(H(t)G)_{ij},
		\label{green_function_entry_expansion}
	\end{equation} 
where we recall that $\underline{A}$ denotes the normalized trace of $A$.
	
	We assume that $T$ has degree $d$. Since $T$ is unmatched, there is some summation index, which we assume to be $v_1$, that is unmatched. Also there is some factor $G_{x_i y_i}$ such that $x_i = v_1$ and $y_i \neq v_1$. Assume without loss of generality that $i = 1$. We then have \begin{equation*}
		\begin{split}
			T	& = \frac{1}{N^{\# \mathcal{I}}} \sum_{\mathcal{I}} \EX \left[G_{v_1 y_1} \prod_{i=2}^n G_{x_i y_i}\right] = \frac{1}{N^{\# \mathcal{I}}} \sum_{\mathcal{I}} \EX \left[\delta_{v_1 y_1} m \prod_{i=2}^n G_{x_i y_i} \right] \\
			& \quad + \frac{1}{N^{\# \mathcal{I}}} \sum_{\mathcal{I}}  \EX \left[\underline{H(t)G} G_{v_1 y_1} \prod_{i=2}^n G_{x_i y_i} - \underline{G} (H(t)G)_{v_1 y_1} \prod_{i=2}^n G_{x_i y_i}\right] \\
			& = \frac{1}{N^{\# \mathcal{I}}} \sum_{\mathcal{I}}  \EX \left[\frac{1}{N} \sum_{j,k} h_{jk}(t) G_{kj} G_{v_1 y_1} \prod_{i=2}^n G_{x_i y_i}\right] \\
			& \quad - \frac{1}{N^{\# \mathcal{I}}} \sum_{\mathcal{I}}  \EX \left[\frac{1}{N} \sum_{j} G_{jj} \sum_{k} h_{v_1 k}(t) G_{k y_1} \prod_{i=2}^n G_{x_i y_i}\right] + O_{\prec}\left(\frac{1}{N}\right).
		\end{split}
	\end{equation*} Using the definition of $H(t)$ and cumulant expansions to order $l$ we obtain:
	\begin{equation}
		\begin{split}
			T = & \frac{1}{N^{\# \mathcal{I} + 2}} \sum_{\mathcal{I} \cup \{j, k\}}   \sum_{p+1 = 1}^{l} \sqrt{1-e^{-t}} \frac{\kappa_{p+1}}{q^{p-1}} \EX \left[\frac{\partial^p(G_{kj} G_{v_1 y_1} \prod_{i=2}^{n} G_{x_i y_i})}{\partial h_{jk}^p} \right] \\
			& \quad \, + \frac{1}{N^{\# \mathcal{I} + 2}} \sum_{\mathcal{I} \cup \{j, k\}} e^{-t/2} \EX \left[\frac{\partial ( G_{kj} G_{v_1 y_1} \prod_{i=2}^{n} G_{x_i y_i})}{\partial w_{jk}}\right] \\
			& - \frac{1}{N^{\# \mathcal{I} + 2}} \sum_{\mathcal{I} \cup \{j, k\}}  \sum_{p+1 = 1}^{l} \sqrt{1-e^{-t}} \frac{\kappa_{p+1}}{q^{p-1}} \EX \left[\frac{\partial^p (G_{jj} G_{k y_1} \prod_{i=2}^{n} G_{x_i y_i})} {\partial h_{v_1 k}^p} \right]  \\
			& \quad \, - \frac{1}{N^{\# \mathcal{I} + 2}} \sum_{\mathcal{I} \cup \{j, k\}}  e^{-t/2} \EX \left[\frac{\partial (G_{jj} G_{k y_1} \prod_{i=2}^{n} G_{x_i y_i})}{\partial w_{v_1 k}} \right] +  O_{\prec}\left(\frac{1}{q^{l-1}}\right) + O_{\prec}\left(\frac{1}{N}\right).
		\end{split}
		\label{unmatched_middle_expansion_eq}
	\end{equation} 
	Before continuing we ease notation by introducing the differential operator $D_{ab}$ acting on functions $f = f(t, z)$ through \begin{equation}
		D_{ab} f := \left(\frac{\partial}{\partial h_{ab}(t)} f(t, z)\right)\bigg|_{z = z(t)}.
		\label{eq_def_of_D}
	\end{equation} In other words, $D_{ab}$ is the partial derivative $ \partial/\partial h_{ab}(t)$ not acting on $z(t)$. We note how it acts on $G_{ij}$: \begin{equation}
		D_{ab} G_{ij} = -\Big(G_{ia}G_{bj} + G_{ib}G_{aj}\Big) + \delta_{ab} G_{ia}G_{bj}.
		\label{G_D_derivative}
	\end{equation}
	
	Shifting focus back to \eqref{unmatched_middle_expansion_eq} we can see that the term from $p+1 = 1$ is zero since $\kappa_1 = 0$. Also, from the differentiation rules \eqref{G_h_derivative} and \eqref{G_w_derivative}, and $\sqrt{1-e^{-t}}^2 + \left(e^{-t/2}\right)^2 = 1$, we can group the terms with $p+1 = 2$ and the ones with $w$-derivatives using the notation introduced in \eqref{eq_def_of_D} as follows:
	\begin{equation}
		\begin{split}
			T = & \frac{1}{N^{\# \mathcal{I} + 2}} \sum_{\mathcal{I} \cup \{j, k\}} \EX \left[ D_{jk} (G_{kj} G_{v_1 y_1} \prod_{i=2}^{n} G_{x_i y_i})\right] \\
			& \quad \, + \frac{1}{N^{\# \mathcal{I} + 2}} \sum_{\mathcal{I} \cup \{j, k\}}   \sum_{p+1 = 3}^{l} \sqrt{1-e^{-t}} \frac{\kappa_{p+1}}{q^{p-1}} \EX \left[\frac{\partial^p (G_{kj} G_{v_1 y_1} \prod_{i=2}^{n} G_{x_i y_i})}{\partial h_{jk}^p} \right]  \\
			& \quad \, - \frac{1}{N^{\# \mathcal{I} + 2}} \sum_{\mathcal{I} \cup \{j, k\}}  \EX \left[D_{v_1 k} (G_{jj} G_{k y_1} \prod_{i=2}^{n} G_{x_i y_i})\right] \\
			& \quad \, - \frac{1}{N^{\# \mathcal{I} + 2}} \sum_{\mathcal{I} \cup \{j, k\}}   \sum_{p+1 = 3}^{l} \sqrt{1-e^{-t}} \frac{\kappa_{p+1}}{q^{p-1}} \EX \left[\frac{\partial (G_{jj} G_{k y_1} \prod_{i=2}^{n} G_{x_i y_i})}{\partial h_{v_1 k}} \right] + O_{\prec}\left(\frac{1}{q^{l-1}}\right) + O_{\prec}\left(\frac{1}{N}\right).
		\end{split}
		\label{unmatched_expansion}
	\end{equation} By Assumption \ref{H_assumptions}, $q \geq N^{1/6 + \bd}$, so we can carry out the cumulant expansions to large enough order so that $q^{l-1} \geq N$, then the error term in \eqref{unmatched_expansion} will be $O_{\prec}(1/N)$. Also, note that the terms containing $h_{ab}$ coming from the derivative rule \eqref{G_h_derivative} for $p+1=2$ are estimated as $O_{\prec}(1/N)$, since these terms contain a $h_{ab}$ and may be further cumulant expanded, reducing $d_N$ by one. 
	
	Next we investigate what happens with the degrees of the terms in \eqref{unmatched_expansion}.  
	Using the derivative rules \eqref{G_D_derivative} and \eqref{G_h_derivative} one may expand the last expression in \eqref{unmatched_expansion} into a linear combination of terms of the form \eqref{very_general_term} with $d_N = 0$. The terms having $p_h \geq 1$ or $d_{\delta} \geq 1$, are estimated as $O_{\prec}(1/N)$. The first property we will show is that the terms appearing in the last expression in \eqref{unmatched_expansion} with $p_h = d_{\delta} = 0$ will have at least degree $d$. We consider what happens when applying $D_{ab}$ or $\partial / \partial h_{ab}$ successively, using \eqref{G_D_derivative} or \eqref{G_h_derivative}. One sees that the product of entries in the resulting terms arising from differentiating $G_{x_i y_i}$ can be written as \begin{equation}
		G_{x_i \beta_1} G_{\alpha_2 \beta_2} \cdots G_{\alpha_{p-1} \beta_{p-1}} G_{\alpha_p y_i},
		\label{differentiation_result}
	\end{equation} where the $\alpha_i$ and $\beta_i$ either are $a$ or $b$, or a fresh index originating from the third line of \eqref{G_h_derivative}.
	
	In the first and second lines of \eqref{unmatched_expansion} we have $(a, b) = (j, k)$. So, both of the factors $G_{x_i \beta_1}$ and $G_{\alpha_p y_i}$ in \eqref{differentiation_result} will be off-diagonal, since $j$ and $k$ are fresh indices. So we will always have at least $d$ off-diagonal entries (in fact, the only way to not increase the degree is for all derivatives to hit $G_{kj}$). 
	
	For the third and fourth lines of \eqref{unmatched_expansion} we have $(a, b) = (v_1, k)$ so it is possible to get $G_{x_i \beta_1}$ and $G_{\alpha_p y_i}$ in \eqref{differentiation_result} to be diagonal. However, if $x_i=y_i$ we cannot decrease the degree, and if $x_i \neq y_i$, we must have $x_i = \beta_1$ and $y_i = \alpha_p$ in order to decrease the degree, which means that both $x_i$ and $y_i$ must be either $v_1$, $k$, or a fresh index. But this cannot happen since $k$ is a fresh index and only appears in $G_{k y_1}$, and $y_1 \neq v_1$. So we have shown that all terms in \eqref{unmatched_expansion} will have at least degree $d$. 
	
	We also see that all the terms in the last expression of \eqref{unmatched_expansion} will again be unmatched. For the first and second line of \eqref{unmatched_expansion} we only add the fresh indices, so $v_1$ will remain unmatched. Further, for the third and fourth lines of \eqref{unmatched_expansion} we note that there must be at least another unmatched index besides $v_1$, that will remain unmatched after carrying out the differentiation, since we only add the index $v_1$ or fresh indices when differentiating. 
	
	All the terms in the sums with $p+1$ starting as 3 gets an additional factor of $1/q^{p-1}$. So we turn our focus to the leading order terms in \eqref{unmatched_expansion}, \ie the terms with the $D_{jk}$- and $D_{v_1 k}$-derivatives. 
	We wish to investigate what happens with the degrees. First, consider the terms with $D_{jk}$ in \eqref{unmatched_expansion}. As remarked earlier, if $D_{jk}$ hits $G_{x_i y_i}$ for some $i$, the degree is increased. And if $D_{jk}$ hits $G_{kj}$ we get two terms with factors $-G_{kk}G_{jj}$ and $-G_{kj}G_{kj}$ respectively. For the factor $G_{kj}G_{kj}$ we gain one degree, but for $G_{kk}G_{jj}$ the degree is still $d$. But looking at the terms from $D_{v_1 k}$ in \eqref{unmatched_expansion} and consider the case when $D_{v_1 k}$ hits $G_{ky_1}$ we get the two factors $-G_{kk}G_{v_1 y_1}$ and $-G_{k v_1}G_{k y_1}$. So we see that the term with the factor $-G_{kk}G_{v_1 y_1}$ will exactly cancel the term from $G_{kk}G_{jj}$ in the first part of \eqref{unmatched_expansion}. So we see that all terms coming from $D_{jk}$ in \eqref{unmatched_expansion} have at least degree $d+1$. 
	
	Now we turn our focus to the third line of \eqref{unmatched_expansion}, \ie the terms coming from $D_{v_1 k}$. If  $D_{v_1 k}$ hits $G_{jj}$ the degree is increased (by 2). The case when we hit $G_{ky_1}$ was handled in the last paragraph: One term gets canceled, and in $G_{k v_1}G_{k y_1}$ the degree has increased by one. Now we consider what happens if $D_{v_1 k}$ hits $G_{x_i y_i}$ for $2 \leq i \leq n$. Then we get $-(G_{x_i v_1} G_{k y_i} + G_{x_i k} G_{v_1 y_i})$. If both $x_i \neq v_1$ and $y_i \neq v_1$ we see that the degree is increased at least by one. Also, if both $x_i = y_i = v_1$, the degree is increased by one. But, if exactly one of $x_i$ and $y_i$ equals $v_1$, assume without loss of generality $x_i = v_1$ and $y_i \neq v_1$, we get $-(G_{v_1 v_1} G_{k y_i} + G_{v_1 k} G_{v_1 y_i})$. The term coming from $ G_{v_1 k} G_{v_1 y_i}$ will have the degree increased by one, but the term from $G_{v_1 v_1} G_{k y_i}$ will still have degree $d$. We investigate these types of terms further. For each such term we can reorder the $G_{x_i y_i}$ and assume that the $y_i \neq v_1$ we had before is $y_2$. So the term can be written as
	\begin{equation}
		\frac{1}{N^{\# \mathcal{I} + 2}} \sum_{\mathcal{I} \cup \{j, k\}} \EX \left[G_{jj} G_{v_1 v_1} G_{k y_1} G_{k y_2} \prod_{i=3}^n G_{x_i y_i}\right].
		\label{v_1_moved_eq}
	\end{equation} If we compare the occurrences of $v_1$ in \eqref{v_1_moved_eq} to the term $T$ we started expanding, we see that two $v_1$ have been moved from $G_{v_1 y_1} G_{v_1 y_2}$ to $G_{v_1 v_1}$. In particular this means that $v_1$ is still unmatched. So we may assume that $G_{x_3 y_3} = G_{v_1 y_3}$ with $y_3 \neq v_1$. So we have the expression \begin{equation}
		\frac{1}{N^{\# \mathcal{I} + 2}} \sum_{\mathcal{I} \cup \{j, k\}} \EX \left[G_{v_1 y_3} G_{jj} G_{v_1 v_1} G_{k y_1} G_{k y_2} \prod_{i=4}^n G_{x_i y_i}\right].
		\label{next_iteration_start}
	\end{equation} This term is again unmatched and most importantly it has the property that $v_1$ appears two times less as an index in off-diagonal terms in the product of Green function entries, compared to $T$. This means that we can carry out the same expansions as we did to obtain \eqref{unmatched_expansion}, still expanding in $v_1$, and we will end up with a sum consisting of terms of degree $\geq d+1$, terms of degree $\geq d$ with at least an additional factor of $1/q$, and terms of degree $d$ with $v_1$ still unmatched, and appearing two times less in off-diagonal entries, also we again get an error of $O_{\prec}(1/N)$. Repeating this procedure at most $n$ times, the terms with degree $d$ and no $1/q$-factors will have only one $v_1$ appearing in an off-diagonal entry. By the above argument we then see that  by expanding this term once more, the degree of all resulting terms will be at least $d+1$ or degree $d$ and contain an additional factor of $1/q$.

	What we have shown is that we can write \begin{equation}
		T  = \sum_{\substack{T' \in {\Tau'}_{d+1}^o}} T' + \frac{1}{q} \sum_{\substack{T' \in {\Tau'}_{d}^o}} T' + O_{\prec} \left(\frac{1}{N}\right),
		\label{expansion_step}
	\end{equation} 
	where ${\Tau'}_{d+1}^o$ is a finite subset of
	$\Tau_{d+1}^{o}$, and
	${\Tau'}_{d}^o$ is a finite subset of
	$\Tau_{d}^{o}$. But now we can use \eqref{expansion_step} to recursively expand the terms in the right hand side of \eqref{expansion_step}. In each step gaining either one degree or a factor of $1/q$. Repeating the expansion in \eqref{expansion_step} $D$ times we get \begin{equation}
		T = \sum_{s=0}^D \frac{1}{q^{D-s}} \sum_{\substack{T' \in {\Tau'}_{d+s}^o}} T' + O_{\prec} \left(\frac{1}{N}\right),
		\label{iterater_expansion}
	\end{equation} where ${\Tau'}_{d+s}^o$ denotes a finite subset of $\Tau_{d+s}^o$. Using the a priori local law bound, \eqref{naive_local_law_bound} we get \begin{equation*}
		\abs{T} \prec \sum_{s=0}^D \frac{1}{q^{D-s}} \left(\Psi^{d+s} + \frac{1}{N} \right) + \frac{1}{N} \prec \sum_{s=0}^D N^{- \frac{1}{6}(D-s)} N^{- \epsilon(d+s)/2} +  \frac{1}{N} \prec N^{- \epsilon (D+d)/2} + \frac{1}{N}.
	\end{equation*} Now, choosing $D$ large enough for $- (\epsilon/2) (D +d) \leq -1$ to hold, we get that $\abs{T} \prec 1/N$. By noting that the local law bound \eqref{naive_local_law_bound} holds uniformly in $z(t) \in \edge(t)$ and $t \in [0, 10\log N]$, the proof of Lemma \ref{preliminary_unmatched_bound} is finished.
\end{proof}

 The next step is to improve the bound from Lemma \ref{preliminary_unmatched_bound} by taking the imaginary part of the unmatched term $T$ and hence prove Lemma \ref{prop_unmatched_term_bound}. To accomplish this, we will carefully examine where the error terms $O_{\prec}(1/N)$ in the proof of Lemma \ref{preliminary_unmatched_bound} originate, and show that when taking the imaginary part, they can be bounded effectively using Lemma \ref{lemma_general_imm_bound}.

In the following, we will first present the proof of Lemma \ref{lemma_general_imm_bound}, and then use it and Lemma \ref{preliminary_unmatched_bound} to complete the proof of Lemma \ref{prop_unmatched_term_bound}.

	\begin{proof}[Proof of Lemma \ref{lemma_general_imm_bound}]
	Consider a term	$T$ on the form \eqref{very_general_term}. It suffices to show $\abs{\im T} \prec \EX \left[\im m \right] + \frac{1}{N}$, for $d_N = d_q = 0$. First we deal with the cases where $d_\delta \geq 1$ or $p_h \geq 1$. If $d_{\delta} \geq 1$ we can use $a=b$ to rewrite $T$ into another term with $d_N$ decreased by one as in \eqref{example} and bound it by $O_{\prec}(1/N)$ using \eqref{naive_local_law_bound}. If $p_h \geq 1$ we can expand $T$ via $h_{ab}$ into terms with $d_N$ decreased by one as in \eqref{111}, and again use \eqref{naive_local_law_bound} to bound the resulting terms as $O_{\prec}(1/N)$. In the remainder of the proof we deal with the case $d_\delta = p_h = 0$.
	
	We now split the proof in three cases, depending on how many off-diagonal entries $T$ has. Assume without loss of generality that $c = 1$.
	
	\textbf{Case 1:} First we consider $T$ with at least two off-diagonal entries. Assume that they are $G_{x_1 y_1}$ and $G_{x_2 y_2}$. Then we can use the triangle inequality, the Cauchy-Schwarz inequality and the Ward identity in the
	following way: \begin{equation}
		\begin{split}
			\abs{\im T} & = \abs{\im \frac{1}{N^{\# \mathcal{I}}} \sum_{\mathcal{I}} \EX \left[\prod_{i=1}^{n} G_{x_i y_i}\right]} \\
			& \leq \frac{1}{N^{\# \mathcal{I}}} \sum_{\mathcal{I}}  \EX \left[\prod_{i=1}^{n} \abs{G_{x_i y_i}}\right] \\
			& \leq \frac{1}{N^{\# \mathcal{I}}} \EX \left[\left(\max_{1 \leq i,j \leq N} \abs{G_{ij}} \right)^{n-2} \sum_{\mathcal{I}}  \abs{G_{x_1 y_1}} \abs{G_{x_2 y_2}}\right] \\
			& \leq \frac{1}{N^{\# \mathcal{I}}} \EX \left[\left(\max_{1 \leq i,j \leq N} \abs{G_{ij}} \right)^{n-2} \sqrt{\sum_{\mathcal{I}}  \abs{G_{x_1 y_1}}^2} \sqrt{\sum_{\mathcal{I}}  \abs{G_{x_2 y_2}}^2} \right] \\
			& = \frac{1}{N^{\# \mathcal{I}}} \EX \left[\left(\max_{1 \leq i,j \leq N} \abs{G_{ij}} \right)^{n-2} N^{\# \mathcal{I}} \frac{\im m}{N \eta} \right] \\
			& \prec \frac{\EX \left[\im m\right]}{N \eta}.
		\end{split}
		\label{lemma_first_case}
	\end{equation} In the last step we used that $\max_{1 \leq i,j \leq N} \abs{G_{ij}} \prec 1$, which follows from the local law.
	
	\textbf{Case 2:} Next we treat the case of no off-diagonals. For this case we proceed in a similar manner, but instead of the Ward identity we use the following inequality: \begin{equation}
		\begin{split}
			\abs{\im \prod_{i=1}^{n} G_{x_i x_i}} & \leq 2^{n-1} \sum_{i=1}^{n} \abs{\im G_{x_i x_i}} \left(\max_{k} \abs{G_{kk}}\right)^{n-1} ,
		\end{split}
		\label{im_inequality}
	\end{equation} which follows directly from considering the following expression for $\prod_{i=1}^{n} G_{x_i x_i}$: \begin{equation}
		\left(\re G_{x_1 x_1} + \ii \im  G_{x_1 x_1}\right) \left(\re G_{x_2 x_2} + \ii \im  G_{x_2 x_2}\right) \cdots \left(\re G_{x_n x_n} + \ii \im  G_{x_n x_n}\right).
		\label{product_expression}
	\end{equation} We will get $2^n$ terms from expanding \eqref{product_expression}, and each of them containing an $\ii$ must contain an $\im G_{x_k x_k}$, and there are exactly $2^{n-1}$ terms containing $\im G_{x_k x_k}$. We can now carry out the calculations similar to \eqref{lemma_first_case}, \begin{equation}
		\begin{split}
			\abs{\im \EX [T]} & = \abs{\im \frac{1}{N^{\# \mathcal{I}}} \sum_{\mathcal{I}} \EX \left[\prod_{i=1}^{n} G_{x_i y_i}\right]} \\
			& \leq \frac{1}{N^{\# \mathcal{I}}} \sum_{\mathcal{I}}  \EX \left[\abs{ \im \prod_{i=1}^{n} G_{x_i y_i}}\right] \\
			& \leq \frac{2^{n-1}}{N^{\# \mathcal{I}}} \EX \left[\left(\max_{1 \leq k \leq N} \abs{G_{kk}} \right)^{n-1} \sum_{k=i}^{n} \sum_{\mathcal{I}} \im G_{x_i x_i}\right] \\
			& = \frac{2^{n-1}}{N^{\# \mathcal{I}}} \EX \left[\left(\max_{1 \leq k \leq N} \abs{G_{kk}} \right)^{n-1} n N^{\# \mathcal{I}} \im m\right] \\
			& \prec \EX \left[\im m\right].
		\end{split}
	\end{equation}
	
	\textbf{Case 3:} In this case we have exactly one off-diagonal entry. But this implies that $T$ must be unmatched, so by the preliminary bound, Lemma \ref{preliminary_unmatched_bound}, we have that $\abs{T} \prec 1/N$. And of course we have $\abs{\im T} \leq \abs{T}$.
	
	Note that all the bounds used above hold uniformly in $z(t) \in \edge(t)$ and $t \in [0, 10 \log N]$. This proves Lemma \ref{lemma_general_imm_bound}.
\end{proof}

	\begin{proof}[Proof of Lemma \ref{prop_unmatched_term_bound}]
	Consider again the expansion of $T$ leading to \eqref{unmatched_expansion}, and take the imaginary part on both sides. %
	 In \eqref{unmatched_expansion}, the $O_{\prec}(1/N)$ error term comes from $$\frac{1}{N^{\# \mathcal{I}}} \sum_{\mathcal{I}} c \EX \left[\delta_{v_1 y_1} G_{v_1 y_1} \prod_{i=2}^n G_{x_i y_i} \right],$$ or $$\frac{1}{N^{\# \mathcal{I}}} \sum_{\mathcal{I}} c \EX \left[h_{v_1 y_1} G_{v_1 y_1} \prod_{i=2}^n G_{x_i y_i} \right].$$ 
	However, as argued earlier, both of these terms can be rewritten as a sum of terms of the form \eqref{very_general_term} with $d_N$ reduced by 1. Hence we can use Lemma \ref{lemma_general_imm_bound} to improve the error term to $O_{\prec}(1/N \cdot (\im m + 1/N)) = O_{\prec}(\im m /N + 1/N^2)$. In \eqref{unmatched_expansion} we also have the error term $O_{\prec}(1/q^{l-1})$. But we can easily handle this by choosing $l$ large enough for $q^{l-1} \geq N^2$ to hold. All in all, we have showed that if we add an $\im$ to the expansion computations leading to \eqref{unmatched_expansion} the final error term in \eqref{unmatched_expansion} can be improved to $O_{\prec}(\frac{\im m}{N} + \frac{1}{N^2})$. 
	
	With the new error term and the additional $\im$, \eqref{iterater_expansion} becomes \begin{equation}
		\im T = \sum_{s=0}^D \frac{1}{q^{D-s}} \sum_{\substack{T_{d_s'}^o \in \Tau_{d_s'}^o \\ d_s' \geq d+s}} \im T_{d_s'}^o + O_{\prec} \left(\frac{\im m}{N} + \frac{1}{N^2}\right).
		\label{im_expansion_improved_error}
	\end{equation} 
	
	We also need to improve the naive power counting bound $O_{\prec}(\Psi^d + 1/N)$ we used before, which can be done in a similar manner in which we improved the error term in \eqref{unmatched_expansion}. Consider a general term as in \eqref{very_general_term} with $d_{\delta} = p_h = d_N = d_q = 0$ of degree $d$, and denote it by $T_d$. We then split the sum into two, the first for the case where all indices take on distinct values, and the second were at least two indices take the same value. 
	
	We introduce an indicator function $\mathbbm{1}_{\sim}(\mathcal{I})$ defined on $[1, N]^{\# \mathcal{I}}$, that is one exactly if some pair of indices have taken the same value and zero otherwise. Using the principle of inclusion and exclusion we can write it in terms of Kronecker deltas: \begin{equation}
		\mathbbm{1}_{\sim}(i_1, i_2, \ldots, i_m) = \sum_{p=1}^{\binom{m}{2}} (-1)^{p+1} \sum_{(e_1, e_2, \ldots, e_p) \subset \mathcal{E}} \delta_{e_1} \delta_{e_2} \cdots \delta_{e_p},
		\label{pie_indicator}
	\end{equation} where $\mathcal{E}$ is the edge set of the complete graph having the summation indices $i_j$ as vertices, so $\delta_e$ is one when the two indices corresponding to the vertices at the endpoints of the edge $e$ are equal, and zero otherwise. This expression is convenient for us, since a Kronecker delta in two different indices corresponds to replacing all occurrences of one of the indices with the other one, thus effectively reducing the size of $\mathcal{I}$ by one, and again obtaining an expression of the form \eqref{very_general_term} but with an additional factor of $1/N$, \ie $d_N$ is reduced by one. We are now ready to write down the improved power counting: \begin{equation}
		\begin{split}
			\abs{\im T} & \leq \abs{\im \frac{c}{N^{\# \mathcal{I}}} \sum_{\mathcal{I}} (1 - \mathbbm{1}_{\sim}(\mathcal{I})) \EX \left[\prod_{i=1}^{n} G_{x_i y_i}\right] } + \abs{\im \frac{c}{N^{\# \mathcal{I}}} \sum_{\mathcal{I}} \mathbbm{1}_{\sim}(\mathcal{I}) \EX \left[\prod_{i=1}^{n} G_{x_i y_i}\right] } \\
			& \leq \frac{\abs{c}}{N^{\# \mathcal{I}}} \sum_{\mathcal{I}} \EX \left[ \left(\max_{ i \neq j} \abs{G_{ij}} \right)^{d} \left(\max_{ i, j} \abs{G_{ij}} \right)^{n-d}\right] \\
			& \quad + \sum_{p=1}^{\binom{m}{2}} \sum_{(e_1, e_2, \ldots, e_p) \subset \mathcal{E}} \abs{ \im  \frac{1}{N^{\# \mathcal{I}}} \sum_{\mathcal{I}} c \delta_{e_1} \delta_{e_2} \cdots \delta_{e_p} \EX \left[\prod_{i=1}^{n} G_{x_i y_i}\right] } \\
			& \prec \Psi^d + \frac{\im m}{N} + \frac{1}{N^2},
		\end{split}
		\label{improved_power_counting}
	\end{equation} where in the last domination we used Lemma \ref{lemma_general_imm_bound} and that we gain at least one factor of $1/N$ for a product of Kronecker deltas but still have a term on the form \eqref{very_general_term}. Finally, using \eqref{improved_power_counting} on \eqref{im_expansion_improved_error} we obtain \eqref{unmatched_improved_bound}, if we choose $D$ large enough. This finishes the proof of Lemma \ref{prop_unmatched_term_bound}.
\end{proof}

\subsection{Proof of Lemmas \ref{lemma_F_imm_bound} -- \ref{prop_unmatched_term_bound_F}}
\label{appendix_F_proofs}
\begin{proof}[Proof of Lemma \ref{lemma_F_imm_bound}]
	We begin by noting the following a priori bound for a term $T^F$ of the form \eqref{F_very_general_term}:  
	\begin{equation}
		\abs{T^F} \prec \frac{N^{d_N}}{q^{d_q}}.
		\label{naive_term_estimate}
	\end{equation} It follows from the fact that $\max_{i,j}\abs{G_{ij}(t, \widehat{L}_t + \gamma_{k} + \ii \eta)} \prec 1$, for $k = 1, 2$, when $\eta \geq N^{-1 + \epsilon}$ (which follows from the local law estimates (\ref{time_dependent_local_law})), and that $\abs{h_{ab}} \prec 1$. 
	
	As in the proof of Lemma \ref{lemma_general_imm_bound}, if $p_h \geq 1$ or $d_{\delta} \geq 1$, we may expand/rewrite the term into terms with $d_N$ decreased by one and use \eqref{naive_term_estimate} to obtain \eqref{imm_lemma_F_eq}. The case $d_{\delta} = p_h = 0$ follows by multiplying out the products in \eqref{F_very_general_term} and using the same cases as in the proof of Lemma \ref{lemma_general_imm_bound}.\begin{equation}
		\begin{split}
			& \quad \, \abs{\alpha(t) \frac{N^{d_N}}{q^{d_q} N^{\# \mathcal{I}}} c \sum_{\mathcal{I}} \EX \left[F^{(i_0)}(X(t)) \prod_{i=1}^{i_0} \Dim \left(\prod_{l=1}^{n_i} G_{x_l^{(i)} y_l^{(i)}}\right)\right]} \\
			& \leq \frac{N^{d_N}}{q^{d_q} N^{\# \mathcal{I}}} c \sum_{\mathcal{I}} \EX \left[\abs{F^{(i_0)}(X(t))} \sum_{\sigma \in \{1, 2\}^{i_0}}  \abs{\prod_{i=1}^{i_0} \im \left(\prod_{l=1}^{n_i} G_{x_l^{(i)} y_l^{(i)}}\left(t, \widehat{L}_t +  \gamma_{\sigma(i)} + \ii \eta\right)\right)}\right] \\
			& \leq \frac{N^{d_N}}{q^{d_q} N^{\# \mathcal{I}}} C \sum_{\mathcal{I}} \sum_{\sigma \in \{1, 2\}^{i_0}}  \EX \left[\abs{\prod_{i=1}^{i_0} \im \left(\prod_{l=1}^{n_i} G_{x_l^{(i)} y_l^{(i)}}\left(t, \widehat{L}_t +  \gamma_{\sigma(i)} + \ii \eta\right)\right)}\right].
		\end{split}	
		\label{F_general_term_expand_bound}
	\end{equation} The notation $\sum_{\sigma \in \{1, 2\}^{i_0}}$ means that we sum over all functions $\sigma : \{1, 2, \ldots, i_0\} \to \{1, 2\}$. In the last inequality of \eqref{F_general_term_expand_bound} we used that $F$ has uniformly bounded derivatives. We now see that for the last expression we may use the same methods as used in the proof of Lemma \ref{lemma_general_imm_bound}. To handle terms of degree $\geq 2$ we can use the AM-GM inequality and Ward identity, even on products of Green function entries evaluated in different $z$ since \begin{equation*}
		\begin{multlined}
			\abs{G_{x_1^{(1)}y_1^{(1)}}(t, \widehat{L}_t + \gamma_{\sigma(1)} + \ii \eta) G_{x_1^{(2)}y_1^{(2)}} (t, \widehat{L}_t + \gamma_{\sigma(2)} + \ii \eta)} \\ \leq \frac{1}{2} \left(\abs{G_{x_1^{(1)}y_1^{(1)}}(t, \widehat{L}_t + \gamma_{\sigma(1)} + \ii \eta)}^2 + \abs{G_{x_1^{(2)}y_1^{(2)}} (t, \widehat{L}_t + \gamma_{\sigma(2)} + \ii \eta)}^2\right).
		\end{multlined} 
	\end{equation*} 
	
	The proof for terms of degree 0 follows directly from \eqref{F_general_term_expand_bound}, that $\abs{m(t, \widehat{L}_t + \gamma_{j} + \ii \eta)} \prec 1$ for $j= 1, 2$, and from the inequality in \eqref{im_inequality}. 
	
	To handle terms of degree 1 we also proceed in a similar way as before. If a term has degree 1 it must be unmatched, and we can in analogy to Lemma \ref{preliminary_unmatched_bound} show that unmatched terms $T^F$ of the form \eqref{F_very_general_term} satisfy $\abs{T^F} \prec N^{d^N - 1} / q^{d_q}$. The strategy to show this is exactly the same as in the proof of Lemma \ref{preliminary_unmatched_bound}, we motivate further below in the proof of Lemma \ref{prop_unmatched_term_bound_F}. \end{proof}

\begin{proof}[Proof of Lemma \ref{prop_unmatched_term_bound_F}]
	We can follow the same strategy used to show the previous bound on unmatched terms, Lemma \ref{prop_unmatched_term_bound}. We omit the details, but briefly motivate why the same strategy works. We may use the identity in \eqref{green_function_entry_expansion} to expand an off-diagonal entry in the $\Dim$-factors, and since $\partial/\partial h_{ab}$ and $\Dim$ commute (as well as $\partial/\partial w_{ab}$ and $\Dim$), the same arguments regarding the degrees of terms in the expansion and the index $v_1$ work for terms of the form \eqref{F_very_general_term}. Using this we can derive the bound $\abs{T^F} \prec N^{d_N - 1}/q^{d_q}$, as we mentioned in the proof of Lemma \ref{lemma_F_imm_bound}. 
	
	Finally, we can repeat the last argument in the proof of Lemma \ref{prop_unmatched_term_bound} to improve the bound from $\abs{T^F} \prec N^{d_N - 1}/q^{d_q}$ to \eqref{prop_unmatched_term_bound_F_eq}. In particular we see from \eqref{F_general_term_expand_bound} that for terms $T^F$ with degree $d \geq 2$ of the form \eqref{F_very_general_term} with $d_{\delta} = p_h = 0$ we may use the computations in \eqref{ward_identity_bound} to bound $T^F$ as \begin{equation}
		\begin{multlined}
			\abs{T^F} \prec \frac{N^{d_N}}{q^{d_q}} \Bigg(\left(\frac{N^{-(d-2)\epsilon/2}}{N\eta}\right) \EX[\im m(t, \widehat{L}_t + \gamma_1 + \ii \eta) + \im m(t, \widehat{L}_t + \gamma_2 + \ii \eta)] \\ + \frac{\EX[\im m(t, \widehat{L}_t + \gamma_1 + \ii \eta) + \im m(t, \widehat{L}_t + \gamma_2 + \ii \eta)]}{N} + \frac{1}{N^2}\Bigg).
		\end{multlined}
		\label{F_ward_bound}
	\end{equation} From \eqref{F_ward_bound} we obtain the bound in Lemma \ref{prop_unmatched_term_bound_F} for terms of sufficiently large degree. 
\end{proof}
 
\section{Some examples for deriving \eqref{list_of_identities} and solving (\ref{goal_linear_eq_sys})}
\label{details_about_imm_system}

 In order for help the reader learn the method to derive (\ref{list_of_identities}) and solve (\ref{goal_linear_eq_sys}), we write out some terms in $\Tau^{\mathrm{basis}}$ as examples, as well as several identities listed in (\ref{list_of_identities}).  

The first six basis terms in $\Tau^{\mathrm{basis}}$ generated by the method described in Section \ref{basis_term_implementation_section} are the following: 
\begin{equation*}
	\begin{split}
		& \e^{-t}(1 - \e^{-t}) \frac{N}{q^2N^5}  \sum_{a,b,c,d,e} \EX \left[G_{aa}G_{ac}G_{ae}G_{bb}^{2}G_{cd}G_{de}\right], \\
		& \e^{-t}(1 - \e^{-t}) \frac{N}{q^2N^5}  \sum_{a,b,c,d,e} \EX \left[G_{aa}^{2}G_{bb}^{2}G_{cc}G_{de}^{2}\right], \\
		& \e^{-t}(1 - \e^{-t}) \frac{N}{q^2N^5}  \sum_{a,b,c,d,e} \EX \left[G_{aa}^{2}G_{bb}^{2}G_{cd}G_{ce}G_{de}\right], \\
		& \e^{-t}(1 - \e^{-t}) \frac{N}{q^2N^4}  \sum_{a,b,c,d} \EX \left[G_{aa}^{2}G_{bb}^{2}G_{cd}^{2}\right], \\
		& \e^{-t}(1 - \e^{-t}) \frac{N}{q^2N^3}  \sum_{a,b,c} \EX \left[G_{aa}^{2}G_{bb}G_{bc}^{2}\right], \\
		& \e^{-t}(1 - \e^{-t}) \frac{N}{q^2N^3}  \sum_{a,b,c} \EX \left[G_{aa}G_{ab}G_{ac}G_{bb}G_{bc}\right].
	\end{split}
\end{equation*}

The third column of the coefficient matrix $\widetilde{C}=\left(\widetilde{c}_{ml}\right)_{1 \leq \beta \leq M,1 \leq l \leq L}$ has six non-zero entries at rows $6, 10, 11, 13, 15, 44$. And the entries at those positions are: $-1, -5/2, -1, -1/2, -1/2, -1/2$. Written out in actual terms, the identity stored in the third column of $\widetilde{C}$ is given by
\begin{equation}
	\begin{split}
		0  = & -1 \e^{-t}(1 - \e^{-t}) \frac{N}{q^2N^3}  \sum_{a,b,c} \EX \left[G_{aa}G_{ab}G_{ac}G_{bb}G_{bc}\right] \\
		&-5/2 \e^{-t}(1 - \e^{-t}) \frac{N}{q^2N^4}  \sum_{a,b,c,d} \EX \left[G_{aa}G_{ab}G_{ac}G_{bb}G_{bd}G_{cd}\right] \\
		& -1 \e^{-t}(1 - \e^{-t}) \frac{N}{q^2N^4}  \sum_{a,b,c,d} \EX \left[G_{ab}^{2}G_{ac}G_{ad}G_{bb}G_{cd}\right] \\
		& -1/2 \e^{-t}(1 - \e^{-t}) \frac{N}{q^2N^4}  \sum_{a,b,c,d} \EX \left[G_{aa}G_{ab}^{2}G_{bb}G_{cd}^{2}\right] \\
		& -1/2 \e^{-t}(1 - \e^{-t}) \frac{N}{q^2N^4}  \sum_{a,b,c,d} \EX \left[G_{aa}G_{ab}G_{ad}G_{bb}G_{bd}G_{cc}\right] \\
		& -1/2 \e^{-t}(1 - \e^{-t}) \frac{N}{q^2N^4}  \sum_{a,b,c,d} \EX \left[G_{aa}^{2}G_{bb}G_{bc}G_{bd}G_{cd}\right] \\
		& + O(\Tau_G) + O_{\prec}(N^{-D}).
	\end{split}
	\label{first_identity}
\end{equation}  This identity is generated using expansion rule 1 with $G_{ac}$ on the first term on the right in \eqref{first_identity}. In particular all of the terms appearing in \eqref{first_identity} are of the same type; they are type-AB terms. Note that the first term in \eqref{first_identity} is the same as the sixth term in $\Tau^{\mathrm{basis}}$, since the first non-zero entry of the third column is in row six. 

Furthermore, the third entry of the solution vector $(x_{m})_{m=1}^{M}$  to the matrix equation (\ref{goal_linear_eq_sys}) is $-1824$. This means that along with 61 other identities, the following identity leads to the cancellations on the right side of \eqref{imm_diff_expansion}:
\begin{equation*}
	\begin{split}
		-1824 \cdot \Bigl( & -1 \e^{-t}(1 - \e^{-t}) \frac{N}{q^2N^3}  \sum_{a,b,c} \EX \left[G_{aa}G_{ab}G_{ac}G_{bb}G_{bc}\right] \\
		&-5/2 \e^{-t}(1 - \e^{-t}) \frac{N}{q^2N^4}  \sum_{a,b,c,d} \EX \left[G_{aa}G_{ab}G_{ac}G_{bb}G_{bd}G_{cd}\right] \\
		& -1 \e^{-t}(1 - \e^{-t}) \frac{N}{q^2N^4}  \sum_{a,b,c,d} \EX \left[G_{ab}^{2}G_{ac}G_{ad}G_{bb}G_{cd}\right] \\
		& -1/2 \e^{-t}(1 - \e^{-t}) \frac{N}{q^2N^4}  \sum_{a,b,c,d} \EX \left[G_{aa}G_{ab}^{2}G_{bb}G_{cd}^{2}\right] \\
		& -1/2 \e^{-t}(1 - \e^{-t}) \frac{N}{q^2N^4}  \sum_{a,b,c,d} \EX \left[G_{aa}G_{ab}G_{ad}G_{bb}G_{bd}G_{cc}\right] \\
		& -1/2 \e^{-t}(1 - \e^{-t}) \frac{N}{q^2N^4}  \sum_{a,b,c,d} \EX \left[G_{aa}^{2}G_{bb}G_{bc}G_{bd}G_{cd}\right] \Bigr) \\
		& =  O(\Tau_G) + O_{\prec}(N^{-D}).
	\end{split}
\end{equation*}

Another example of identity, which occurs in the 54th column of the coefficient matrix $\widetilde{C}$ is 
\begin{equation}
	\begin{split}
		0 = & -2 \e^{-t}(1 - \e^{-t}) \frac{N}{q^2N^3}  \sum_{a,b,c} \EX \left[G_{aa}G_{ab}G_{ac}G_{bc}\right] \\
		& -\e^{-t}(1 - \e^{-t}) \frac{N}{q^2N^3}  \sum_{a,b,c} \EX \left[G_{aa}G_{ab}^{2}G_{cc}\right] \\
		& -2 \e^{-t}(1 - \e^{-t}) \frac{N}{q^2N^2}  \sum_{a,b} \EX \left[G_{aa}G_{ab}^{2}\right] \\
		& -3 \e^{-t}(1 - \e^{-t}) \frac{N}{q^2N^3}  \sum_{a,b,c} \EX \left[G_{ab}^{2}G_{ac}^{2}\right] \\
		& -\e^{-t}(1 - \e^{-t}) \frac{N}{q^2N^2}  \sum_{a,b} \EX \left[G_{ab}^{2}\right] \\
		& + O(\Tau_G) + O_{\prec}(N^{-D}).
	\end{split}
\end{equation} Here $\e^{-t}(1 - \e^{-t}) \frac{N}{q^2N^2}  \sum_{a,b} \EX \left[G_{ab}^{2}\right]$ has been used as $T_{\text{start}}$ and it has been expanded using Rule 3 with the existing index $a$. In particular, $\e^{-t}(1 - \e^{-t}) \frac{N}{q^2N^2}  \sum_{a,b} \EX \left[G_{ab}^{2}\right]$ is a type-0 term, and all other terms occurring in the identity are type-A terms. The 54th entry of the solution vector to \eqref{goal_linear_eq_sys} is $-12$.

\bigskip

{\bf Acknowledgment:} This manuscript is based upon work supported by the Swedish Research Council under grant No.\ 2021-06594 while the authors were in residence at Institut Mittag-Leffler in Djursholm, Sweden during the fall semester 2024. T.B.\ and K.S.\ are supported by Grant No.\ VR-2021-04703 from the Swedish Research Council. Y.X.\ is partially supported by Grant No.\ 2024YFA1013503 from the National Key R\&D Program of China.

				\printbibliography

\end{document}